\documentclass[a4paper]{article}
\usepackage[utf8]{inputenc}
\usepackage[top=1.45in, bottom=1.65in, left=1.55in, right=1.55in]{geometry}

\usepackage{amsmath, amsthm, amssymb, bbm, mathrsfs, url}

\newtheorem{thm}{Theorem}[section]
\newtheorem{lem}[thm]{Lemma}
\newtheorem{ex}[thm]{Example}

\numberwithin{equation}{section}

\theoremstyle{definition}
\newtheorem{definition}[thm]{Definition}

\newcommand{\thistheoremname}{}
\newtheorem*{genericthm}{\thistheoremname}
\newenvironment{namedthm}[1]
    {\renewcommand{\thistheoremname}{#1}%
        \begin{genericthm}}
    {\end{genericthm}}

\theoremstyle{remark}
\newtheorem*{remark}{Remark}

\newcommand{\db}{\textrm{\normalfont{db}}}
\newcommand{\var}{\textrm{\normalfont{Var}}}
\newcommand{\oct}{\textsc{Oct}}
\newcommand{\F}{\mathbb{F}}
\newcommand{\Z}{\mathbb{Z}}
\newcommand{\R}{\mathbb{R}}

\newcommand{\str}{\textrm{\normalfont{str}}}
\newcommand{\qsr}{\textrm{\normalfont{qsr}}}
\newcommand{\err}{\textrm{\normalfont{err}}}

\title{Quasirandomness in additive groups and \\ hypergraphs}
\author{Davi Castro-Silva}
\date{\today}

\newcommand{\Addresses}{{
    \bigskip
    \small
    
    D. de Castro Silva, \textsc{Department Mathematik/Informatik, Abteilung Mathematik, Universit\"at zu K\"oln, Weyertal 86--90, 50931 K\"oln, Germany.}\par\nopagebreak
    \textit{E-mail address:} \texttt{davisilva15@gmail.com}
}}

\begin{document}

\maketitle

\begin{abstract}
Quasirandomness is a general mathematical concept meant to encapsulate several characteristics usually satisfied by random combinatorial objects, and which we regard as describing when a given object `looks random'.
In this survey we explore this general concept as it applies to graphs, hypergraphs and additive groups, making clear their many connections to each other and showing how they can be used in order to better study these objects.
\end{abstract}

\section{Introduction} \label{intro}

While mathematicians have undoubtedly thought of objects that `look random' ever since the notion of randomness was first developed, the systematic study of the overarching concept of pseudorandomness is much more recent.
In the field of combinatorics, its foundation lies in the seminal works of Thomason \cite{PseudographsThomason, PseudoGhost} and Chung, Graham and Wilson \cite{QuasirandomGraphs} during the second half of the 1980s;
in these papers they introduced and studied deterministic properties of graphs which capture many characteristics associated with random graphs.

Thomason introduced the notion of `jumbledness', which is a measure of how well-distributed the edges of a given graph are, and showed that any graph which is sufficiently jumbled will behave in many ways like a random graph of the same edge density.
Chung, Graham and Wilson then showed that several properties which are characteristic for random graphs are in fact essentially \emph{equivalent} to each other;
such properties are called quasirandom, and inspired a great deal of research in this subject.
By now, notions of quasirandomness have been introduced and studied for several combinatorial objects, such as hypergraphs \cite{QuasirandomHypergraphs, QuasiClasses, HypergraphQuasirandomnessRegularity}, subsets of abelian groups \cite{QuasirandomZn, LinearQuasirandomness}, tournaments \cite{QuasiTournaments} and general oriented graphs \cite{QuasiOriented}, permutations \cite{QuasiPermutations, 4pointPermutations}, groups \cite{QuasiGroups} and words \cite{QuasiWords}.

The purpose of this paper is to explore the general concept of quasirandomness as it applies to graphs, hypergraphs and additive groups.
We will survey the main pertinent results which can be found in the literature, making also an effort to highlight their many connections to each other and to give an idea of how they can be used in order to better analyze these structures.

\bigskip
\emph{This survey is intended for all those who are interested in the notion of quasirandomness, from specialists in the field to students.
We assume only familiarity with basic undergraduate topics such as linear algebra and elementary calculus.}

\subsection{A few words on quasirandomness}

As suggested by its name, we intuitively think of quasirandom mathematical objects as those which behave like a random object of the same type.
Expanding on this point a little, the main idea behind this concept is the following:
we first identify some important characteristics a random mathematical structure will satisfy with high probability, and then define a quasirandom structure of this type as one which shares these same properties.

These characteristics are usually related to the \emph{lack of correlation} between distinct sub-parts of the object considered, which gives it strong uniformity properties.
Indeed, one can usually break the sampling process of a large random object into several smaller random choices, all independent from each other;
it is the independence of these sub-choices that gives rise to most characteristics usually associated with random objects.

But why would such a study be interesting, or useful?
Of course, as most things in pure mathematics, one of the main reasons for its study is its inherent mathematical beauty
(which this particular concept has no lack of).
Moreover, random structures and probabilistic arguments have become a staple of extremal combinatorics, with probabilistic constructions providing the best known extremal objects for many problems and several existence proofs relying on showing that a well-chosen random event has a positive (and usually very high) probability of occurring;
we refer the reader to Alon and Spencer's book \cite{AlonSpencer} for a veritable wealth of examples.
As remarked by Thomason in his founding paper \cite{PseudographsThomason}, in such cases
it would be useful to have a criterion by which to decide whether a specific object behaves like a random one of the same type, that is, has the property of those random objects that interests us.

We also feel it is important to stress that, while some
quasirandom properties
may at first appear rather strong or rigid, by definition it is satisfied by \emph{almost all} objects of the considered type.
The study of quasirandomness then permits us to analyze almost all elements from a given class of discrete mathematical objects, while using methods and intuition from probability theory to help us along.

Finally, we remark on a fundamental phenomenon in combinatorics which can be expressed as the \emph{dichotomy between structure and randomness};
we refer the reader to Tao's survey articles \cite{DichotomyAndThePrimes, StructureRandomnessCombinatorics} for an excellent discussion and several examples.
This phenomenon is made explicit (and quite useful) in various decomposition theorems usually known as `regularity lemmas', which allow us to decompose \emph{any} object in some given class into a highly structured component and a quasirandom component
(with possibly an additional small error term).
The first component should be easier to analyze directly due to its specific structure, while the second component is analyzed using the methods outlined in this survey;
in this way, the study of quasirandomness provides tools to investigate arbitrary combinatorial objects and are essential ingredients in the proof of several very general results.

\subsection{Overview of the survey}

In Section \ref{graphs} we will present the simpler and more well-know notion of quasirandomness in the setting of graphs, where already many of the methods and ideas will be present in a more easily pictured and less notationally cumbersome way.
We will motivate this concept and then show that several natural and sometimes distant-looking properties one usually associates with random graphs are all roughly equivalent to each other.

We then introduce similar notions of quasirandomness for subsets of additive groups (where the concept is usually called \emph{uniformity} instead of quasirandomness) in Section \ref{GroupsSection}, and for hypergraphs in Section \ref{hypergraphs}.
It turns out that in these two settings there is a hierarchy of several natural notions of quasirandomness, making the theory much richer but also more technical for these objects than it is for graphs.

There are several interesting parallels between the concepts of quasirandomness in additive groups and in hypergraphs, and we shall devote Section \ref{Cayley} to studying them.
We will also show, as an illustrative application of these concepts and their connections to each other, how to use hypergraph theoretic methods to estimate the number of linear configurations inside uniform additive sets.
The ability to estimate this count with high accuracy is very important for many problems in additive combinatorics, and serves to highlight not only the parallels between these two theories but also some of their differences.


In Section \ref{regularity} we will give a brief exposition on regularity lemmas, which allow one to decompose any object of a given class (such as graphs, hypergraphs or subsets of additive groups) into a highly structured component and a quasirandom component.
These results are embodiments of the
dichotomy between structure and randomness in combinatorics, and permit us to use the methods from earlier sections of the paper in order to analyze \emph{arbitrary} objects of the considered class, rather than only those which are quasirandom.

Finally, Appendix \ref{App} gives a succinct account of the (fairly basic) probabilistic notions and results which are useful to us and will be used throughout the paper.

\subsection{General notation and definitions}

We write $O(1)$ to denote any quantity bounded above by an absolute constant, and use $O(X)$ to mean $O(1)X$.
If the implied absolute constant depends also on an additional parameter $K$, we highlight this by writing $O_K(X)$.
Given a positive function $f: \mathbb{N} \rightarrow (0, \infty)$, we use the asymptotic notation $g(n) = o(f(n))$ to mean that $\lim_{n \rightarrow \infty} g(n)/f(n) = 0$;
in particular, $o(1)$ denotes some quantity that goes to zero as the asymptotic variable gets large.
For real numbers $a, b, c$ with $c > 0$, we write $a = b \pm c$ to mean $b - c \leq a \leq b + c$.

The same denomination will be used both for a set and for its indicator function.
If $X$ is a finite set, we use the averaging notation
$\mathbb{E}_{x \in X} := |X|^{-1} \sum_{x \in X}$
so that $\mathbb{E}_{x \in X}[f(x)]$ denotes the average of the function $f$ inside $X$;
we also write $\mathbb{E}_{x \in X, y \in Y} := \mathbb{E}_{x \in X} \mathbb{E}_{y \in Y}$.
The discrete interval $\{1, 2, \dots, n\}$ is denoted more succinctly as $[n]$, and we occasionally use `iff' as a shorthand for `if and only if'.

There is a specific notion of equivalence between properties of combinatorial objects which is of crucial importance when studying quasirandomness.
Suppose we have two properties $P_1 = P_1(c_1)$ and $P_2 = P_2(c_2)$ which a given object $H$ might satisfy, where each property $P_i$ involves a positive constant $0 < c_i \leq 1$.
We say that $P_1$ and $P_2$ are \emph{asymptotically equivalent} if for all $\varepsilon > 0$ there are $\delta > 0$ and $n_0 \geq 1$ so that the following holds:
\begin{itemize}
    \item If $H$ has size at least $n_0$ and satisfies $P_1$ with constant $c_1 \leq \delta$, then it must also satisfy $P_2$ with constant $c_2 = \varepsilon$;
    \item If $H$ has size at least $n_0$ and satisfies $P_2$ with constant $c_2 \leq \delta$, then it must also satisfy $P_1$ with constant $c_1 = \varepsilon$.
\end{itemize}
Being interested also in the quantitative aspects of these equivalences, we will say that a set of properties $P_1, \dots, P_k$ are \emph{polynomially equivalent} if they are (pairwise) asymptotically equivalent with polynomial bounds on all quantities involved
(so there is a constant $C > 0$ such that $\delta \geq \varepsilon^C/C$ and $n_0 \leq C/\varepsilon^C$ in the definition above).

\section{Quasirandom graphs} \label{graphs}

It was in the setting of graphs that the concepts of pseudorandomness and quasirandomness first originated in combinatorics, mainly due to the work of Thomason \cite{PseudographsThomason, PseudoGhost} and of Chung, Graham and Wilson \cite{QuasirandomGraphs} during the
1980s.\footnote{Before then there had already been some examples and applications of pseudorandom graphs, but without it being developed into a systematic study as done by those authors.
We refer the reader to Krivelevich and Sudakov's excellent survey \cite{PseudographsSurvey} for a much fuller discussion on pseudorandom graphs and their history.}
The informal idea of these notions is that a graph is pseudo- or quasirandom if its edge distribution resembles the one of a truly random graph with the same
edge density.\footnote{The way in which pseudorandom graphs resemble their random counterparts may be different for each specific application, while quasirandom graphs are rigorously defined as those  satisfying properties in a large equivalence class that happen to be shared by random graphs;
see Fan Chung's website \cite{ChungSite} for a discussion and for several references related to quasirandom objects.}

There is a very natural and well-studied model of random graphs for any given edge density $0 < p < 1$, which is called the \emph{Erd\H{o}s-R\'enyi random graph $G(n, p)$}:
this is a random graph on $n$ vertices (say $[n] = \{1, 2, \dots, n\}$) where every pair of vertices has probability $p$ of being an edge, all choices independent.
By a common abuse of notation, we will denote by $G(n, p)$ both the `random graph' just defined (which is in fact a probability distribution over graphs) and a graph sampled from this probability distribution.

An important property of this model of random graphs is that their edges are very uniformly distributed, and this is the property to be mimicked by quasirandom graphs.
To make this idea precise, let us define \emph{cuts} in a graph:

\begin{definition}
Given a graph $G$ and two sets $A, B \subseteq V(G)$, we define the \emph{cut between $A$ and $B$ in $G$} as
$$E_G(A, B) := \big\{(x, y) \in A \times B: \, xy \in E(G)\big\},$$
where we write $V(G)$ for the vertex set of $G$ and $E(G)$ for its edge set.
Note that we are considering \emph{ordered} pairs of vertices, so an edge whose vertices are both in $A \cap B$ will be represented twice in the cut.
\end{definition}

Let us first show that the edges of the Erd\H{o}s-R\'enyi random graph $G = G(n, p)$ are (with high probability) uniformly distributed along all cuts.
(See Appendix \ref{App} for the relevant notions and results in finite probability theory.)
For each $1 \leq i < j \leq n$, let $X_{ij} = \mathbf{1}_{\{ij \in G\}}$ be the random variable representing whether or not $ij$ is an edge of $G$;
these variables are jointly independent and satisfy $\mathbb{P}(X_{ij}) = p$.
Given sets $A, B \subseteq [n]$, note that
\begin{equation} \label{ERedges}
    |E_G(A, B)| = \sum_{i \in A} \sum_{j \in B:\, j > i} X_{ij} + \sum_{i \in A} \sum_{j \in B:\, j < i} X_{ji}.
\end{equation}
The expected size of the cut $E_G(A, B)$ is then $p|A| |B| - p|A \cap B| = p|A| |B| \pm n$.

Fix some number $0 < \varepsilon < 1$, and suppose $n \geq 2/\varepsilon$.
Denote the first double sum in equation (\ref{ERedges}) by $Y$ and the second by $Z$;
the indicator random variables which form \emph{each one} of these double sums are jointly independent, so both $Y$ and $Z$ have variance at most $p(1-p) \binom{n}{2} \leq n^2/8$.
Using Chernoff's inequality (Lemma \ref{Chernoff}) for each of the random variables $Y$ and $Z$ separately, we obtain
\begin{align*}
    &\mathbb{P}\big(\, \big| |E_G(A, B)| - p|A| |B| \big| \geq \varepsilon n^2 \,\big) \\
    &\hspace{2cm}\leq \mathbb{P}\big( |Y - \mathbb{E}[Y]| \geq \varepsilon n^2/4 \big)
    + \mathbb{P}\big( |Z - \mathbb{E}[Z]| \geq \varepsilon n^2/4 \big) \\
    &\hspace{2cm}\leq 4 e^{-\varepsilon^2 n^2/8}.
\end{align*}
Since this holds for all pairs $(A, B)$ of subsets of $[n]$ and there are $2^{2n}$ such pairs, it follows from union bound that
$$\mathbb{P} \big( \exists A, B \subseteq [n]:\, \big| |E_G(A, B)| - p|A| |B| \big| \geq \varepsilon n^2 \big) \leq 2^{2n} \cdot 4 e^{-\varepsilon^2 n^2/8} \xrightarrow{n \rightarrow \infty} 0.$$
The actual number of edges in \emph{every} cut $E_G(A, B)$ will thus w.h.p. be highly concentrated around their (approximate) mean $p|A| |B|$, with error $o(n^2)$.\footnote{This argument in fact shows that the error bound can be lowered to $O(n^{3/2})$, but for the purpose of defining quasirandomness the coarser error estimate $o(n^2)$ is more suitable.}

If a graph $G$ satisfies this uniform distribution of edges over all cuts, we shall then say that it is \emph{quasirandom} (a more quantitative definition will be given later, after we define the cut norm of graphs and functions).
We can now state our main result on quasirandom graphs, first obtained by Chung, Graham and Wilson \cite{QuasirandomGraphs}.

For a graph $G$, we denote its number of edges by $|G|$ and its number of vertices by $v(G)$;
its \emph{edge density} is defined as $2|G|/v(G)^2$.
The \emph{adjacency matrix} of $G$ is the symmetric matrix $A = (A_{ij})_{i, j \in V(G)}$ indexed by pairs of vertices, and whose entry $a_{ij}$ is $1$ if $ij \in E(G)$ and is $0$ otherwise.

\begin{thm}[Equivalence theorem for quasirandom graphs] \label{quasigraphs}
Let $G$ be a graph with $n$ vertices and edge density $\delta$.
Then the following statements are polynomially equivalent:
\begin{itemize}
    \item[$(i)$] For any two subsets $A, B \subseteq V(G)$, the size of the cut $E_G(A, B)$ differs from $\delta |A||B|$ by at most $c_1 n^2$.
    \item[$(ii)$] The number of labelled copies of any given graph $F$ in $G$ differs from $\delta^{|F|} n^{v(F)}$ by at most $c_2 |F| n^{v(F)}$.
    \item[$(iii)$] The number of labelled 4-cycles in $G$ is at most $(\delta^4 + c_3)n^4$.
    \item[$(iv)$] The largest eigenvalue of the adjacency matrix of $G$ is $(\delta \pm c_4)n$, and all other eigenvalues are at most $c_4 n$ in absolute value.
\end{itemize}
\end{thm}

\begin{remark}
The main theorem in the paper of Chung, Graham and Wilson also considers several other properties that are polynomially equivalent to $(i)-(iv)$, but we shall restrict our attention to just these four stated.
\end{remark}

All of these properties $(i)-(iv)$ were already known to be satisfied by the Erd\H{o}s-R\'enyi random graph $G(n, p)$ with high probability, suggesting they indeed provide some measure of pseudorandomness:
\begin{itemize}
    \item As discussed, the size of the cut $E_G(A, B)$ on a random graph of edge probability $p$ is highly concentrated around its mean $p|A||B|$.
    \item There are $n(n-1) \dots (n-v(F)+1)$ ways of choosing the vertices for a copy of $F$ in $G(n, p)$, and each of the $|F|$ edges has probability $p$ of being in $G(n, p)$.
    The expected number of (labeled) copies of $F$ is thus $p^{|F|} n^{v(F)} + O_F(n^{v(F)-1})$, and its variance is easily seen to be $O_F(n^{2v(F)-2})$;
    by Chebyshev's inequality, the number of copies of $F$ in $G$ then differs from $p^{|F|} n^{v(F)}$ by $O_F(n^{v(F)-1})$ with high probability.
    \item Item $(iv)$ is a well-known property of random graphs, first proven (in a stronger form) by Juh\'asz \cite{RandomEigenvalues} in 1978.
\end{itemize}
As remarked in \cite{QuasirandomGraphs}, the most surprising fact in this result is how strong the seemingly weak property $(iii)$ actually is:
just knowing a graph $G$ has a `small' number of 4-cycles already suffices to estimate the number of every other subgraph in $G$.

Before proving Theorem \ref{quasigraphs}, it will be useful to define a couple of notions in a more analytical/probabilistic framework that will simplify its proof.

\subsection{Cut norm and homomorphism densities}

The central notion of quasirandomness for graphs that we use here is related to the edges having low discrepancy over cuts $(A, B) \subseteq V \times V$.
This can be conveniently measured by the \emph{cut norm} $\|\cdot\|_{\square}$, originally introduced by Frieze and Kannan \cite{FriezeKannan}:

\begin{definition}
Given a function $f: V \times V \rightarrow \R$, we define its \emph{cut norm} by
\begin{align*}
    \| f \|_{\square} &= \frac{1}{|V|^2} \max_{A, B \subseteq V} \Bigg| \sum_{x \in A} \sum_{y \in B} f(x, y) \Bigg| \\
    &= \max_{A, B \subseteq V} \big| \mathbb{E}_{x, y \in V} \big[ f(x, y) A(x) B(y) \big] \big|.
\end{align*}
\end{definition}

For a graph $G$ of edge density $\delta$, the value of $\|G - \delta\|_{\square}$
(where $G$ denotes the indicator function of the edge set $E(G)$ and $\delta = \delta \mathbf{1}$ denotes a constant function)
then quantifies how much the size of a cut $E_G(A, B)$ can deviate from its `expected value' $\delta |A| |B|$, over all sets $A, B \subseteq V(G)$:
$$\|G - \delta\|_{\square} = \frac{1}{|V(G)|^2} \max_{A, B \subseteq V(G)} \big| |E_G(A, B)| - \delta |A| |B| \big|.$$
We say that $G$ is \emph{$\varepsilon$-quasirandom} if $\|G - \delta\|_{\square} \leq \varepsilon$, where $\delta = 2|G|/v(G)^2$ denotes its edge density;
item $(i)$ in Theorem \ref{quasigraphs} is then precisely the assertion that $G$ is $c_1$-quasirandom.

We note that there is an equivalent `weighted' expression for the cut norm of a function $f: V \rightarrow \R$, where the maximum is taken over bounded functions rather than over sets:
$$\| f \|_{\square} = \max_{u,\, v:\, V \rightarrow [0, 1]} \big| \mathbb{E}_{x, y \in V} \big[ f(x, y) u(x) v(y) \big] \big|.$$
Indeed, since the expectation above is bilinear in $u$ and $v$, the extrema occur when $u$ and $v$ are $\{0, 1\}$-valued and can thus be identified with their supports.
This weighted formulation is sometimes more suitable than the one given before.


Another notion that will be very useful for us is that of homomorphism densities, which give a convenient way of counting copies of small graphs inside a large graph:

\begin{definition}
The \emph{homomorphism density} of a graph $F$ in a graph $G$, denoted $t(F, G)$, is the probability that a randomly chosen map $\phi: V(F) \rightarrow V(G)$ preserves edges:
\begin{align*}
    t(F, G) &= \mathbb{P}_{x_1, \dots, x_{v(F)} \in V(G)} \big( x_i x_j \in E(G) \text{ whenever } ij \in E(F) \big) \\
    &= \mathbb{E}_{x_1, \dots, x_{v(F)} \in V(G)} \Bigg[ \prod_{ij \in E(F)} G(x_i, x_j) \Bigg].
\end{align*}
\end{definition}

Those functions $\phi: V(F) \rightarrow V(G)$ which map edges of $F$ to edges of $G$ are called \emph{homomorphisms} from $F$ to $G$, which explains the terminology.
While a homomorphism might map several vertices of $F$ to a single vertex of $G$, if we assume that the considered graph $G$ is large, then only a negligible fraction of all possible maps will be degenerate in this sense;
one may then safely ignore the distinction between homomorphism density and subgraph density when discussing polynomial equivalence.

With this notation, and up to negligible lower-order terms, item $(ii)$ in Theorem \ref{quasigraphs} can be written as $t(F, G) = \delta^{|F|} \pm c_2 |F|$, and item $(iii)$ becomes $t(C_4, G) \leq \delta^4 + c_3$
(where we use $C_4$ to denote the $4$-cycle).

\subsection{Proof of the equivalence theorem}

Using the definitions and notation now developed, we can reformulate our first theorem in a more succinct and convenient way as follows.
We shall denote by $\lambda_1, \lambda_2, \dots, \lambda_n$ the eigenvalues of the adjacency matrix of $G$ ordered in decreasing absolute value:
$|\lambda_1| \geq |\lambda_2| \geq \dots \geq |\lambda_n|$.

\begin{thm}[= Theorem \ref{quasigraphs}] \label{quasigraphs2}
Let $G$ be a graph with $n$ vertices and edge density $\delta$.
Then the following statements are polynomially equivalent:
\begin{itemize}
    \item[$(i)$] \emph{$G$ has `low discrepancy':} $\| G - \delta \|_{\square} \leq c_1$.
    \item[$(ii)$] \emph{$G$ `correctly' counts all graphs:} $t(F, G) = \delta^{|F|} \pm c_2 |F|$ for all graphs $F$.
    \item[$(iii)$] \emph{$G$ has `few' 4-cycles:} $t(C_4, G) \leq \delta^4 + c_3$.
    \item[$(iv)$] \emph{Only the first eigenvalue is `large':} $\lambda_1 = (\delta \pm c_4)n$, $|\lambda_2| \leq c_4 n$.
\end{itemize}
\end{thm}

\begin{proof}
$(i) \Rightarrow (ii)$:
This implication is usually known as the \emph{counting lemma};
the simple proof we present here is taken from \cite{ConvSeqDenseGraphs}.
Let $m$ be the number of vertices of the graph $F$, and assume $V(F) = [m]$ and $E(F) = \{e_1, \dots, e_{|F|}\}$.
For $k = 1, \dots, |F|$, let $i_k$, $j_k$ be the endpoints of the edge $e_k$.
Then $\big| t(F, G) - \delta^{|F|} \big|$ can be rewritten as
\begin{align*}
    \Bigg|\mathbb{E}_{x_1, \dots, x_m \in V(G)}& \Bigg[ \prod_{ij \in E(F)}{G(x_i, x_j)} - \delta^{|F|} \Bigg] \Bigg| \\
    &= \Bigg| \mathbb{E}_{x_1, \dots, x_m \in V(G)} \Bigg[ \sum_{k = 1}^{|F|}{\delta^{k - 1} \big( G(x_{i_k}, x_{j_k}) - \delta \big) \prod_{\ell = k+1}^{|F|}{G(x_{i_{\ell}}, x_{j_{\ell}})}} \Bigg] \Bigg| \\
    &\leq \sum_{k = 1}^{|F|}{\delta^{k - 1} \Bigg| \mathbb{E}_{x_1, \dots, x_m \in V(G)} \Bigg[ \big( G(x_{i_k}, x_{j_k}) - \delta \big) \prod_{\ell = k+1}^{|F|}{G(x_{i_{\ell}}, x_{j_{\ell}})} \Bigg] \Bigg|}.
\end{align*}

Consider the $k$-th term of this last sum, and assume for notational convenience that $i_k = 1$ and $j_k = 2$.
Then for any fixed $x_3, \dots, x_{m} \in V(G)$ we have
\begin{multline*}
\Bigg| \mathbb{E}_{x_1, x_2 \in V(G)} \Bigg[ \big( G(x_1, x_2) - \delta \big) \prod_{\ell = k+1}^{|F|}{G(x_{i_{\ell}}, x_{j_{\ell}})} \Bigg] \Bigg| \\
= \big| \mathbb{E}_{x_1, x_2 \in V(G)} \big[ (G(x_1, x_2) - \delta) a_k(x_1) b_k(x_2) \big] \big|,
\end{multline*}
where $a_k$ and $b_k$ are the functions given by
$$a_k(x_1) := \prod_{\substack{\ell > k \\ 1 \in e_{\ell}}}{G(x_{i_{\ell}}, x_{j_{\ell}})}
\hspace{3mm} \text{and} \hspace{3mm} b_k(x_2) := \prod_{\substack{\ell > k \\ 1 \notin e_{\ell}}}{G(x_{i_{\ell}}, x_{j_{\ell}})}.$$
By hypothesis $\|G - \delta\|_{\square} \leq c_1$, so the expression on the right is at most $c_1$ for all fixed $x_3, \dots, x_{m}$.
Thus
$$\Bigg| \mathbb{E}_{x_1, \dots, x_m \in V(G)} \Bigg[ \big( G(x_{i_k}, x_{j_k}) - \delta \big) \prod_{\ell = k+1}^{|F|}{G(x_{i_{\ell}}, x_{j_{\ell}})} \Bigg] \Bigg| \leq c_1$$
for all $1 \leq k \leq |F|$, implying that $\big| t(F, G) - \delta^{|F|} \big| \leq c_1 |F|$.
We may then take $c_2 = c_1$.

\medskip
$(ii) \Rightarrow (iii)$:
This is just a special case, and we can take $c_3 = 4c_2$.

\medskip
$(iii) \Rightarrow (iv)$:
Suppose the vertices of $G$ are labelled by $\{1, 2, \dots, n\}$, and denote the adjacency matrix of $G$ by $A$.
First note that $\lambda_1 \geq \delta n$, since $\lambda_1$ is positive and
$$|\lambda_1| = \max_{v \neq 0} \frac{\|Av\|_2}{\|v\|_2} \geq \frac{\|Ae\|_2}{\|e\|_2} \geq \frac{e^t A e}{\|e\|_2^2} = \frac{\delta n^2}{n} = \delta n,$$
where $e = (1, 1, \dots, 1)^t$ is the all-ones vector.

An easy induction argument shows that, for any $k \in \mathbb{N}$, the entry $(i, j)$ on the matrix $A^k$ counts the number of walks of length $k$ on $G$ which start at vertex $i$ and end at vertex $j$.
In particular, $(A^4)_{ii}$ counts the number of labelled 4-cycles starting (and ending) at vertex $i$.
This implies that
$$t(C_4, G) = \frac{1}{n^4} \sum_{i=1}^n (A^4)_{ii} = \frac{1}{n^4} tr(A^4) = \frac{1}{n^4} \sum_{i=1}^n \lambda_i^4 \geq \frac{\lambda_1^4 + \lambda_2^4}{n^4}.$$

By assumption we have that $t(C_4, G) \leq \delta^4 + c_3$, which together with $\lambda_1 \geq \delta n$ implies that $\lambda_1 \leq (\delta + c_3^{1/4}) n$ and $|\lambda_2| \leq c_3^{1/4} n$.
We may then take $c_4 = c_3^{1/4}$.

\medskip
$(iv) \Rightarrow (i)$:
We will first show that $\|A - \delta J\|_{sp} \leq 6 c_4^{1/2} n$, where $A$ is the adjacency matrix of $G$, $J = e e^t$ is the $n \times n$ all-ones matrix and $\|\cdot\|_{sp}$ is the spectral norm
(i.e. the largest singular value of the matrix).
For this, let $\{v_1, v_2, \dots, v_n\}$ be an orthonormal basis of eigenvectors of $A$, where $v_i$ is an eigenvector associated to the eigenvalue $\lambda_i$ for all $1 \leq i \leq n$.

If we suppose the graph $G$ is regular of degree $pn$, then the result we want to prove is simple:
in this case $e_1 := e/\sqrt{n}$ is a unitary eigenvector of $A$ with eigenvalue $\lambda_1 = \delta n$, and so
$$A - \delta J = A - \delta n e_1 e_1^t = \sum_{j=2}^n \lambda_j v_j v_j^t$$
has spectral norm equal to $|\lambda_2| \leq c_4 n$.

If we do not suppose $G$ is regular, then we can decompose
\begin{align*}
    A - \delta J &= \sum_{i=1}^n \lambda_i v_i v_i^t - \delta n e_1 e_1^t \\
    &= \lambda_1 v_1 v_1^t - \delta n v_1 v_1^t + \sum_{j=2}^n \lambda_j v_j v_j^t + \delta n v_1 v_1^t - \delta n e_1 e_1^t \\
    &= M_1 + M_2 + M_3,
\end{align*}
where
$$M_1 = (\lambda_1 - \delta n)v_1 v_1^t, \hspace{3mm} M_2 = \sum_{j=2}^n \lambda_j v_j v_j^t, \hspace{3mm} M_3 = \delta n (v_1 v_1^t - e_1 e_1^t).$$
Clearly $\|M_1\|_{sp} = |\lambda_1 - \delta n| \leq c_4 n$ and $\|M_2\|_{sp} = |\lambda_2| \leq c_4 n$.

Let us now bound $\|M_3\|_{sp}$.
Since $M_3$ is symmetric real, we know that
$$\|M_3\|_{sp} = \max_{\|u\|_2 = 1} |u^t M_3 u|.$$
Moreover, for any fixed $u \in \R^n$ with $\|u\|_2 = 1$ we have that
\begin{align*}
    |u^t M_3 u| &= \delta n \cdot |(u^t v_1)^2 - (u^t e_1)^2| \\
    &= \delta n \cdot |u^t (v_1 + e_1)| \cdot |u^t (v_1 - e_1)| \\
    &\leq 2\delta n \cdot \|v_1 - e_1\|_2,
\end{align*}
where the last inequality follows from Cauchy-Schwarz.
It thus suffices to bound $\|v_1 - e_1\|_2$.

Decompose $e_1 = \mu v_1 + w$, where $\mu = e_1^t v_1$ and $w$ is orthogonal to $v_1$.
Note that $\|w\|_2 \leq 1$ (by Pythagoras' theorem) and that, up to changing $v_1$ by $-v_1$, we can assume $\mu \geq 0$.
Then
\begin{align*}
    \delta n = e_1^t A e_1 \leq \|A e_1\|_2 &= \|A (\mu v_1 + w)\|_2 \\
    &\leq \mu \|A v_1\|_2 + \|A w\|_2 \\
    &\leq \mu \lambda_1 + |\lambda_2| \cdot \|w\|_2 \\
    &\leq \mu (\delta + c_4) n + c_4 n
\end{align*}
$$\implies\, e_1^t v_1 = \mu \geq \frac{\delta - c_4}{\delta + c_4} \geq 1 - \frac{2c_4}{\delta}.$$
From this we deduce that
$$\|v_1 - e_1\|_2^2 = v_1^t v_1 - 2 v_1^t e_1 + e_1^t e_1 = 2 (1 - v_1^t e_1) \leq \frac{4 c_4}{\delta}.$$
We then have that $\|M_3\|_{sp} \leq 2\delta n \cdot \|v_1 - e_1\|_2 \leq 4 c_4^{1/2} n$ and
$$\|A - \delta J\|_{sp} \leq \|M_1\|_{sp} + \|M_2\|_{sp} + \|M_3\|_{sp} \leq 6 c_4^{1/2} n,$$
as wished.

The rest follows easily from Cauchy-Schwarz.
Indeed, for any subsets $X, Y \subseteq V(G)$ we have that
\begin{align*}
    \frac{1}{n^2} \Bigg| \sum_{x \in X} \sum_{y \in Y} (G(x, y) - \delta) \Bigg| &= \frac{1}{n^2} \Bigg| \sum_{i=1}^n \sum_{j=1}^n (A_{ij} - \delta) \mathbbm{1}_X(i) \mathbbm{1}_Y(j) \Bigg| \\
    &\leq \frac{1}{n^2} \|A -\delta J\|_{sp} \|\mathbbm{1}_X\|_2 \|\mathbbm{1}_Y\|_2 \\
    &\leq 6 c_4^{1/2}.
\end{align*}
We thus obtain property $(i)$ with $c_1 = 6 c_4^{1/2}$.
\end{proof}

\subsection{Quasirandom partite graphs} \label{PartiteGraphs}

Many of the results given in this section (and also their proofs) can be easily generalized to the case of \emph{partite graphs}, where the vertex set of the graph considered is partitioned into several classes with no edges inside any single class.
This greater generality will be needed when we consider graphs encoding linear systems of equations in additive groups (as will be done in later sections), and also when we discuss the graph regularity lemma in Section \ref{GraphReg}.

There is a natural notion of quasirandomness for partite graphs, which corresponds to the idea that their edges are uniformly distributed across each pair of partition classes.
In order to measure this, we will first extend the definition of cut norm to bipartite graphs and functions:

\begin{definition}
Given a function $f: V_1 \times V_2 \rightarrow \R$, we define its \emph{cut norm} by
$$\| f \|_{\square} := \max_{A\subseteq V_1,\, B \subseteq V_2} \big| \mathbb{E}_{x \in V_1,\, y \in V_2} \big[ f(x, y) A(x) B(y) \big] \big|.$$
If $G$ is a bipartite graph on $(V_1, V_2)$ with edge density $\delta := |G|/|V_1 \times V_2|$, we say $G$ is \emph{$\varepsilon$-quasirandom} if $\| G - \delta \|_{\square} \leq \varepsilon$.
\end{definition}

For a given graph $G$ and two disjoint subsets $U$, $W \subset V(G)$, let us denote by $G[U, W]$ the bipartite graph on $(U, W)$ whose edges are the restriction of $E(G)$ to $U \times W$.
If $G$ is an $\ell$-partite graph on $(V_1, \dots, V_{\ell})$, note that we can decompose it as an edge-disjoint union
$G = \bigcup_{1 \leq i < j \leq \ell} G[V_i, V_j]$
of $\binom{\ell}{2}$ bipartite graphs;
we then say $G$ is $\varepsilon$-quasirandom if each one of these bipartite induced subgraphs is $\varepsilon$-quasirandom.

As in the case of usual (non-partite) graphs, it is possible to estimate with high accuracy the number of copies of each small graph $F$ contained inside a large quasirandom partite graph $G$, just by knowing the edge density between each pair of partition classes.
It is usually more convenient to consider only those \emph{canonical} copies where each vertex of $F$ belongs to the `correct' partition class of $G$;
this is the idea behind the next definition:

\begin{definition}
Let $F$ and $G$ be $\ell$-partite graphs with vertex partition $(U_1, \dots, U_{\ell})$ and $(V_1, \dots, V_{\ell})$, respectively.
A map $\phi: V(F) \rightarrow V(G)$ is an \emph{$\ell$-partite function} if $\phi(U_i) \subseteq V_i$ for all $1 \leq i \leq \ell$;
if moreover $\phi$ is a homomorphism of $F$ on $G$ (i.e. it maps edges of $F$ to edges of $G$), we say that it is a \emph{canonical homomorphism}.
The \emph{canonical homomorphism density} of $F$ on $G$ is the probability that a uniformly chosen $\ell$-partite function is a (canonical) homomorphism:
$$t_{can}(F, G) := \mathbb{E}_{x_{i_1} \in V_1\; \forall i_1 \in U_1} \,\cdots\, \mathbb{E}_{x_{i_{\ell}} \in V_{\ell}\; \forall i_{\ell} \in U_{\ell}} \Bigg[ \prod_{ij \in E(F)} G(x_i, x_j) \Bigg].$$
\end{definition}

\begin{remark}
This definition depends not only on the edge set of the considered graphs $F$ and $G$ but also on their partition classes and how these classes are labeled;
we shall assume this data to be part of the description of partite graphs.
It is most commonly used when the vertex classes of the large graph $G$ are labelled by the vertices of the smaller graph $F$, which is then regarded as a $v(F)$-partite graph with a single vertex in each class.
\end{remark}

Using the notion of canonical homomorphisms one can easily obtain a generalization of the counting lemma adapted to the setting of partite graphs, whose proof is essentially identical to the one given (when proving Theorem \ref{quasigraphs2}) in the usual non-partite setting:

\begin{lem} [Counting lemma]
Let $F$ and $G$ be $\ell$-partite graphs with partition classes $(U_1, \dots, U_{\ell})$ and $(V_1, \dots, V_{\ell})$, respectively. For each vertex $j$ of $F$ let $\imath(j)$ be the index for which $j \in U_{\imath(j)}$, and denote the density of each bipartite graph $G[V_a, V_b]$ by $\delta_{ab}$.
If $G$ is $\varepsilon$-quasirandom, then we have
$$t_{can}(F, G) = \prod_{ij \in E(F)} \delta_{\imath(i) \imath(j)} \pm \varepsilon |F|.$$
\end{lem}

It is also possible to generalize the main equivalence theorem for quasirandom graphs to the setting of quasirandom \emph{bipartite} graphs.
In order to do so we only need to substitute the usual adjacency matrix of a bipartite graph for its bipartite adjacency
matrix,\footnote{If $G$ is a bipartite graph on $(V_1, V_2)$, then its \emph{bipartite adjacency matrix} is the $|V_1| \times |V_2|$ matrix whose element at position $(i, j)$ (with $i \in V_1$ and $j \in V_2$) is $1$ if $ij \in E(G)$ and $0$ otherwise.}
and let the definition of polynomial equivalence take into account the size of \emph{each one} of the partition classes of the graph
(so they are both assumed to be large enough depending on the parameters $c_i$).

For ease of reference we will present this generalization here, but leave the necessary modifications in the proof to the interested reader.

\begin{thm}
Let $G$ be a bipartite graph on $(V_1, V_2)$ with edge density $\delta$.
Then the following statements are polynomially equivalent:
\begin{itemize}
    \item[$(i)$] \emph{$G$ has low discrepancy:} $\| G - \delta \|_{\square} \leq c_1$.
    \item[$(ii)$] \emph{$G$ correctly counts all bipartite graphs:} $t_{can}(F, G) = \delta^{|F|} \pm c_2 |F|$ for all bipartite graphs $F$.
    \item[$(iii)$] \emph{$G$ has few 4-cycles:} $t_{can}(C_4, G) \leq \delta^4 + c_3$.
    \item[$(iv)$] \emph{Only the first singular value is large:} $\sigma_1 = (\delta \pm c_4) |V_1|^{1/2} |V_2|^{1/2}$ and $\sigma_2$ is at most $c_4 |V_1|^{1/2} |V_2|^{1/2}$, where $\sigma_1, \sigma_2$ are the two largest singular values of the bipartite adjacency matrix of $G$.
\end{itemize}
\end{thm}

\section{Uniformity and quasirandomness in additive groups} \label{GroupsSection}

Another fruitful setting for studying quasirandomness is that of \emph{subsets of additive groups}, usually called \emph{additive sets} for short, which are the main subjects of study in the area of additive combinatorics.

\begin{definition}
An \emph{additive group} is an Abelian group $G$ written additively (that is, with group operation denoted by $+$ and identity element by $0$).
We also define a multiplication operation $nx \in G$ for all $n \in \Z$ and $x \in G$ in the usual way:
$0x = 0$, $nx = x + x + \cdots + x$ ($n$ times) if $n>0$,
and $nx = (-n) (-x)$ if $n<0$.
\end{definition}

\begin{remark}
All additive groups considered here will be \emph{finite}, as this is the most natural setting for the kind of results we are interested in.
Most results regarding finite subsets of integers can also be (and many times \emph{are}) analyzed in this framework, by restricting to the first $n$ positive integers for some large enough $n$ and then embedding $[n]$ into the cyclic group $\Z_N$ for some $N$ sufficiently large to prevent `wrapping around'.
\end{remark}

In this setting our main goal is to identify suitable properties which are satisfied by randomly chosen sets (with high probability) and which capture the essence of such `random lack of structure', then study how these properties relate to each other and what interesting consequences one can deduce from them.

Let us now be more specific about which notion of quasirandomness we are interested in.
Since the only kind of structure intrinsic to this setting is that which comes from the group operation, we intuitively think of quasirandom sets as those which have no correlation with the additive structure of the group it is inserted in.
An interesting and useful way of making this idea precise is by using the \emph{Fourier transform}, whose use in additive problems in number theory dates back to Vinogradov's seminal work on his three primes theorem.

We shall now recall the basic definitions and results regarding Fourier analysis on additive groups that will be useful for us.
An excellent source for more details on Fourier analysis and its generalizations in additive combinatorics is Gowers' survey \cite{GeneralizationsFourier}, and we owe much of our presentation to that paper.

\subsection{Review of Fourier analysis on additive groups}

Let $G$ be a finite additive group.
In words, the Fourier transform of a function $f: G \rightarrow \mathbb{C}$ measures the correlation between $f$ and the \emph{characters} of the group $G$.
Let us then take a look at those first:

\begin{definition}
The \emph{characters} of an additive group $G$ are the group homomorphisms from $G$ to the complex multiplicative group $\mathbb{C}^{\times}$.
\end{definition}

More explicitly, a character is a map $\gamma: G \rightarrow \mathbb{C} \setminus \{0\}$ satisfying
$$\gamma(x+y) = \gamma(x)\gamma(y) \text{ for all } x, y \in G.$$
From this formula we immediately obtain that $\gamma(0) = 1$ (where 0 represents the identity element of $G$) and $\gamma(x-y) = \gamma(x)/\gamma(y)$.
Moreover, since the order of every element in $G$ divides its size $|G|$, we have that $|G| x := x+x+\dots+x$ ($|G|$ times) is equal to $0$ in $G$, and so $\gamma(x)^{|G|}=\gamma(0)=1$;
we conclude that $\gamma$ takes values on the $|G|$-th roots of unity, and in particular $\gamma(x)^{-1} = \overline{\gamma(x)}$.

It is easy to see that the pointwise multiplication (or division) of two characters is still a character, as is the identically one function $\mathbf{1}$ (we call it the \emph{trivial character}).
It follows that the set of characters of $G$ forms an Abelian group (with group operation of pointwise multiplication), which we call the \emph{dual group} of $G$ and denote by $\widehat{G}$.

The notion of characters in this generality might seem rather abstract at first, so the reader should keep in mind the following important example:

\begin{ex} \label{exmodn}
Let $G = \Z_n$ be the cyclic group of integers modulo $n$.
Since this group is generated by the element $1$, the value of the character at $1$ determines its value at every other element;
as we have seen that it takes values on the $n$-th roots of unity, it follows that the characters of $G$ are given by $\chi_r: x \mapsto e^{2 \pi i rx/n}$ for $r \in \Z_n$.
The dual group $\widehat{\Z_n}$ is then given by $\{\chi_r: r \in \Z_n\}$ with pointwise multiplication, and it is easy to see that $\chi_r \cdot \chi_s = \chi_{r+s}$ for all $r, s \in \Z_n$.
The map $r \mapsto \chi_r$ thus gives an isomorphism from $\Z_n$ to $\widehat{\Z_n}$.
\end{ex}

Together with the \emph{structure theorem of finite Abelian groups}, this example gives an explicit formula for the characters of any finite additive group $G$.
Indeed, the structure theorem gives a decomposition of $G$ as a direct product of cyclic groups $\Z_n$, and it is easy to show that the characters of $G$ are exactly the products of the characters of the cyclic groups in this decomposition.
This also shows that $G$ and $\widehat{G}$ are isomorphic, since this is true for the cyclic groups.

Another interesting example to consider is that of finite vector spaces $\F_p^n$ over a finite field of prime order $\F_p$, which are also frequent subjects of study in additive combinatorics:

\begin{ex}
Let $G = \F_p^n$, for some prime $p$ and integer $n \geq 1$.
By its obvious decomposition into a direct sum of $n$ copies of $\F_p \cong \Z_p$, we conclude from the last example that its characters are given by $\chi_y: x \mapsto e^{2 \pi i x \cdot y/p}$ for $y \in \F_p^n$ (where $\cdot$ here denotes the inner product in $\F_p^n$).
\end{ex}

A very useful property of characters is that they satisfy the following \emph{orthogonality relations}:
\begin{equation} \label{orthog}
    \mathbb{E}_{x \in G}[\chi(x)] = \mathbbm{1}_{\{\chi = \mathbf{1}\}} \hspace{3mm} \text{and} \hspace{3mm} \sum_{\gamma \in \widehat{G}} \gamma(y) = |G| \mathbbm{1}_{\{y = 0\}}.
\end{equation}
Indeed, the case when $\chi$ is the trivial character $\mathbf{1}$ or when $y = 0$ is clear, so let us suppose $\chi$ is a non-trivial character and $y \in G$ is a non-zero element such that $\chi(y) \neq 1$.
Since
$$\mathbb{E}_{x \in G}[\chi(x)] = \mathbb{E}_{x \in G}[\chi(x+y)] = \chi(y)\mathbb{E}_{x \in G}[\chi(x)],$$
and similarly
$$\sum_{\gamma \in \widehat{G}} \gamma(y) = \sum_{\gamma \in \widehat{G}} (\chi \gamma)(y) = \chi(y) \sum_{\gamma \in \widehat{G}} \gamma(y),$$
it follows that $\mathbb{E}_{x \in G}[\chi(x)] = 0$ and $\sum_{\gamma \in \widehat{G}} \gamma(y) = 0$.

Using these relations one can easily prove that the characters form an orthonormal basis of $\mathbb{C}^G$, with inner product given by $\langle f, g \rangle_{L^2(G)} := \mathbb{E}_{x \in G} \big[ f(x) \overline{g(x)} \big]$.
Indeed, if $\gamma$ and $\chi$ are two distinct characters, then $\gamma \chi^{-1} \in \widehat{G}$ is a non-trivial character and so by the orthogonality relations (\ref{orthog}) we have
$$\langle \gamma, \chi \rangle_{L^2(G)} = \mathbb{E}_{x \in G} \big [\gamma(x) \overline{\chi(x)} \big] = \mathbb{E}_{x \in G} \big[ \gamma\chi^{-1}(x) \big] = 0.$$
Moreover, there are $|\widehat{G}| = |G|$ distinct characters which are all linearly independent by orthogonality, and so they span the $|G|$-dimensional vector space $\mathbb{C}^G$.

It is then natural to expand a function $f$ in terms of this basis, and it is from doing so that we obtain the Fourier transform:

\begin{definition}
Given a function $f: G \rightarrow \mathbb{C}$, we define its \emph{Fourier transform} as the function $\widehat{f}: \widehat{G} \rightarrow \mathbb{C}$ given by
$$\widehat{f}(\gamma) := \langle f, \gamma \rangle_{L^2(G)} = \mathbb{E}_{x \in G} \big[ f(x) \overline{\gamma(x)} \big].$$
\end{definition}

Again, it might be instructive to keep the following example in mind, which is very similar to the usual Fourier transform on the circle $\R/\Z$.

\begin{ex}
If $G = \Z_n$, the Fourier transform translates into the usual \emph{discrete Fourier transform}:
using the same notation as in Example \ref{exmodn}, we can write $\widehat{f}(\chi_r) = \mathbb{E}_{x \in \Z_n} \big[ f(x) e^{-2 \pi i rx/n} \big]$.
Note the similarity between this formula and the formula $\widehat{f}(\xi) = \int_0^1 f(x) e^{-2 \pi i \xi x} dx$ for the classical Fourier transform on the circle $\R/\Z$ (which will not be used in this paper).
\end{ex}

Writing a function $f$ in the orthonormal basis of characters immediately gives us the \emph{Fourier inversion formula}:
$$f(x) \,=\, \sum_{\gamma \in \widehat{G}} \langle f,\, \gamma \rangle_{L^2(G)} \gamma(x) \,=\, \sum_{\gamma \in \widehat{G}} \widehat{f}(\gamma) \gamma(x).$$

Note that so far we have always used the expectation notation $\mathbb{E}_{x \in G}$ for the `physical space' $G$ and the usual sum $\sum_{\gamma \in \widehat{G}}$ for the `frequency space' $\widehat{G}$.
As a matter of fact, even though $G$ and $\widehat{G}$ are isomorphic, it is more convenient to use different measures on them:
\begin{itemize}
    \item For $G$ we use the normalized measure $\mathbb{E}_{x \in G}$, and denote the associated Euclidean space $\mathbb{C}^G$ by $L^2(G)$.
    \item For $\widehat{G}$ we use the counting measure $\sum_{\gamma \in \widehat{G}}$, and denote the associated Euclidean space $\mathbb{C}^{\widehat{G}}$ by $\ell^2(\widehat{G})$.
\end{itemize}
This is done so that the Fourier transform becomes an \emph{isometry} from $L^2(G)$ to $\ell^2(\widehat{G})$,
a fact that follows easily from the Fourier inversion formula and orthonormality of the characters:
\begin{align*}
    \langle f,\, f \rangle_{L^2(G)} \,&=\, \Bigg\langle \sum_{\gamma \in \widehat{G}} \widehat{f}(\gamma) \gamma,\, \sum_{\chi \in \widehat{G}} \widehat{f}(\chi) \chi \Bigg\rangle_{L^2(G)} \\
    &=\, \sum_{\gamma, \chi \in \widehat{G}} \widehat{f}(\gamma) \overline{\widehat{f}(\chi)} \,\langle \gamma,\, \chi \rangle_{L^2(G)} \\
    &=\, \sum_{\gamma \in \widehat{G}} \widehat{f}(\gamma) \overline{\widehat{f}(\gamma)} \\
    &=\, \big\langle \widehat{f},\, \widehat{f} \big\rangle_{\ell^2(\widehat{G})}.
\end{align*}
This result, which can also be written as $\|f\|_{L^2} = \|\widehat{f}\|_{\ell^2}$, is known as \emph{Parseval's identity}.

\subsection{Uniform additive sets}

Let us now return to our main subject of study in this section, namely quasirandom subsets of additive groups.

To see the connection between Fourier analysis and quasirandomness, we make the following simple observation:
the Fourier transform of a function $f$ evaluated at a character $\gamma$ gives how much $f$ correlates with $\gamma$.
Since the characters of $G$ encode the additive structure of the group, correlation with a (non-trivial) character is a good measure of how much additive structure a given function or set has.

With this in mind, we can now define a measure of how quasirandom or uniform a given function/set is:

\begin{definition}
A function $f: G \rightarrow \R$ is said to be \emph{Fourier $\varepsilon$-uniform} if $|\widehat{f}(\gamma)| \leq \varepsilon$ for all $\gamma \in \widehat{G} \setminus \{\mathbf{1}\}$.
A set $A \subseteq G$ is Fourier $\varepsilon$-uniform if its indicator function is.
\end{definition}

We then informally say that a set $A$ is uniform if it is Fourier $\varepsilon$-uniform for some small $\varepsilon > 0$.
Note that $\widehat{A}(\mathbf{1})$ is just the density of $A$ in $G$, which is the reason why we take the trivial character out of our definition.

It is easy to show that random sets are very uniform with high probability, providing a first indication that this notion is a good measure of quasirandomness.
Indeed, suppose $A \subseteq G$ is a random set with $\mathbb{P}(a \in A) = p$ independently for all $a \in G$ (and some $0 < p < 1$ fixed).
For a given nontrivial character $\gamma \in \widehat{G}$, let us define the families of random variables $(X_a)_{a \in G}$ and $(Y_a)_{a \in G}$ by
$$X_a := A(a) \text{Re}(\gamma(a)) \hspace{3mm} \text{and} \hspace{3mm} Y_a := -A(a) \text{Im}(\gamma(a)) \hspace{3mm} \text{for all } a \in G.$$
Denoting their sums by $X := \sum_{a \in G} X_a$, $Y := \sum_{a \in G} Y_a$ we see that
\begin{align*}
    &\mathbb{E}[X] = p \,\text{Re} \bigg( \sum_{a \in G} \gamma(a) \bigg) =0, &\var(X) = p(1 - p) \sum_{a \in G} \text{Re}(\gamma(a))^2 \leq \frac{|G|}{4}, \\
    &\mathbb{E}[Y] = -p \,\text{Im} \bigg( \sum_{a \in G} \gamma(a) \bigg) =0, &\var(Y) = p(1 - p) \sum_{a \in G} \text{Im}(\gamma(a))^2 \leq \frac{|G|}{4},
\end{align*}
and by definition $X + iY = |G| \widehat{A}(\gamma)$.
Using Chernoff's inequality (Lemma \ref{Chernoff}) for each of these families separately with $\lambda = \varepsilon \sqrt{2|G|}$ (for some small enough $\varepsilon > 0$) we conclude that
$$\mathbb{P} \big( |\widehat{A}(\gamma)| \geq \varepsilon \big) \,\leq\, \mathbb{P} \bigg( |X| \geq \frac{\varepsilon |G|}{\sqrt{2}} \bigg) + \mathbb{P} \bigg( |Y| \geq \frac{\varepsilon |G|}{\sqrt{2}} \bigg) \,\leq\, 4 e^{-\varepsilon^2 |G|/2}.$$
Using union bound over all $\gamma \in \widehat{G} \setminus \{\mathbf{1}\}$, we conclude that\footnote{We in fact obtain from this argument that $A$ is (say) Fourier $\sqrt{5 \log |G|/|G|}$-uniform with probability at least $1 - 1/|G|$, provided $|G|$ is large enough depending on $p \in (0, 1)$.} $A$ is Fourier $\varepsilon$-uniform with high probability for any $\varepsilon > 0$ fixed and $|G|$ large.

As with graphs, there are several other natural properties usually satisfied by random subsets $A \subseteq G$ that are all roughly equivalent to $A$ being uniform, and this provides a much stronger indication that uniformity is a good measure of quasirandomness.
We can then obtain a similar result relating such properties as the one we got for quasirandom graphs (see Theorem \ref{linearsets} below).

In fact, the connection to quasirandom graphs is rather strong, as will be made clear by considering the \emph{Cayley graph} of a uniform set $A$.
For our purposes it will be better to consider a slightly different definition of Cayley graphs than the usual one, which admits a natural generalization to hypergraphs as we will see in Section \ref{Cayley}.

\begin{definition}
Given a subset $A \subseteq G$, we define its \emph{Cayley graph} $\Gamma_A$ by
$$V(\Gamma_A) = G, \hspace{3mm} E(\Gamma_A) = \{xy: x + y \in A\}.$$
\end{definition}

With these definitions we can now state our main result of this section.
It was first obtained by Chung and Graham \cite{QuasirandomZn} in the particular case of the group $\Z_n$, but the methods we use here work just as well for any other finite additive group $G$.

\begin{thm}[Equivalence theorem for uniform sets] \label{linearsets}
Let $G$ be an additive group of order $n$ and let $A \subseteq G$ be a set of size $|A| = \delta n$.
Then the following are polynomially equivalent:
\begin{itemize}
    \item[$(i)$] \emph{Fourier uniformity:}
    $|\widehat{A}(\gamma)| \leq c_1$ for all non-trivial characters $\gamma$.
    \item[$(ii)$] \emph{Additive quadruples:}
    There are at most $(\delta^4 + c_2)n^3$ solutions in $A$ of the equation $x + y = z + w$.
    \item[$(iii)$] \emph{Strong translation:}
    For all sets $B \subseteq G$, all but at most $c_3 n$ elements $x \in G$ satisfy $|A \cap (B + x)| = \delta |B| \pm c_3 n$.
    \item[$(iv)$] \emph{Weak translation:}
    All but at most $c_4 n$ elements $x \in G$ satisfy $|A \cap (A + x)| = \delta^2 n \pm c_4 n$.
    \item[$(v)$] \emph{Cayley graph:}
    The Cayley graph $\Gamma_A$ is $c_5$-quasirandom.
\end{itemize}
\end{thm}

\begin{remark}
Chung and Graham also proved (when $G = \Z_n$) that several other properties are polynomially equivalent to $(i) - (v)$, but we will restrict our attention to these five and refer the interested reader to their paper \cite{QuasirandomZn} for the full result;
see also \cite{LinearQuasirandomness}.
\end{remark}

\subsection{The $U^2$ norm} \label{U2section}

In order to prove Theorem \ref{linearsets}, it will be useful to introduce a new norm (due to Gowers \cite{NewProofLengthFour}) which also measures the uniformity of a function:

\begin{definition}
Given a real function $f: G \rightarrow \R$, we define its $U^2$ norm by
$$\left\|f\right\|_{U^2(G)} := \mathbb{E}_{x, h_1, h_2 \in G} \big[f(x) f(x + h_1) f(x + h_2) f(x + h_1 + h_2)\big]^{1/4}.$$
The notion of quasirandomness measured by the $U^2$ norm is called \emph{linear uniformity}.
\end{definition}

\begin{remark}
While not immediately obvious, this definition does indeed give a norm.
This follows, for instance, from Lemma \ref{U2l4} below.
\end{remark}

Note that quadruples of the form $(x,\, x+h_1+h_2,\, x+h_1,\, x+h_2)$ are the same as `additive quadruples' $(x, y, z, w)$ satisfying $x + y = z + w$, and so $\left\|f\right\|_{U^2(G)}^4$ can be seen as a weighted count of additive quadruples.
Property $(ii)$ of our last theorem might then serve as a motivation for this (perhaps mysterious-looking) definition.

Recall that we had already defined a measure for the uniformity of a function, based on its Fourier transform.
It turns out that the $U^2$ norm has a close connection to the Fourier transform, and these measures of uniformity are compatible with each other (at least in the case of bounded functions, as we are interested in here).
This is a simple consequence of the next lemma:

\begin{lem} \label{U2l4}
    For all real functions $f: G \rightarrow \R$, we have
    $\|f\|_{U^2(G)} = \|\widehat{f}\|_{\ell^4(\widehat{G})}$.
\end{lem}

\begin{proof}
Since $f$ is real-valued, for all $\gamma \in \widehat{G}$ we have that
\begin{align*}
    |\widehat{f}(\gamma)|^4 &= \mathbb{E}_{x \in G} \big[ f(x) \overline{\gamma(x)} \big] \, \mathbb{E}_{y \in G} \big[ f(y) \overline{\gamma(y)} \big] \, \mathbb{E}_{z \in G} \big[ f(z) \gamma(z) \big] \, \mathbb{E}_{w \in G} \big[ f(w) \gamma(w) \big] \\
    &= \mathbb{E}_{x, y, z, w \in G} \big[ f(x) f(y) f(z) f(w) \gamma(-x-y+z+w) \big].
\end{align*}
Using the orthogonality relations of characters (\ref{orthog}) we then obtain
\begin{align*}
    \| \widehat{f} \|_{\ell^4(\widehat{G})}^4 &= \mathbb{E}_{x, y, z, w \in G} \bigg[ f(x) f(y) f(z) f(w) \sum_{\gamma \in \widehat{G}} \gamma(-x-y+z+w) \bigg] \\
    &= \mathbb{E}_{x, y, z, w \in G} \big[ f(x) f(y) f(z) f(w) \cdot |G| \mathbbm{1}_{\{ x + y = z + w \}} \big].
\end{align*}

Now we note that, when $x, \, h_1, \, h_2$ are uniformly distributed over $G$, the quadruple $(x,\, x+h_1+h_2,\, x+h_1,\, x+h_2)$ is uniformly distributed over all solutions in $G$ to $x+y=z+w$.
Thus the last expression is equal to
$$\mathbb{E}_{x, h_1, h_2 \in G} \big[ f(x) f(x + h_1 + h_2) f(x + h_1) f(x + h_2) \big] \,=\, \|f\|_{U^2(G)}^4,$$
finishing the proof.
\end{proof}

An important property of the $U^2$ norm, which might help explain why it is more suitable for us than the Fourier analytic notion of uniformity, is that it satisfies a kind of Cauchy-Schwarz inequality.
Let us define the \emph{generalized inner product} $\langle \cdot \rangle_{U^2(G)}$ by
$$\langle f_1, f_2, f_3, f_4 \rangle_{U^2(G)} := \mathbb{E}_{x, h_1, h_2 \in G} \big[ f_1(x) f_2(x + h_1) f_3(x + h_2) f_4(x + h_1 + h_2) \big],$$
so that $\|f\|_{U^2(G)} = \langle f, f, f, f \rangle_{U^2(G)}^{1/4}$.
We then have:

\begin{lem}[Gowers-Cauchy-Schwarz inequality] \label{U2GCS}
    For any real functions $f_1$, $f_2$, $f_3$, $f_4: G \rightarrow \R$ we have
    $$\langle f_1, f_2, f_3, f_4 \rangle_{U^2(G)} \leq \| f_1 \|_{U^2(G)} \| f_2 \|_{U^2(G)} \| f_3 \|_{U^2(G)} \| f_4 \|_{U^2(G)}.$$
\end{lem}

\begin{proof}
By the usual Cauchy-Schwarz inequality applied to the variable $h_1$, we see that
\begin{align*}
    \langle f_1, f_2, f_3, f_4 \rangle_{U^2}^2 &= \mathbb{E}_{x, h_1, h_2 \in G} \big[ f_1(x) f_2(x + h_1) f_3(x + h_2) f_4(x + h_1 + h_2) \big]^2 \\
    &= \mathbb{E}_{h_1} \Big[ \mathbb{E}_{x} \big[ f_1(x) f_2(x + h_1) \big] \cdot \mathbb{E}_{y} \big[f_3(y) f_4(y + h_1) \big] \Big]^2 \\
    &\leq \mathbb{E}_{h_1} \Big[\mathbb{E}_{x} \big[ f_1(x) f_2(x + h_1) \big]^2 \Big]
    \cdot \mathbb{E}_{h_1} \Big[ \mathbb{E}_{y} \big[ f_3(y) f_4(y + h_1) \big]^2 \Big] \\
    &= \langle f_1, f_2, f_1, f_2 \rangle_{U^2} \langle f_3, f_4, f_3, f_4 \rangle_{U^2}.
\end{align*}
Applying the same argument to the variable $h_2$ instead of $h_1$, we obtain
$$\langle f_1, f_2, f_3, f_4 \rangle_{U^2}^2 \leq \langle f_1, f_1, f_3, f_3 \rangle_{U^2} \langle f_2, f_2, f_4, f_4 \rangle_{U^2}.$$
We conclude by using both inequalities one after the other:
\begin{align*}
    \langle f_1, f_2, f_3, f_4 \rangle_{U^2}^4 &\leq \langle f_1, f_2, f_1, f_2 \rangle_{U^2}^2 \langle f_3, f_4, f_3, f_4 \rangle_{U^2}^2 \\
    &\leq \langle f_1, f_1, f_1, f_1 \rangle_{U^2} \langle f_2, f_2, f_2, f_2 \rangle_{U^2} \langle f_3, f_3, f_3, f_3 \rangle_{U^2} \langle f_4, f_4, f_4, f_4 \rangle_{U^2} \\
    &= \| f_1 \|_{U^2}^4 \| f_2 \|_{U^2}^4 \| f_3 \|_{U^2}^4 \| f_4 \|_{U^2}^4.
\end{align*}

\vspace*{-1.7\baselineskip}
\end{proof}

With these preparations, we are now ready to prove Theorem \ref{linearsets}.

\subsection{Proof of the equivalence theorem}

First of all, we note that condition $(ii)$ is exactly equivalent to saying that $\|A\|_{U^2(G)}^4 \leq \delta^4 + c_2$.
This will be the central property that we will use to prove the equivalences.

\begin{proof}[Proof of Theorem \ref{linearsets}]
$(i) \Rightarrow (ii)$:
Suppose $|\widehat{A}(\mathbf{1})| = \mathbb{E}_{x \in G}[A(x)] = \delta$ and $|\widehat{A}(\gamma)| \leq c_1$ for all $\gamma \in \widehat{G} \setminus \{\mathbf{1}\}$.
Then
\begin{align*}
	\|A\|_{U^2(G)}^4 = \| \widehat{A} \|_{\ell^4(\widehat{G})}^4 &= \delta^4 + \sum_{\gamma \neq \mathbf{1}} |\widehat{A}(\gamma)|^4 \\
	&\leq \delta^4 + \sum_{\gamma \neq \mathbf{1}} c_1^2 |\widehat{A}(\gamma)|^2 \\
	&\leq \delta^4 + c_1^2 \| \widehat{A} \|_{\ell^2(\widehat{G})}^2 \\
	&= \delta^4 + c_1^2 \| A \|_{L^2(G)}^2 \leq \delta^4 + c_1^2,
\end{align*}
so we can take $c_2 = c_1^2$.

\medskip
$(ii) \Rightarrow (i)$:
If $\|A\|_{U^2(G)}^4 \leq \delta^4 + c_2$, then
$$\delta^4 + c_2 \geq \| \widehat{A} \|_{\ell^4(\widehat{G})}^4 = \delta^4 + \sum_{\gamma \neq \mathbf{1}} |\widehat{A}(\gamma)|^4 \geq \delta^4 + \max_{\gamma \neq \mathbf{1}} |\widehat{A}(\gamma)|^4.$$
This implies that $\max_{\gamma \neq \mathbf{1}} |\widehat{A}(\gamma)| \leq c_2^{1/4}$, and so we can take $c_1 = c_2^{1/4}$.

\medskip
$(ii) \Rightarrow (iii)$:
We first note that, for any fixed $x \in G$, we have
$$\left| A \cap (B - x) \right| - \delta|B| \,=\, \sum_{y \in G} (A(y) - \delta) B(x + y).$$
Using this identity and the Gowers-Cauchy-Schwarz inequality we see that
\begin{align*}
    \sum_{x \in G} \big( \left| A \cap (B - x) \right| - \delta|B| \big)^2 &= \sum_{x \in G} \bigg( \sum_{y \in G} (A(y) - \delta) B(x + y) \bigg)^2 \\
    &= \sum_{x, y, z \in G} (A(y) - \delta) B(x + y) (A(z) - \delta) B(x + z) \\
    &= n^3 \langle A - \delta, \, B, \, A - \delta, \, B \rangle_{U^2(G)} \\
    &\leq n^3 \|A - \delta\|_{U^2(G)}^2.
\end{align*}
Defining the function $f(x) := A(x) - \delta$ in $G$, we easily see that $\widehat{f}(\mathbf{1}) = 0$ and $\widehat{f}(\gamma) = \widehat{A}(\gamma)$ for all $\gamma \in \widehat{G} \setminus \{\mathbf{1}\}$.
Supposing $\|A\|_{U^2(G)}^4 \leq \delta^4 + c_2$, we then obtain
$$\|A - \delta\|_{U^2(G)}^4 = \| \widehat{f} \|_{\ell^4(\widehat{G})}^4 = \| \widehat{A} \|_{\ell^4(\widehat{G})}^4 - \delta^4 = \|A\|_{U^2(G)}^4 - \delta^4 \leq c_2.$$

If less than $(1 - c_3) n$ values $x \in G$ satisfy $\left| A \cap (B - x) \right| = \delta|B| \pm c_3 n$, then
$$\sum_{x \in G} \big( \left| A \cap (B - x) \right| - \delta|B| \big)^2 > c_3 n \cdot (c_3 n)^2 = c_3^3 n^3.$$
It thus suffices to take $c_3 = c_2^{1/6}$ for this last inequality to be incompatible with our previous bound.

\medskip
$(iii) \Rightarrow (iv)$:
This is just a special case, and we may take $c_4 = c_3$.

\medskip
$(iv) \Rightarrow (ii)$:
As in the proof that $(ii) \Rightarrow (iii)$, we see that
$$\sum_{x \in G} |A \cap (A + x)|^2 = n^3 \|A\|_{U^2(G)}^4.$$
Assuming $c_4 \leq 1$ (as otherwise we may just take $c_2 = 1$), we conclude that
\begin{equation*}
    n^3 \|A\|_{U^2(G)}^4 \leq n \cdot (\delta^2 + c_4)^2 n^2 + c_4 n \cdot n^2 \leq (\delta^4 + 4 c_4) n^3.
\end{equation*}
We may then take $c_2 = 4 c_4$.

\medskip
$(ii) \Leftrightarrow (v)$:
We will show that property $(ii)$ applied to $A$ is actually the same as property $(iii)$ of Theorem \ref{quasigraphs2} for quasirandom graphs applied to $\Gamma_A$.
Indeed,
\begin{align*}
    t(C_4, \Gamma_A) &= \mathbb{E}_{a, b, c, d \in G} \big[ \Gamma_A(a, b) \Gamma_A(b, c) \Gamma_A(c, d) \Gamma_A(d, a) \big] \\
    &= \mathbb{E}_{a, b, c, d \in G} \big[ A(a+b) A(b+c) A(c+d) A(d+a) \big].
\end{align*}
Let us now make the change of variables $x := a+b$, $h_1 := c-a$, $h_2 := d-b$.
It is easy to see that $x, h_1, h_2$ are uniformly distributed on $G$, so the last expression is equal to
$$\mathbb{E}_{x, h_1, h_2 \in G} \big[ A(x) A(x+h_1) A(x+h_1+h_2) A(x+h_2) \big] = \|A\|_{U^2(G)}^4.$$
Thus $t(C_4, \Gamma_A) \leq \delta^4 + c_2$ if and only if $\|A\|_{U^2(G)}^4 \leq \delta^4 + c_2$, as wished.

By the proof of Theorem \ref{quasigraphs2}, we may then take $c_5 = 6c_2^{1/8}$ (for $(ii) \Rightarrow (v)$) or $c_2 = 4c_5$ (for $(v) \Rightarrow (ii)$).
\end{proof}

\subsection{Application: counting linear configurations} \label{countLinearGraph}


Many problems in additive combinatorics can be cast in the following general form:
given a set $A \subset G$ and a system $\Phi$ of linear forms $\phi_1, \dots, \phi_m: G^k \rightarrow G$, how many elements $x_1, \dots, x_k \in G$ are there for which $\phi_1(x_1, \dots, x_k)$, $\dots$, $\phi_m(x_1, \dots, x_k)$ simultaneously belong to $A$?
This is the kind of question where the theory of quasirandomness comes in useful.

For instance, we have seen in Theorem \ref{linearsets} that uniformity suffices for us to count additive quadruples in $A$:
if $A$ is Fourier $\varepsilon$-uniform, then it contains between $\delta^4 n^3$ and $(\delta^4 + \varepsilon^2) n^3$ quadruples $(x, y, z, w)$ satisfying to the equation $x + y = z + w$.
Such quadruples are the same as the image of $G^3$ by the system of linear forms
$$\Phi = (\phi_1, \phi_2, \phi_3, \phi_4): (x_1, x_2, x_3) \mapsto (x_1,\, x_1 + x_2 + x_3,\, x_1 + x_2,\, x_1 + x_3).$$

What other linear configurations can we count in $A$ by knowing it is uniform?
As one of the most basic types of linear configurations, let us start by considering three-term arithmetic progressions $(x,\, x+r,\, x+2r)$.

\begin{lem} \label{3APlem}
Let $G$ be an additive group of odd order and suppose $A \subseteq G$ is Fourier $\varepsilon$-uniform.
Then there are between $(\delta^3 - \varepsilon)n^2$ and $(\delta^3 + \varepsilon)n^2$ 3-term arithmetic progressions in $A$.
\end{lem}

Before proving this lemma, let us remark that the assumption that $G$ has odd order cannot be dropped.
This is due to somewhat uninteresting divisibility issues, as can be most easily seen by considering the extreme case where $G = \F_2^n$:
for \emph{any} set $A \subset \F_2^n$ of density $0 < \delta < 1$, we see that
$$\mathbb{E}_{x, r \in \F_2^n}[A(x) A(x+r) A(x+2r)] = \mathbb{E}_{x, r \in \F_2^n}[A(x) A(x+r)] = \delta^2$$
is bounded away from the `expected' value of $\delta^3$.

In order to illustrate the use of Fourier analysis to tackle such problems, we shall give a Fourier analytic proof of Lemma \ref{3APlem}:

\begin{proof}
We will use the identity
$$\mathbb{E}_{x, r \in G} \big[f_1(x) f_2(x+r) f_3(x+2r)\big] = \sum_{\gamma \in \widehat{G}} \widehat{f_1}(\gamma) \widehat{f_2}(\gamma^{-2}) \widehat{f_3}(\gamma),$$
where $\gamma^{-2}$ is the character satisfying $\gamma^{-2}(x) = \gamma(x)^{-2}$ for all $x \in G$.
Indeed, the last sum is equal to
\begin{align*}
    \sum_{\gamma \in \widehat{G}} \mathbb{E}_{x \in G} \big[f_1(x) \overline{\gamma(x)}\big] \, \mathbb{E}_{y \in G} &\big[f_2(y) \overline{\gamma^{-2}(y)}\big] \, \mathbb{E}_{z \in G} \big[f_3(z) \overline{\gamma(z)}\big] \\
    &= \sum_{\gamma \in \widehat{G}} \mathbb{E}_{x, y, z \in G} \big[f_1(x) f_2(y) f_3(z) \gamma(-x+2y-z)\big] \\
    &= \mathbb{E}_{x, y, z \in G} \big[ f_1(x) f_2(y) f_3(z) \cdot |G| \mathbbm{1}_{\{ x+z = 2y\}} \big] \\
    &= \mathbb{E}_{x, r \in G} \big[f_1(x) f_2(x+r) f_3(x+2r)\big],
\end{align*}
where we used the orthogonality relations of characters for the second equality.

As $\widehat{A}(\mathbf{1}) = \delta$, we conclude that
$$\mathbb{E}_{x, r \in G} \big[A(x) A(x+r) A(x+2r)\big] = \delta^3 + \sum_{\gamma \in \widehat{G} \setminus \{\mathbf{1}\}} \widehat{A}(\gamma)^2 \widehat{A}(\gamma^{-2}).$$
We then bound the absolute value of the last sum by
\begin{align*}
    \Bigg| \sum_{\gamma \in \widehat{G} \setminus \{\mathbf{1}\}} \widehat{A}(\gamma)^2 \widehat{A}(\gamma^{-2}) \Bigg|
    &\leq \bigg( \max_{\gamma \in \widehat{G} \setminus \{\mathbf{1}\}} |\widehat{A}(\gamma)| \bigg) \sum_{\gamma \in \widehat{G} \setminus \{\mathbf{1}\}} |\widehat{A}(\gamma)| \cdot |\widehat{A}(\gamma^{-2})| \\
    &\leq \varepsilon \Bigg( \sum_{\gamma \in \widehat{G}} |\widehat{A}(\gamma)|^2 \Bigg)^{1/2} \Bigg( \sum_{\gamma \in \widehat{G}} |\widehat{A}(\gamma^{-2})|^2 \Bigg)^{1/2},
\end{align*}
where for the last inequality we used Cauchy-Schwarz and the fact that $A$ is Fourier $\varepsilon$-uniform.

We will next show that $\{\gamma^{-2}: \gamma \in \widehat{G}\} = \widehat{G}$.
Note that this will conclude the proof, since it implies that the right-hand side of the last inequality is equal to $\varepsilon \|\widehat{A}\|_{\ell^2(\widehat{G})}^2 = \varepsilon \|A\|_{L^2(G)}^2 \leq \varepsilon$.
Since clearly $\{\gamma^{-2}: \gamma \in \widehat{G}\} \subseteq \widehat{G}$, it suffices to show that $\gamma^{-2} \neq \chi^{-2}$ whenever $\gamma$ and $\chi$ are distinct characters.

But if $\gamma \neq \chi$ and $\gamma^{-2} = \chi^{-2}$, then $\chi \gamma^{-1}$ is a nontrivial character satisfying $(\chi \gamma^{-1})^2 = \mathbf{1}$.
This implies that the order of $\chi \gamma^{-1}$ is $2$, which is impossible since it must divide $|\widehat{G}| = |G|$ which is odd.
This contradiction finishes the proof.
\end{proof}


In general, we can count the number of pre-images of $A$ by any system of linear forms which can be expressed as a subgraph of a `Cayley-like' graph of $A$
(in groups where these linear forms incur in no divisibility issues).
This follows from the equivalence between uniformity of $A$ and quasirandomness of its Cayley graph $\Gamma_A$ (or other similar graphs), as the following examples illustrate:

\begin{ex}
Additive quadruples in $A$ correspond to 4-cycles in the Cayley graph $\Gamma_A$.
This has been shown in the proof of Theorem \ref{linearsets}, where we saw that each additive quadruple $(x,\, x + h_1 + h_2,\, x + h_1,\, x + h_2) \in A^4$ is in one-to-$|G|$ correspondence with quadruples $(a+b,\, c+d,\, a+c,\, b+d) \in A^4$ representing 4-cycles in $\Gamma_A$.
\end{ex}

\begin{ex} \label{3APex}
Consider the tripartite graph $\Gamma'$ formed by three copies $X, Y, Z$ of the group $G$, and with the three edge classes between these copies defined by the relations $-2x - y \in A$, $-x + z \in A$ and $y + 2z \in A$ (for $x \in X$, $y \in Y$ and $z \in Z$).
It is easy to check that each\footnote{We shall sometimes write `$k$-AP' as a short for `$k$-term arithmetic progression'.} 3-AP in $A$ gives rise to $|G|$ triangles in $\Gamma'$, and conversely every triangle in $\Gamma'$ represents one 3-AP in $A$.
Moreover, if $G$ has odd order then $A$ is uniform if and only if the tripartite graph $\Gamma'$ is quasirandom (which provides a `graph theoretical' proof of Lemma \ref{3APlem} by using the counting lemma from Section \ref{PartiteGraphs}).
\end{ex}

\begin{ex} \label{SchurEx}
Let $\Gamma''$ be the tripartite graph formed by three copies $X, Y, Z$ of the group $G$, with the three edge classes between these copies defined by the relations $x - y \in A$, $z - x \in A$ and $z - y \in A$ (for $x \in X$, $y \in Y$ and $z \in Z$).
Then Schur triples\footnote{A Schur triple in an additive group $G$ is a triple of the form $(x, y, x + y)$ for some $x, y \in G$.} in $A$ correspond to triangles in $\Gamma''$, with each triple $(x,\, y,\, x + y)$ contained in $A$ being associated to exactly $|G|$ triangles in $\Gamma''$.
Again, $A$ will be a uniform set if and only if the tripartite graph $\Gamma''$ is quasirandom (this time without needing assumptions on the order of $G$).
\end{ex}

These examples might make it seem like Fourier uniformity is a sufficient condition to estimate the number of any linear configuration inside a given set, but this is true only for very `simple' types of linear patterns.
Indeed, as Example \ref{quad_ex} below shows, Fourier uniformity \emph{does not} suffice to estimate the number of 4-term arithmetic progressions.

This example is essentially due to Gowers \cite{NewProofSzemeredi}, and neatly illustrates some issues that lie at the heart of using quasirandomness to count linear configurations in additive sets.
Due to its importance in the theory we will analyze it in detail, following the approach given by Granville \cite{IntroductionGranville}.

\begin{ex} \label{quad_ex}
Let $N$ be a large prime number.
For any $\delta \in (0, 1)$ define the set $A_{\delta} := \{x \in \Z_N: \, \|x^2/N\|_{\R/\Z} \leq \delta/2\} \subset \Z_N$, where we identify $\Z_N$ with $[N]$ in the obvious way and denote by $\|x\|_{\R/\Z}$ the distance from $x \in \R$ to the nearest integer;
note that $\Z_N$ is a field, so multiplication is well-defined and also all nonzero elements are invertible.
In order to lighten the notation, assume all expectations in this example are over $\Z_N$ and let us denote $\omega := e^{2 \pi i/N}$, so that the group characters are given by $\chi_r(x) = \omega^{rx}$ for $r \in \Z_N$.

We will first compute the Fourier coefficients $\widehat{A_{\delta}}(\chi_r)$ of the set $A_{\delta}$.
Denoting $M_{\delta} := \lfloor \delta N/2 \rfloor$, we can write its indicator function as
$$A_{\delta}(x) = \sum_{m=-M_{\delta}}^{M_{\delta}} \mathbbm{1}_{\{x^2 = m\}} = \sum_{m=-M_{\delta}}^{M_{\delta}} \mathbb{E}_s \big[\omega^{s(x^2-m)}\big],$$
where we used the fact that $\mathbb{E}_s[\omega^{sy}]$ is $1$ if $y = 0$ and $0$ otherwise.
Thus
\begin{align*}
    \widehat{A_{\delta}}(\chi_r) &= \mathbb{E}_x \big[A_{\delta}(x) \omega^{-rx}\big] \\
    &= \mathbb{E}_x \Bigg[\Bigg(\sum_{m=-M_{\delta}}^{M_{\delta}} \mathbb{E}_s \big[\omega^{s(x^2-m)}\big] \Bigg) \omega^{-rx}\Bigg] \\
    &= \mathbb{E}_s \Bigg[ \Bigg(\sum_{m=-M_{\delta}}^{M_{\delta}} \omega^{-sm}\Bigg) \mathbb{E}_x \big[\omega^{sx^2 - rx}\big] \Bigg].
\end{align*}
The expression inside the expectation in the last line has two terms which we will analyze separately.

Let us first take a look at the sum inside the parenthesis.
When $s = 0$ it is clearly equal to $2M_{\delta}+1$, and when $s \neq 0$ we obtain
$$\left|\sum_{m=-M_{\delta}}^{M_{\delta}} \omega^{-sm}\right| = \left|\frac{\omega^{-M_{\delta}s} - \omega^{(M_{\delta}+1)s}}{1 - \omega^{s}}\right| \leq \frac{2}{|1 - \omega^{s}|}.$$
We now use the bound $|1 - e^{i \theta}| \geq 2|\theta|/\pi$, which is valid for $-\pi \leq \theta \leq \pi$.
Letting $\theta = 2\pi s/N$ we see that $|1 - \omega^s| \geq 4|s|/N$ whenever $1 \leq |s| \leq N/2$, so
$$\left|\sum_{m=-M_{\delta}}^{M_{\delta}} \omega^{-sm}\right| \leq \frac{N}{2|s|}$$
in this case (which comprises all of $\Z_N \setminus \{0\}$ if we substitute $s$ by $s-N$ when $s > N/2$).

Now let us consider the term $\mathbb{E}_x \big[\omega^{sx^2 - rx}\big]$.
When $s = 0$ it is equal to $\mathbbm{1}_{r=0}$, and when $s \neq 0$ it is a Gauss sum that can be computed very simply as follows.
Denote by $f_s$ the function on $\Z_N$ defined by $f_s(x) = \omega^{sx^2}$, so that $\widehat{f_s}(\chi_r) = \mathbb{E}_x \big[\omega^{sx^2 - rx}\big]$.
Then
$$\mathbb{E}_x \big[\omega^{sx^2 - rx}\big] = \mathbb{E}_x \big[\omega^{s(x-r/2s)^2}\omega^{-r^2/4s}\big] = \omega^{-r^2/4s} \mathbb{E}_x \big[\omega^{sx^2}\big],$$
which implies that $|\widehat{f_s}(\chi_r)| = \big|\mathbb{E}_x \big[\omega^{sx^2}\big] \big|$ is the same for all $r \in \Z_N$.
Since $\|f_s\|_{L^2} = 1$, by Parseval's identity we obtain $\big|\mathbb{E}_x \big[\omega^{sx^2 - rx}\big] \big| = N^{-1/2}$ for all $r$.

Putting everything together, we get for $r \neq 0$
\begin{align*}
    |\widehat{A_{\delta}}(\chi_r)| &\leq \frac{1}{N}\sum_{s=1}^{N-1} \Bigg|\sum_{m=-M_{\delta}}^{M_{\delta}} \omega^{-sm}\Bigg| \big|\mathbb{E}_x \big[\omega^{sx^2 - rx}\big] \big| \\
    &\leq \frac{2}{N}\sum_{s=1}^{(N-1)/2} \frac{N}{2s} \cdot N^{-1/2} \\
    &\leq \frac{\log N}{\sqrt{N}},
\end{align*}
and similarly for $r=0$ we have
\begin{align*}
    \widehat{A_{\delta}}(\mathbf{1}) &= \frac{1}{N}(2M_{\delta}+1) \pm \frac{1}{N}\sum_{s=1}^{N-1} \Bigg|\sum_{m=-M_{\delta}}^{M_{\delta}} \omega^{-sm}\Bigg| \big|\mathbb{E}_x \big[\omega^{sx^2}\big] \big| \\
    &= \frac{1}{N} \bigg(2 \bigg\lfloor \frac{\delta N}{2} \bigg\rfloor + 1 \bigg) \pm \frac{\log N}{\sqrt{N}} \\
    &= \delta \pm \frac{2 \log N}{\sqrt{N}}.
\end{align*}
This shows that the density of $A_{\delta}$ (which is equal to $\widehat{A_{\delta}}(\mathbf{1})$) is very close to $\delta$, and all its non-trivial Fourier coefficients are extremely small in absolute value.
This set is then very uniform, and by Lemma \ref{3APlem} it contains $(1 + o(1))\delta^3 N^2$ 3-term arithmetic progressions.

Let us now consider 4-term arithmetic progressions.
From the easily verified identity
$$(a + 3d)^2 = 3(a + 2d)^2 - 3(a + d)^2 + a^2$$
we see that whenever $a,\, a + d,\, a + 2d \in A_{\delta}$ we have $\|(a + 3d)^2/N\|_{\R/\Z} \leq 7\delta/2$, and so $a,\, a + d,\, a + 2d,\, a + 3d \in A_{7\delta}$.
But then $A_{7\delta}$ is a very uniform set of density $7\delta + o(1)$ which contains at least $(1 + o(1))\delta^3 N^2$ 4-term arithmetic progressions, which is far more than the expected $(7\delta)^4N^2$ if $\delta$ is small enough.
\end{ex}

The moral that one should take from this last example is the following:
while quadratically structured sets may have negligible correlation with the linear patterns measured by Fourier analysis, the relationship between the squares of the individual terms of a 4-term arithmetic progression makes it possible for this quadratic structure to influence the count of 4-APs.

It is a deep and very interesting fact that both these `quadratic dependencies' are in a certain sense \emph{necessary} for what is written in the last paragraph.
Indeed, it turns out (at least when $G$ is $\Z_N$ for $N$ prime or $\F_p^n$) that the $U^2$ norm does \emph{not} control the count of a given linear configuration if and only if the squares of its terms are linearly dependent \cite{TrueComplexity, QuadraticUniformityFpn, QuadraticUniformityZn, HigherDegreeUniformityFpn, ArithmeticRegularity}.
Moreover, a uniform set does \emph{not} have the `correct' count of 4-APs only if it exhibits some (well-defined) kind of generalized quadratic behaviour \cite{NewProofLengthFour, InverseTheoremU3}.

We will have more to say about this in the next subsection and in Section \ref{LinearConfigSection}.

\subsection{Higher-degree uniformity and the Gowers norms} \label{HigherDegUniformity}

The last example has shown us the need to consider stronger notions of quasirandomness in order to control the count of more complicated linear configurations.
In particular, these stronger notions should also be able to detect quadratic (or higher degree) behaviour.
This is the main reason for the $U^2$ norm to be a more suitable measure of quasirandomness than the more natural Fourier analytic notion of uniformity, as it is much better suited for such generalizations.

Indeed, using the combinatorial interpretation of the expression
$$\left\| f \right\|_{U^2(G)}^4 = \mathbb{E}_{x, h_1, h_2 \in G} \big[f(x) f(x + h_1) f(x + h_2) f(x + h_1 + h_2)\big]$$
as a weighted count of `parallelograms' $(x, x+h_1, x+h_2, x+h_1+h_2)$ in $G$, one might be led to consider a similar weighted count of \emph{three-dimensional parallelepipeds} in $G$.
This naturally leads to the following definition, due to Gowers \cite{NewProofLengthFour}:

\begin{definition}
For a real-valued function $f: G \rightarrow \R$, its $U^3$ norm is given by the equation
\begin{align*}
    \left\| f \right\|_{U^3(G)}^8 &= \mathbb{E}_{x, h_1, h_2, h_3 \in G} \big[ f(x) f(x + h_1) f(x + h_2) f(x + h_1 + h_2)\\
    &\hspace{1cm} \times f(x + h_3) f(x + h_1 + h_3) f(x + h_2 + h_3) f(x + h_1 + h_2 + h_3) \big].
\end{align*}
\end{definition}

One can show that this expression indeed provides a norm on $\R^G$, which is stronger than the $U^2$ norm in the sense that $\left\| f \right\|_{U^2(G)} \leq \left\| f \right\|_{U^3(G)}$ for any function $f$.
Moreover, it is not hard to prove that a random $\{-1, 1\}$-valued function $f$ on a large additive group $G$ will have very small $U^3(G)$ norm with high probability.

In analogy with the $U^2$ norm, one might then think of the $U^3$ norm as a measure of quasirandomness.
We shall say that a function $f: G \rightarrow \R$ is \emph{quadratically $\varepsilon$-uniform} if $\|f\|_{U^3(G)} \leq \varepsilon$, and that a set $A \subset G$ is quadratically $\varepsilon$-uniform if its balanced function $A - \delta$ is.
Since the $U^3$ norm is stronger than the $U^2$ norm, being quadratically uniform is a stronger notion than being (linearly) uniform.

In order to make clearer the connection of the $U^3$ norm with quadratic behaviour,
let us first make more explicit the connection of the $U^2$ norm with \emph{linear} behaviour.
This can be done by writing a character $\chi \in \widehat{G}$ as $e^{2\pi i \phi(x)}$ for some \emph{linear phase function} $\phi: G \rightarrow \R/\Z$, that is, a map satisfying the linearity
property\footnote{This is of course the same as a group homomorphism from $G$ to $\R/\Z$, but here we wish to draw attention to its `linearity'.}
$\phi(x + y) = \phi(x) + \phi(y)$.
We conclude from the equivalence of linear uniformity and Fourier uniformity that a bounded function $f: G \rightarrow [-1, 1]$ has non-negligible $U^2$ norm if and only if it correlates with $e^{2\pi i \phi(x)}$ for some linear phase function $\phi$.

Similarly, we will now see that correlation with a quadratic phase implies large $U^3$ norm.
Due to the lack of multiplicative structure on general additive groups, the definition of a quadratic phase function is a bit more indirect and proceeds by considering \emph{discrete derivatives}:
given $u \in G$, we define the difference operator $\nabla_u$ applied to a phase function $\phi: G \rightarrow \R/\Z$ as $\nabla_u \phi(x) := \phi(x + u) - \phi(x)$.
We then say that $\phi$ is a \emph{quadratic phase function} if its third (discrete) derivative vanishes on $G$, i.e. if $\nabla_u \nabla_v \nabla_w \phi \equiv 0$
for all $u, v, w \in G$.
Note that linear phase functions satisfy $\nabla_u \nabla_v \phi \equiv 0$, and conversely any function $\phi: G \rightarrow \R/\Z$ whose second derivative vanishes in this sense can be written as a linear phase function plus a constant.
Moreover, in cyclic groups $\Z_N$ (where there \emph{is} a multiplicative structure) our definition of quadratic phase functions coincides with that of usual quadratic polynomials $x \mapsto (ax^2 + bx + c)/N$ for some $a, b, c \in \Z_N$
(where the map $x \mapsto x/N$ from $\Z_N$ to $\R/\Z$ is defined in the obvious manner).

A simple application of the (complex-valued) \emph{Gowers-Cauchy-Schwarz inequality} for the $U^3$
norm\footnote{This is a generalization of our Lemma \ref{U2GCS}, and follows easily from Lemma \ref{GowersCS} given next section and the definition of the $U^3$ norm for complex-valued functions
(which is obtained by taking complex conjugates of the terms $f(x+h_1)$, $f(x+h_2)$, $f(x+h_3)$ and $f(x+h_1+h_2+h_3)$ in our real-valued definition).}
implies that $\|f\|_{U^3(G)} \geq |\mathbb{E}_{x\in G}[f(x) e^{2\pi i \phi(x)}]|$ holds
whenever $\phi: G \rightarrow \R/\Z$ is a quadratic phase function.
The $U^3$ norm is thus able to detect `quadratic behaviour' of a set/function just like the $U^2$ norm is able to detect their `linear behaviour' measured by the Fourier transform.

As an example, one can show that the set $A_{\delta} = \{x \in \Z_N: \, \|x^2/N\|_{\R/\Z} \leq \delta/2\}$ considered in Example \ref{quad_ex} is \emph{not} quadratically uniform (so $\|A_{\delta} - \delta\|_{U^3(\Z_N)} > c(\delta)$ for some constant $c(\delta) > 0$ independent of $N$), even though it is linearly $o(1)$-uniform.
This greater strength is important since it allows us to count how many 4-term arithmetic progressions $(x, x+r, x+2r, x+3r)$ are contained in a quadratically uniform set $A$:
we have that
$$\mathbb{E}_{x, r \in G} \big[A(x) A(x+r) A(x+2r) A(x+3r)\big] = \delta^4 \pm 4 \|A - \delta\|_{U^3(G)}.$$
This result was first obtained by Gowers \cite{NewProofLengthFour} using repeated applications of Cauchy-Schwarz, and will be proven (in a more general form) in Section \ref{LinearConfigSection}.

In general, for every integer $k \geq 2$ one can define the \emph{Gowers uniformity norm of degree $k$} by the equation
$$\| f \|_{U^k(G)}^{2^k} = \mathbb{E}_{x, h_1, \dots, h_k \in G} \Bigg[ \prod_{\omega \in \{0, 1\}^k} f \bigg( x + \sum_{i=1}^k \omega_i h_i \bigg) \Bigg].$$
These norms were first introduced and studied by Gowers \cite{NewProofLengthFour, NewProofSzemeredi}, with the purpose of providing a new proof (with far better bounds) of Szemer\'edi's theorem on arithmetic progressions \cite{SzemerediTheorem}.

As in the cases where $k=2$ or $3$, the $U^k$ norm can be seen as a weighted count of $k$-dimensional parallelepipeds, and it is able to detect behaviour of degree up to $k-1$ of the considered function.
Moreover, for any fixed $k \geq 2$ a random function $f: G \rightarrow \{-1, 1\}$ will have negligible $U^k(G)$ norm with high probability (assuming $|G|$ is very large).

The Gowers uniformity norms also form a hierarchy where the $U^{k+1}$ norm is stronger than the $U^k$ norm for each $k \geq 2$.
If we define a function $f: G \rightarrow \R$ to be uniform of degree $k$ if it has small $U^{k+1}$ norm, we then obtain an infinite hierarchy of increasingly stronger notions of quasirandomness.

Their significance in additive combinatorics stems from the fact that uniformity of degree $k$ is sufficient to control the count of $(k + 2)$-term arithmetic progressions, as well as several other linear configurations said to have \emph{complexity at most $k$}.
Moreover, \emph{every} `non-degenerate' system of linear forms can be controlled by some uniformity norm $U^k$.

These facts will be proven in Section \ref{Cayley} by making use of the theory of quasirandomness in the hypergraph setting, which is the subject of our next section.

\section{Quasirandomness in hypergraphs} \label{hypergraphs}

We now turn our attention to \emph{hypergraphs}, which are the natural generalization of graphs where edges can contain more than two vertices.
They may also be seen as representing a higher-order relation between elements of a given set, and in this sense are arguably the `purest' form of higher-order objects.

In order to fix notation, let us formally define the notion of a (uniform) hypergraph.
In here and for the rest of this paper we will denote by $\binom{X}{k}$ the collection of all $k$-element subsets of a given set $X$.

\begin{definition}
Given a finite set $V$ and $k \geq 2$, a \emph{$k$-uniform hypergraph} (or \emph{$k$-graph}) on $V$ is defined to be any subset $H \subseteq \binom{V}{k}$.
We call $V$ the \emph{vertex set} of the hypergraph $H$, and denote its cardinality by $v(H)$.
The elements of $H$ are called \emph{edges}, and its edge density is
defined\footnote{Our definition is made so that the edge density of $H$ coincides with the average $\mathbb{E}_{x_1, \dots, x_k \in V}[H(x_1, \dots, x_k)]$.
An alternative (and perhaps more natural) definition for the density would be $|H|/\binom{v(H)}{k}$;
the relative difference between these two quantities is negligible when $|V|$ is very large, and thus essentially irrelevant for our purposes.}
as $k! |H|/v(H)^k$.
\end{definition}

As in the case of graphs, quasirandom hypergraphs are those whose edge distribution resembles the one of a truly random hypergraph of the same edge density.
For this intuition to be made precise we should first specify the model of random hypergraph to be mimicked, and also introduce a quantitative measure for this similarity;
this is what we do next.

\subsection{Motivation: the case of 3-uniform hypergraphs} \label{3graphs}

In order to arrive at natural definitions for these concepts, we shall first quickly review the case of graphs (which are 2-uniform hypergraphs):
\begin{itemize}
    \item The model of random graph is given by $G(n, p)$, where there are $n$ vertices and each pair of vertices has probability $p$ of being an edge independently.
    \item For a two-variable function $f: V \times V \rightarrow \R$ we define the \emph{cut norm} by
    $$\| f \|_{\square} = \max_{A, B \subseteq V} \big| \mathbb{E}_{x, y \in V} \big[ f(x, y) A(x) B(y) \big] \big|.$$
    \item A graph $G$ with edge density $\delta$ is \emph{$\varepsilon$-quasirandom} if $\| G - \delta \|_{\square} \leq \varepsilon$, meaning its edges are uniformly distributed along all cuts (up to an $\varepsilon$ error).
    \item If $G$ is quasirandom, then it contains about $n^{v(F)} \delta^{|F|}$ copies of any given graph $F$ as a subgraph.
\end{itemize}

Let us then try to generalize these notions to higher hypergraphs, concentrating on 3-uniform hypergraphs for simplicity:
\begin{itemize}
    \item The simplest generalization of $G(n, p)$ would be the random 3-graph on $n$ vertices, where each \emph{triple} of vertices has probability $p$ of being an edge independently.
    \item For a three-variable function $f: V \times V \times V \rightarrow \R$, define the norm
    $$\|f\|_{\square^3_1} = \max_{A, B, C \subseteq V} \big| \mathbb{E}_{x, y, z \in V} \big[f(x, y, z) A(x) B(y) C(z)\big] \big|.$$
    \item Let us (for now) say that a 3-uniform hypergraph $H$ with edge density $\delta$ is \emph{$\varepsilon$-quasirandom} if $\| H - \delta \|_{\square^3_1} \leq \varepsilon$, meaning its edges are uniformly distributed along all 3-way vertex cuts.
    \item One can easily show that the random hypergraph defined is very quasirandom w.h.p., and also that it contains about $n^{v(F)} \delta^{|F|}$ copies of any given 3-graph $F$.
\end{itemize}

Up to now it seems that everything went smoothly, and the two notions generalized rather easily.
However, by considering slightly different ways of choosing random hypergraphs, we quickly run into some issues.

For instance, another natural way of choosing a random $3$-uniform hypergraph is by making random choices at the \emph{second level} (that is, for pairs of vertices) instead of the third level (triples of vertices).
This leads us to the following example:

\begin{ex} \label{level2ex}
Let $H$ be a random 3-uniform hypergraph on $n$ vertices chosen in the following way:
first we pick a random graph $G$ according to $G(n, 1/2)$, and then let $H$ be the hypergraph corresponding to the triangles in $G$.

This random hypergraph $H$ will indeed be very quasirandom by our earlier definition, but the counting lemma does \emph{not} hold!
Indeed, let $F$ be the 3-uniform hypergraph on four vertices with two edges.
Then the number of copies of $F$ we would have expected to find in $H$ is about $n^4/64$, while its true number is about $n^4/32$.
\end{ex}

A slightly more complicated (but much more surprising) example of a similar nature was given by R\"odl \cite{RodlExample}, which we reproduce below:

\begin{ex} \label{Rodlex}
Choose a random orientation of the edges of the complete graph $K_n$ on $n$ vertices, each choice being uniform and independent from all others.
This will create a random directed graph $T_n$ on $n$ vertices (known as a \emph{tournament}), and let $H$ be the 3-uniform hypergraph whose edges are the the directed triangles in $T_n$ (i.e. $H = \{\{u, v, w\}:\, \overrightarrow{uv},\, \overrightarrow{vw},\, \overrightarrow{wu} \in T_n\}$).

One can easily show that this hypergraph $H$ will be $o(1)$-quasirandom and have edge density $1/4 + o(1)$ with high probability, but by construction it cannot contain any tetrahedron $K^{(3)}_4 := \binom{[4]}{3}$ at all!
\end{ex}

It turns out that both of these examples have the same issue at heart:
in both cases the hypergraph we wish to count copies of has edges intersecting at \emph{two vertices}, while the cut norm used only measures correlation with functions of \emph{one vertex} at a time.

We will see later that for counting \emph{linear} hypergraphs (i.e. those where any two edges share at most one vertex) such a discrepancy would not happen, and this `weak' cut norm is enough to control the number of copies of any linear hypergraph.
However, in order to control the number of copies of \emph{all} 3-graphs, one has to consider the following stronger norm to measure quasirandomness:
$$\|f\|_{\square^3_2} = \max_{A, B, C \subseteq V \times V} \big| \mathbb{E}_{x, y, z \in V} \big[f(x, y, z) A(x, y) B(x, z) C(y, z)\big] \big|.$$

The need to consider various notions of both random and quasirandom hypergraphs has then become clear.
Let us now define them formally in the general case of $k$-uniform hypergraphs for any $k \geq 3$.

\subsection{Randomness and quasirandomness of every order}

In general, to choose a random $k$-graph $H$ one can make random choices at any level $2 \leq j \leq k$, or indeed at any subset of them:

\begin{namedthm}{Randomness at level $j$}
Pick each $j$-set $f \in \binom{V}{j}$ at random with probability $p_j$, and let $e \in \binom{V}{k}$ be an edge of $H$ iff all its $j$-subsets $f \in \binom{e}{j}$ were chosen.
\end{namedthm}

The general model of random hypergraphs that we will consider here is then `generated' by employing randomness at any subset of the levels $2 \leq j \leq k$.
We shall illustrate this model by providing a recipe for drawing random 3-uniform hypergraphs:

\begin{ex}
To choose a random 3-graph $H$ on the vertex set $V$, pick:
\begin{itemize}
    \item A random subset $G^{(2)} \subseteq \binom{V}{2}$ of all pairs of vertices, each being in $G^{(2)}$ independently with probability $p_2$;
    \item A random subset $G^{(3)} \subseteq \binom{V}{3}$ of all triples of vertices, each being in $G^{(3)}$ independently with probability $p_3$.
\end{itemize}
Then $\{x, y, z\} \in \binom{V}{3}$ is an edge of $H$ iff $\{x, y, z\} \in G^{(3)}$ and each pair $\{x, y\}$, $\{x, z\}$, $\{y, z\}$ is in $G^{(2)}$.
This event has probability $p_2^3 p_3$, so this is the (expected) edge density of the hypergraph, but now the presence of two given edges are no longer independent events if they share a pair of vertices.
\end{ex}



For each level of randomness in the choice of a random $k$-graph there will be an associated notion of quasirandomness.
Intuitively, this notion of quasirandomness associated to a given level $j$ of random choices is related to \emph{lack of correlation} with structures of any order $d < j$.

In order to define this concept more precisely, we will need the following general piece of notation.
For a finite set $A$, any tuple $\mathbf{x} = (x_i)_{i \in A}$ indexed by the elements of $A$ and any subset $B \subseteq A$, we denote by $\mathbf{x}_B := (x_j)_{j \in B}$ the projection of $\mathbf{x}$ onto its $B$-coordinates.

\begin{definition}
Given a function $f: V^{[k]} \rightarrow \R$ and an integer $1 \leq d \leq k-1$, we define the \emph{$(k, d)$-cut norm} of $f$ by
$$\|f\|_{\square^k_d} := \max_{S_B \subseteq V^{B}\; \forall B \in \binom{[k]}{d}} \Bigg| \mathbb{E}_{\mathbf{x} \in V^{[k]}} \Bigg[ f(\mathbf{x}) \prod_{B \in \binom{[k]}{d}}{S_B(\mathbf{x}_{B})} \Bigg] \Bigg|,$$
where the maximum is over all collections of sets $(S_B)_{B \in \binom{[k]}{d}}$ where each $S_B$ is a subset of $V^B$.
We say that the function $f$ is \emph{$\varepsilon$-quasirandom of order $d$} if $\|f\|_{\square^k_d} \leq \varepsilon$, and that a $k$-graph $H$ of edge density $\delta$ is $\varepsilon$-quasirandom of order $d$ if $\|H - \delta\|_{\square^k_d} \leq \varepsilon$.
\end{definition}

\begin{remark}
Some authors prefer the more combinatorial notion of \emph{clique discrepancy}, which is used for instance in \cite{QuasiClasses, HypergraphQuasirandomnessRegularity, PosetQuasirandomness}.
The $d$-clique discrepancy of a $k$-uniform hypergraph $H$ is defined as
$$\frac{1}{|V|^k} \max_{G \subseteq \binom{V}{d}} \big| |H \cap \mathcal{K}_k(G)| - \delta(H) |\mathcal{K}_k(G)| \big|,$$
where $V = V(H)$ is the vertex set of $H$, $\delta(H)$ is its edge density and $\mathcal{K}_k(G)$ is set of $k$-cliques of the $d$-uniform hypergraph $G$ (i.e. the collection of $k$-sets of vertices whose $d$-subsets are all edges of $G$).
This notion is formally very similar to our measure $\|H - \delta(H)\|_{\square^k_d}$ of quasirandomness of order $d$
(once one unpacks all the notation).
We will not use the notion of clique discrepancy here, but in the interest of being through we remark that having small $d$-clique discrepancy is polynomially equivalent to being quasirandom of order $d$.\footnote{This is surprisingly tricky to prove, but it follows by combining the arguments given in the proof of Lemma 4.8 in \cite{SigmaAlgebrasHypergraphs} to those of Proposition 2.9 in \cite{HypergraphQuasirandomnessRegularity}.
The special case where $d=1$ is presented in details in the paper \cite{WeakQuasirandomness}.}
\end{remark}

Note that, as in the case of graphs, a simple argument of multi-linearity implies that the definition given for the $(k, d)$-cut norm is exactly equivalent to
$$\|f\|_{\square^k_d} = \max_{u_{B}: V^{B} \rightarrow [0, 1]\; \forall B \in \binom{[k]}{d}} \Bigg| \mathbb{E}_{\mathbf{x} \in V^{[k]}} \Bigg[ f(\mathbf{x}) \prod_{B \in \binom{[k]}{d}}{u_{B}(\mathbf{x}_{B})} \Bigg] \Bigg|,$$
where the maximum is now taken over all collections of functions $u_B: V^B \rightarrow [0, 1]$ instead of sets $S_B \subseteq V^B$.
This observation will prove useful later on.

The next example shows that random hypergraphs chosen according to our model will be $o(1)$-quasirandom of the suitable order with high probability.
Note that the `suitable order' of quasirandomness associated to a level $j$ of randomness is $j-1$ instead of $j$.

\begin{ex}[Random hypergraphs]
If all levels of randomness involved in the choosing of a random hypergraph $H$ are strictly higher than $d$, then $H$ will be $o(1)$-quasirandom of order $d$ with high probability.
This can be shown by using Chernoff's inequality and union bound in much the same way as we did when proving that $G(n, p)$ is $o(1)$-quasirandom w.h.p. in Section \ref{graphs}.

Conversely, if there is a non-trivial level of randomness in the choice of $H$ which is at most equal to $d$, then $H$ will (with high probability) \emph{not} be quasirandom of order $d$.
This can be seen by taking all sets $S_B$ in the definition of the cut norm to be the collection of elements of $\binom{V}{d}$ chosen in this level of randomness.
\end{ex}

It is clear from the definition that $\|f\|_{\square^k_1} \leq \|f\|_{\square^k_2} \leq \dots \leq \|f\|_{\square^k_{k-1}}$ for any function $f$, and the previous example shows there can be no similar bound on the reverse direction which is valid uniformly on $|V|$.
We thus obtain a hierarchy of quasirandomness concepts for hypergraphs (and more generally for functions), one for each order $1 \leq d < k$;
we will see in Section \ref{Cayley} that this hierarchy is closely related to the one given by the Gowers uniformity norms for functions on additive groups.

We next consider the question of what kind of information one can obtain from these notions of quasirandomness.

\subsection{Counting subhypergraphs}

Perhaps the most important piece of information to have about a large hypergraph is the distribution of what is observed when sampling at random a bounded number of its vertices.
This distribution is characterized by the homomorphism densities of smaller hypergraphs $F$ in the large hypergraph $H$ under consideration:

\begin{definition}
Let $F$ and $H$ be two $k$-uniform hypergraphs, having vertex sets $V(F)$ and $V(H)$ respectively.
The \emph{homomorphism density} of $F$ in $H$ is the probability that a randomly chosen map $\phi: V(F) \rightarrow V(H)$ preserves edges:
\begin{align*}
    t(F, H) &= \mathbb{P}_{x_1, \dots, x_{v(F)} \in V(H)} \big( \{x_i: i \in e\} \in H \text{ for all } e \in F \big) \\
    &= \mathbb{E}_{\mathbf{x} \in V(H)^{V(F)}} \Bigg[ \prod_{e \in F} H(\mathbf{x}_e) \Bigg].
\end{align*}
\end{definition}

We note that computing the homomorphism density of a fixed hypergraph $F$ inside a large hypergraph $H$ is essentially the same as counting the (normalized) number of copies of $F$ inside $H$, up to an error of order $v(F)^2/v(H)$.

It was shown in Examples \ref{level2ex} and \ref{Rodlex} that weaker notions of quasirandomness (e.g. $\|H - \delta\|_{\square^k_1} = o(1)$) are not sufficient to control the count of all subhypergraphs.
There are, however, natural classes of hypergraphs which \emph{can} be counted inside $H$ by knowing it is quasirandom of a given order.
The simplest of these classes is that of linear hypergraphs:
Kohayakawa, Nagle, R\"odl and Schacht \cite{WeakHypergraphRegularity} showed that every $k$-uniform hypergraph $H$ which is quasirandom of order $1$ must necessarily contain approximately the `correct' number of copies of any fixed linear $k$-graph $F$.

More generally, we will next show that quasirandomness of order $d$ suffices to control the number of all \emph{$d$-linear hypergraphs} as defined below:

\begin{definition}
Let $1 \leq d < k$ be positive integers.
We say that a $k$-graph $F$ is \emph{$d$-linear} if every pair of its edges intersect in at most $d$ vertices.
We denote the set of all $d$-linear $k$-graphs by $\mathcal{L}^{(k)}_d$.
\end{definition}

Note that $1$-linear hypergraphs in this definition are the same as usual linear hypergraphs, while every $k$-uniform hypergraph is $(k-1)$-linear.
We have the following lemma:

\begin{lem}[Counting lemma for quasirandomness of order $d$] \label{count_d-linear}
For any $k$-uniform hypergraph $H$ and any number $0 \leq \delta \leq 1$, we have that
$$t(F, H) = \delta^{|F|} \pm |F| \cdot \| H - \delta \|_{\square^k_d} \hspace{5mm} \forall F \in \mathcal{L}^{(k)}_d.$$
\end{lem}

\begin{proof}
Denoting $F = \{e_1, \dots, e_{|F|}\}$, we can write as a telescoping sum
\begin{align*}
    \big|t(F, H) - \delta^{|F|}\big| &=
    \Bigg| \mathbb{E}_{\mathbf{x} \in V(H)^{V(F)}} \Bigg[ \prod_{e \in F} H(\mathbf{x}_e) - \delta^{|F|} \Bigg] \Bigg| \\
    &= \Bigg| \mathbb{E}_{\mathbf{x} \in V(H)^{V(F)}} \Bigg[ \sum_{i = 1}^{|F|}{\delta^{i - 1} \big( H(\mathbf{x}_{e_i}) - \delta \big) \prod_{j = i+1}^{|F|}{H(\mathbf{x}_{e_j})}} \Bigg] \Bigg| \\
    &\leq \sum_{i = 1}^{|F|}{\delta^{i - 1} \Bigg|\mathbb{E}_{\mathbf{x} \in V(H)^{V(F)}} \Bigg[\big(H(\mathbf{x}_{e_i}) - \delta\big) \prod_{j = i+1}^{|F|}{H(\mathbf{x}_{e_j})} \Bigg]\Bigg|}.
\end{align*}
Consider the expectation inside the $i$-th term of the sum above.
If we fix all variables other than $\mathbf{x}_{e_i}$, then all the factors inside the expectation except for $\left(H(\mathbf{x}_{e_i}) - \delta\right)$ have the form $u(\mathbf{x}_{f})$, for some function $0 \leq u \leq 1$ and some set $f \subset e_i$ which is the intersection of $e_i$ with another edge $e_j$.
Since these intersections have size at most $d$, it follows that this expectation can be bounded by
$\| H - \delta \|_{\square^k_d}$.
Summing over all $|F|$ terms we obtain the result.
\end{proof}

\begin{remark}
This proof can be straightforwardly modified in order to show that
$$|t(F, H_1) - t(F, H_2)| \leq |F| \cdot \| H_1 - H_2 \|_{\square^k_d}$$
for every pair of $k$-graphs $H_1, H_2$ and every $d$-linear $k$-graph $F$;
thus hypergraphs which are close in $(k, d)$-cut norm have similar counts of every $d$-linear $k$-graph.
\end{remark}

One can easily generalize our Example \ref{level2ex} in order to show that the assumption of $d$-linearity is \emph{necessary} for the counting lemma of any given order $d < k - 1$
(for $d = k-1$ this assumption is trivial).
This is done in the following example:

\begin{ex} \label{simplekgraph}
For a given $1 \leq d \leq k-2$, let $F$ be the (unique) connected $k$-graph on $2k-d$ vertices containing two edges.
Note that $F$ is $d$-linear, but not $(d-1)$-linear.

Let $V$ be a (large) set of vertices and choose a random subset $G^{(d)} \subseteq \binom{V}{d}$ of all $d$-sets of vertices, each being in $G^{(d)}$ independently with probability $1/2$.
Finally, let $H$ be the random $k$-graph on $V$ where $\{x_1, \dots, x_k\} \in \binom{V}{k}$ is an edge of $H$ iff all its $d$-element subsets are in $G^{(d)}$.

Then with high probability $H$ will be $o(1)$-quasirandom of order $d-1$ and have density $\delta = 2^{-\binom{k}{d}} + o(1)$, but
$t(F, H) = 2^{-2\binom{k}{d} + 1} + o(1) = 2 \delta^{|F|} + o(1)$.
\end{ex}


\subsection{Equivalence theorems for each order $d$} \label{EveryOrderEquiv}

It follows from the counting lemma (Lemma \ref{count_d-linear}) that quasirandomness of order $d$ is sufficient for approximately counting all $d$-linear subhypergraphs.
It turns out that the converse implication also holds, in the sense that any (large) $k$-graph $H$ containing approximately the correct amount of each (small) $d$-linear $k$-graph is necessarily quasirandom of order $d$.

A much more surprising fact is that, as in the case of graphs, there is a hypergraph which is \emph{complete} for quasirandomness of order $d$:
it suffices for $H$ to have the `correct' number of copies of \emph{a single} $d$-linear $k$-graph (denoted $\textsc{M}^{(k)}_{d}$) in order for us to conclude that it is quasirandom of order $d$,
and thus be able to estimate the count of all other $d$-linear $k$-graphs.

Such a result was first obtained by Conlon, H\`an, Person and Schacht \cite{WeakQuasirandomness}
in the case $d = 1$.
These authors constructed a linear $k$-graph $M$ on $k 2^{k-1}$ vertices and $2^k$ edges
whose homomorphism density is at least $\delta^{2^k}$ on any $k$-graph $H$ having edge density $\delta$, and showed that if $t(M, H)$ is close to this minimum then $H$ is quasirandom of order $1$.
They also conjectured that a similar construction (to be presented below) would yield for each order $1 \leq d < k$ a $d$-linear $k$-graph which has the same role for quasirandomness of order $d$ as $M$ has for quasirandomness of order $1$.

Given a $k$-partite $k$-graph $F$ with vertex partition $X_1, \dots, X_k$ and a $d$-set of indices $I \in \binom{[k]}{d}$, we define the \emph{$I$-doubling} of $F$ to be the hypergraph $\db_I(F)$ obtained by taking two copies of $F$ and identifying the corresponding vertices in the classes $X_i$, for all $i \in I$.
More precisely, the vertex set of the $I$-doubling is
\begin{equation*}
    V(\db_I(F)) = Y_1 \cup \dots \cup Y_k \hspace{3mm} \text{where} \hspace{3mm} Y_i =
    \begin{cases}
        X_i & \text{ if } i \in I, \\
        X_i \times \{0, 1\} & \text{ if } i \notin I 
    \end{cases}
\end{equation*}
and its edge set is the collection of all $k$-sets of the form
$$\{x_i:\, i \in I\} \cup \{(x_j, a):\, j \in [k] \setminus I\},$$
where $a \in \{0, 1\}$ and $\{x_i:\, i \in [k]\}$ is an edge of $F$.

Starting with the $k$-partite hypergraph with $k$ vertices and a single edge, and then applying consecutively $\db_I$ for every $I \in \binom{[k]}{d}$ (in some arbitrary order), we get a $d$-linear $k$-graph which we denote by $\textsc{M}^{(k)}_{d}$.
One can think of this construction of $\textsc{M}^{(k)}_{d}$ as encoding the applications of
Cauchy-Schwarz\footnote{An interesting way to formalize this idea is by considering the \emph{Cauchy-Schwarz tree} (as defined in \cite{GraphNorms}) associated with the cut involutions representing which vertices are fixed by each elementary doubling operation in the construction of $\textsc{M}^{(k)}_{d}$.
We refer the reader to sections 3 and 5 of \cite{GraphNorms}.}
needed in the proof that correctly counting $d$-linear $k$-graphs implies quasirandomness of order $d$.

The conjecture that $\textsc{M}^{(k)}_{d}$ is complete for quasirandomness of order $d$ was later proven by Towsner \cite{SigmaAlgebrasHypergraphs}
(using the framework of ultraproducts and graded probability spaces),
who then obtained the main equivalence theorem for quasirandomness of any fixed order.
We reproduce a quantitative version of his result below:

\begin{thm}[Equivalence theorem for quasirandomness of order $d$] \label{Towsnerthm}
Let $1 \leq d < k$ be integers and let $H$ be a $k$-uniform hypergraph with edge density $\delta$.
Then the following properties are polynomially equivalent:
\begin{itemize}
    \item[$(i)$] $H$ is quasirandom of order $d$: \hspace{2mm}
    $\| H - \delta \|_{\square^{k}_d} \leq c_1$.
    \item[$(ii)$] $H$ correctly counts all $d$-linear hypergraphs:
    $$t(F, H) = \delta^{|F|} \pm c_2 |F| \hspace{5mm} \forall F \in \mathcal{L}^{(k)}_{d}.$$
    \item[$(iii)$] $H$ has few copies of $M = \textsc{M}^{(k)}_{d}$: \hspace{2mm}
    $t(M, H) \leq \delta^{|M|} + c_3$.
    \item[$(iv)$] $H$ has small deviation with respect to $M = \textsc{M}^{(k)}_{d}$:
    $$\mathbb{E}_{\mathbf{x} \in V(H)^{V(M)}} \Bigg[ \prod_{e \in M} \big( H(\mathbf{x}_{e}) - \delta \big) \Bigg] \leq c_4.$$
\end{itemize}
\end{thm}

We will not prove Theorem \ref{Towsnerthm} here, as it is somewhat technical
(but see the proofs of Theorem \ref{octequiv} and Theorem \ref{d-deviationThm}, which are conceptually similar).
The interested reader is referred to Towsner's original paper \cite{SigmaAlgebrasHypergraphs} for a qualitative version of this result,
or to \cite{QuasirandomnessHypergraphs}
for a combinatorial proof of the main equivalences $(i) \Leftrightarrow (ii) \Leftrightarrow (iv)$ with polynomial bounds.
The paper \cite{CayleyTypeHypergraphs} also outlines a proof of all equivalences using the same methods we use in this survey, and obtains polynomial bounds for them all.

It is important to note that Towsner's result actually applies to a much larger class of quasirandomness notions than the ones considered here.
His paper extended the work of Lenz and Mubayi \cite{PosetQuasirandomness}, who also studied several distinct notions of quasirandomness for hypergraphs and determined the poset of implications between the quasirandom properties they considered.

We also note that in the case $d = 1$ (i.e. for quasirandomness of order $1$) Lenz and Mubayi \cite{HyperEigRegular, HyperEigNonRegular} obtained an interesting `large spectral gap' property of $k$-uniform hypergraphs, which is an extension to hypergraphs of the eigenvalue property $(iv)$ from Theorem \ref{quasigraphs}, and showed that it is polynomially equivalent to the quasirandomness properties considered in the last theorem
(with $d$ substituted for $1$).\footnote{Lenz and Mubayi also considered related notions of quasirandomness for $k$-graphs corresponding to any proper partition $\pi$ of $k$, proving an analogue of Theorem \ref{Towsnerthm} for these `$\pi$-quasirandom' properties (which are similar to but more general than those for quasirandomness of order $1$, which corresponds to the partition $k = 1 + \dots + 1$ into $k$ ones).}

\subsection{Strong quasirandomness and the octahedral norms} \label{OctNormsSection}

In order to control the number of \emph{every} subhypergraph $F$ in a $k$-graph $H$, we need $H$ to be quasirandom of order $k-1$.
We say that such hypergraphs are \emph{strongly quasirandom}.

We have already seen that a random $k$-graph $H$, where each $k$-set of vertices is chosen to be an edge independently with the same probability $p$, will be strongly quasirandom with high probability.
Let us now take a look at a couple of deterministic examples, both taken from Chung and Graham's paper \cite{QuasiSetSystems};
we refer the reader to this paper for the proofs that they are indeed strongly quasirandom.

\begin{ex}
Let $p$ be a fixed (large) prime number.
We define the \emph{Paley $k$-graph} $P^{(k)}_p$ as the hypergraph whose vertices are the elements of the finite field $\F_p$, and where $\{x_1, \dots, x_k\}$ is an edge iff $x_1 + \dots + x_k$ is a square in $\F_p$ (that is, a quadratic residue).
Then $P^{(k)}_p$ has edge density $1/2 + o(1)$, and due to the strong pseudorandomness properties of quadratic residues one can show that $\|P^{(k)}_p - 1/2\|_{\square^k_{k-1}} = o(1)$.
\end{ex}

\begin{ex}
Given a positive integer $n$, define the \emph{even intersection $k$-graph} $I^{(k)}(n)$ as follows:
the vertices of $I^{(k)}(n)$ are all subsets of $[n]$, and a $k$-set $\{X_1, \dots, X_k\}$, $X_i \subseteq [n]$, is an edge iff
$$\left|\bigcap_{i=1}^k X_i\right| \equiv 0 \mod 2.$$
This hypergraph has edge density $1/2 + o(1)$, and it satisfies $\|I^{(k)}(n) - 1/2\|_{\square^k_{k-1}} = o(1)$.
\end{ex}

It was shown by Chung and Graham \cite{QuasirandomHypergraphs, QuasiSetSystems} (for edge density $1/2$) and by Kohayakawa, R\"odl and Skokan \cite{HypergraphQuasirandomnessRegularity} (for general edge density $0 \leq \delta \leq 1$) that a $k$-graph being strongly quasirandom is asymptotically equivalent to it having the almost minimal number of copies of the \emph{$k$-octahedron} $\oct^{(k)}$,
i.e. the complete $k$-partite $k$-graph with vertex classes of size $2$:
$$V(\oct^{(k)}) = \bigcup_{i=1}^k \big\{x^{(0)}_i, x^{(1)}_i\big\}, \hspace{2mm} \oct^{(k)} = \Big\{ \big\{ x^{(\omega_1)}_1, \dots, x^{(\omega_k)}_k \big\}: \omega \in \{0, 1\}^k \Big\}.$$

Note that, when $k=2$, the octahedron is just the $4$-cycle graph $C_4$.
For $k=3$ it is the 3-graph on vertex set $\{ x^{(0)}, x^{(1)}, y^{(0)}, y^{(1)}, z^{(0)}, z^{(1)} \}$ whose edges are given by $\{x^{(\omega_1)}, y^{(\omega_2)}, z^{(\omega_3)}\}$ for all choices of $\omega_1, \omega_2, \omega_3 \in \{0, 1\}$;
it represents the vertices and faces of an octahedron (the three-dimensional polytope), which explains the name.
We also note that $\oct^{(k)}$ coincides with the hypergraph $\textsc{M}^{(k)}_{k-1}$ defined in the last subsection.

The main strongly quasirandom property equivalences proven by Chung and Graham \cite{QuasirandomHypergraphs, QuasiSetSystems} and by Kohayakawa, R\"odl and Skokan \cite{HypergraphQuasirandomnessRegularity} are then the following:

\begin{thm}[Equivalence theorem for strong quasirandomness] \label{octequiv}
Let $H$ be a $k$-uniform hypergraph with edge density $\delta$.
Then the following properties are polynomially equivalent:
\begin{itemize}
    \item[$(i)$] $H$ is strongly quasirandom: \hspace{2mm}
    $\| H - \delta \|_{\square^{k}_{k-1}} \leq c_1$.
    \item[$(ii)$] $H$ correctly counts \emph{all} hypergraphs:
    $$t(F, H) = \delta^{|F|} \pm c_2 |F| \hspace{5mm} \text{for all k-graphs } F.$$
    \item[$(iii)$] $H$ has few octahedra: \hspace{2mm}
    $t(\oct^{(k)}, H) \leq \delta^{2^k} + c_3$.
    \item[$(iv)$] $H$ has small deviation:
    $$\mathbb{E}_{\mathbf{x}^{(0)},\, \mathbf{x}^{(1)} \in V^k} \Bigg[ \prod_{\omega \in \{0, 1\}^k} \big( H \big(x_1^{(\omega_1)},\, \dots,\, x_k^{(\omega_k)}\big) - \delta \big) \Bigg] \leq c_4.$$
\end{itemize}
\end{thm}

\begin{remark}
As usual, we give only the `core' properties of the theorem and refer the reader to the original papers for the full results.
\end{remark}

The central concept in Chung and Graham's paper \cite{QuasiSetSystems} was the
\emph{deviation}\footnote{The deviation can also be interpreted as the average difference between the numbers of even and odd partial octahedra (that is, subgraphs of $\oct^{(k)}$ with an even or odd number of edges) present in $H$, which might help to explain the name.}
of a hypergraph $H$, whose definition is equivalent to the density $t(\oct^{(k)}, \mu_H)$ of octahedra weighted by the \emph{multiplicative edge function} $\mu_H$, which maps edges of $H$ to $-1$ and non-edges to $1$.
Their work, however, focused on hypergraphs of edge density $1/2$;
for general edge density $\delta$ the corresponding weight function should be given by the balanced function $f_H := H - \delta$, which gives the expression in property $(iv)$ above.

This expression turns out to be always nonnegative, and if it is small then (by the theorem above) $H$ is strongly quasirandom.
This motivates the following definition, which is essentially due to Gowers \cite{Regularity3graphs, HypergraphRegularityGowers}:

\begin{definition}
Given a function $f: V^k \rightarrow \R$, we define its \emph{octahedral norm} by
\begin{equation} \label{octnorm}
    \|f\|_{\oct^k} := \mathbb{E}_{\mathbf{x}^{(0)},\, \mathbf{x}^{(1)} \in V^k} \Bigg[ \prod_{\omega \in \{0, 1\}^k} f \big( \mathbf{x}^{(\omega)} \big) \Bigg]^{1/2^k},
\end{equation}
where we write $\mathbf{x}^{(\omega)} := \big(x_i^{(\omega_i)}\big)_{i \in [k]}$.
\end{definition}

While not obvious that the right-hand side of (\ref{octnorm}) gives a positive number, we will soon show that this is the case and so $\|f\|_{\oct^k}$ is well-defined and positive for all real functions $f$
(it also satisfies the triangle inequality, as we will see later).
Note that $\|f\|_{\oct^k}^{2^k} = t(\oct^{(k)}, f)$ is the weighted count of $k$-octahedra.

The octahedral norm has an associated \emph{inner product of order $k$}, denoted $\langle \cdot \rangle_{\oct^k}$, which we define for $2^k$ functions $f_{\omega}: V^k \rightarrow \R$, $\omega \in \{0, 1\}^k$, by
\begin{equation} \label{OctInnerProduct}
    \left\langle (f_{\omega})_{\omega \in \{0, 1\}^k} \right\rangle_{\oct^k} := \mathbb{E}_{\mathbf{x}^{(0)}, \mathbf{x}^{(1)} \in V^k} \Bigg[ \prod_{\omega \in \{0, 1\}^k} f_{\omega}\big(\mathbf{x}^{(\omega)}\big) \Bigg].
\end{equation}
With this inner product we have that
$\|f\|_{\oct^k}^{2^k} = \left\langle f, f, \dots, f \right\rangle_{\oct^k}$.

A very useful property of the octahedral norm and inner product is that they satisfy a type of Cauchy-Schwarz inequality.
This result was first established by Gowers (though with a different notation), and is now known as the \emph{Gowers-Cauchy-Schwarz inequality}:

\begin{lem}[Gowers-Cauchy-Schwarz inequality] \label{GowersCS}
For any collection of functions $f_{\omega}: V^k \rightarrow \R$, $\omega \in \{0, 1\}^k$, we have
$$\left\langle (f_{\omega})_{\omega \in \{0, 1\}^k} \right\rangle_{\oct^k} \leq \prod_{\omega \in \{0, 1\}^k} \| f_{\omega} \|_{\oct^k}.$$
\end{lem}

\begin{proof}
We first isolate the last pair of variables $x_k^{(0)}$, $x_k^{(1)}$ from the rest, breaking the expectation $\mathbb{E}_{\mathbf{x}^{(0)}, \mathbf{x}^{(1)} \in V^k}$
in the definition (\ref{OctInnerProduct})
into one expectation over
$\mathbf{x}^{(0)}_{[k-1]}$, $\mathbf{x}^{(1)}_{[k-1]} \in V^{k-1}$
and one over $x_k^{(0)}$, $x_k^{(1)} \in V$.
We can then write $\left\langle (f_{\omega})_{\omega \in \{0, 1\}^k} \right\rangle_{\oct^k}$ as
\begin{align*}
    \mathbb{E}_{\mathbf{x}^{(0)}_{[k-1]},\, \mathbf{x}^{(1)}_{[k-1]}} \Bigg[ \mathbb{E}_{x^{(0)}_k} \Bigg[ \prod_{\omega' \in \{0, 1\}^{k-1}} f_{\omega', 0} \big(\mathbf{x}^{(\omega')}, x_k^{(0)}\big) \Bigg] \,\mathbb{E}_{x^{(1)}_k} \Bigg[ \prod_{\omega' \in \{0, 1\}^{k-1}} f_{\omega', 1} \big(\mathbf{x}^{(\omega')}, x_k^{(1)}\big) \Bigg] \Bigg].
\end{align*}
Applying Cauchy-Schwarz to this outer expectation and collecting the terms, we obtain
$$\left\langle (f_{\omega})_{\omega \in \{0, 1\}^k} \right\rangle_{\oct^k} \leq
\left\langle (f_{\omega', 0})_{\omega \in \{0, 1\}^k} \right\rangle_{\oct^k}^{1/2} \left\langle (f_{\omega', 1})_{\omega \in \{0, 1\}^k} \right\rangle_{\oct^k}^{1/2},$$
where we write $\omega' := (\omega_1, \dots, \omega_{k-1}) \in \{0, 1\}^{k-1}$ for the first $k-1$ terms of $\omega$.
Similarly for any other choice of variables $x_i^{(0)}$, $x_i^{(1)}$ to be separated from the rest.

Applying this inequality consecutively for each pair $x_i^{(0)}$, $x_i^{(1)}$ of variables, we obtain at the end
$$\left\langle (f_{\omega})_{\omega \in \{0, 1\}^k} \right\rangle_{\oct^k} \leq \prod_{\omega \in \{0, 1\}^k} \left\langle f_{\omega}, f_{\omega}, \dots, f_{\omega} \right\rangle_{\oct^k}^{1/2^k}.$$
The result now follows from the identity
$\|f\|_{\oct^k} = \left\langle f, f, \dots, f \right\rangle_{\oct^k}^{1/2^k}$.
\end{proof}

With the Gowers-Cauchy-Schwarz inequality in hand it is easy to  show that the octahedral norm satisfies the triangle inequality, and is thus really a
norm.\footnote{That $\|f\|_{\oct^k} \neq 0$ whenever $f$ is non-zero follows immediately from Lemma \ref{cut<oct} given below, for instance.}
Indeed, by linearity of the inner product we have
\begin{align*}
    \|f + g\|_{\oct^k}^{2^k} &= \big\langle f + g,\, f + g,\, \dots,\, f + g \big\rangle_{\oct^k} \\
    &= \big\langle f,\, f + g,\, \dots,\, f + g \big\rangle_{\oct^k} + \big\langle g,\, f + g,\, \dots,\, f + g \big\rangle_{\oct^k} \\
    &\leq \|f\|_{\oct^k} \|f+g\|_{\oct^k}^{2^k-1} + \|g\|_{\oct^k} \|f+g\|_{\oct^k}^{2^k-1},
\end{align*}
from which we deduce that $\|f + g\|_{\oct^k} \leq \|f\|_{\oct^k} + \|g\|_{\oct^k}$.

Another important consequence of the Gowers-Cauchy-Schwarz inequality is that the octahedral norms are \emph{stronger} than the cut norm:

\begin{lem} \label{cut<oct}
For any function $f: V^k \rightarrow \R$, we have $\| f \|_{\square^{k}_{k-1}} \leq \| f \|_{\oct^k}$.
\end{lem}

\begin{proof}
Given functions $u_B: V^B \rightarrow [0, 1]$, $B \in \binom{[k]}{k-1}$, let $f_{\omega_B}: V^{[k]} \rightarrow \R$ be the function defined by $f_{\omega_B}(\mathbf{x}_{[k]}) = u_B(\mathbf{x}_B)$, where $\omega_B \in \{0, 1\}^{[k]}$ is the indicator vector of the set $B$.
Denote also $f_{\mathbf{1}} = f$ and $f_{\omega} \equiv 1$ for all $\omega \in \{0, 1\}^{[k]} \setminus \{\mathbf{1}\}$ not contained in the set $\big\{\omega_B: B \in \binom{[k]}{k-1}\big\}$.

Using the Gowers-Cauchy-Schwarz inequality we conclude that
\begin{align*}
    \Bigg| \mathbb{E}_{\mathbf{x} \in V^{[k]}} \Bigg[ f(\mathbf{x}) \prod_{B \in \binom{[k]}{k-1}}{u_{B}(\mathbf{x}_{B})} \Bigg] \Bigg| &= \Bigg|\mathbb{E}_{\mathbf{x}^{(0)}, \mathbf{x}^{(1)} \in V^k} \Bigg[ \prod_{\omega \in \{0, 1\}^k} f_{\omega} \big(\mathbf{x}^{(\omega)}\big) \Bigg] \Bigg| \\
    &\leq \prod_{\omega \in \{0, 1\}^k} \| f_{\omega} \|_{\oct^k}.
\end{align*}
Since clearly $\| f_{\omega} \|_{\oct^k} \leq \| f_{\omega} \|_{L^{\infty}} \leq 1$ for all $\omega \in \{0, 1\}^{[k]} \setminus \{\mathbf{1}\}$, the last product is at most $\| f \|_{\oct^k}$.
As this inequality is valid for all functions $u_B: V^B \rightarrow [0, 1]$, $B \in \binom{[k]}{k-1}$, the claim follows.
\end{proof}

As a special case of this lemma, we see that $\| f \|_{\oct^k} \geq \left|\mathbb{E}[f]\right|$.
Applying this to a hypergraph $H$ of edge density $\delta$ we conclude that $t(\oct^{(k)}, H) \geq \delta^{2^k}$, showing that any $k$-graph $H$ will contain at least $\delta^{2^k} v(H)^{2k}$ (homomorphic) copies of $\oct^{(k)}$ as a subhypergraph.
This explains why it is enough to require that $t(\oct^{(k)}, H) \leq \delta^{2^k} + c_3$ in item $(iii)$ of Theorem \ref{octequiv}.

We are now ready to prove the equivalence theorem for strong quasirandomness.

\begin{proof}[Proof of Theorem \ref{octequiv}]
$(i) \Rightarrow (ii)$:
This follows immediately from the counting lemma (Lemma \ref{count_d-linear}), and we may take $c_2 = c_1$.

\medskip
$(ii) \Rightarrow (iii)$:
This is a special case, and we may take $c_3 = 2^k c_2$.

\medskip
$(iii) \Rightarrow (iv)$:
We are given that $t(\oct^{(k)}, H) \leq \delta^{2^k} + c_3$, and we wish to bound
\begin{align*}
    t(\oct^{(k)}, H - \delta) &= \mathbb{E}_{\mathbf{x} \in V(H)^{V(\oct^{(k)})}} \Bigg[ \prod_{e \in \oct^{(k)}} \left( H(\mathbf{x}_e) - \delta \right) \Bigg] \\
    &= \mathbb{E}_{\mathbf{x} \in V(H)^{V(\oct^{(k)})}} \Bigg[ \sum_{F \subseteq \oct^{(k)}} \prod_{e \in F} H(\mathbf{x}_e) \prod_{e \in \oct^{(k)} \setminus F} (-\delta) \Bigg] \\
    &= \sum_{F \subseteq \oct^{(k)}} t(F, H) \cdot (-\delta)^{2^k-|F|}
\end{align*}
(where the sum is over all labeled subhypergraphs $F$ of $\oct^{(k)}$).
The main issue in bounding this last sum is that it contains both positive and negative terms;
the idea to get around this problem is to consider instead the related expression
$$t(\oct^{(k)}, H + \delta) + t(\oct^{(k)}, H - \delta) = \sum_{F \subseteq \oct^{(k)}} t(F, H) \big( \delta^{2^k-|F|} + (-\delta)^{2^k-|F|} \big),$$
which contains only nonnegative terms.

Using Gowers-Cauchy-Schwarz we see that $t(F, H) \leq \|H\|_{\oct^k}^{|F|}$ holds whenever $F$ is a subhypergraph of $\oct^{(k)}$:
we can write $t(F, H)$ as the inner product
\begin{equation*}
    \langle (f_{\omega})_{\omega \in \{0, 1\}^k} \rangle_{\oct^k}
    \hspace{3mm} \text{where }
    f_{\omega} =
    \begin{cases}
        H &\text{ if } \omega \text{ is an edge of }F, \\ 
        1 &\text{ otherwise}
    \end{cases}
\end{equation*}
and so $t(F, H) \leq \|H\|_{\oct^k}^{|F|} \|1\|_{\oct^k}^{2^k-|F|} = \|H\|_{\oct^k}^{|F|}$.
We then obtain
\begin{align*}
    t(\oct^{(k)}, H + \delta) + t(\oct^{(k)}, H - \delta)
    &\leq \sum_{F \subseteq \oct^{(k)}} \|H\|_{\oct^k}^{|F|} \big( \delta^{2^k-|F|} + (-\delta)^{2^k-|F|} \big) \\
    &= \big( \|H\|_{\oct^k} + \delta \big)^{2^k} + \big( \|H\|_{\oct^k} - \delta \big)^{2^k},
\end{align*}
where this last equality follows from the binomial expansion.

By assumption we have that $\delta = \mathbb{E}[H] \leq \|H\|_{\oct^k} \leq \delta + c_3^{1/2^k}$, which implies
$$\big( \|H\|_{\oct^k} + \delta \big)^{2^k} + \big( \|H\|_{\oct^k} - \delta \big)^{2^k} \leq \big(2\delta + c_3^{1/2^k}\big)^{2^k} + c_3.$$
A quick computation using that $c_3, \delta \leq 1$ permits us to bound the right-hand side above by $(2\delta)^{2^k} + 2^{2^{k+1}} c_3^{1/2^k}$.
Since $t(\oct^{(k)}, H + \delta) = \|H + \delta\|_{\oct^k}^{2^k} \geq (2\delta)^{2^k}$, we finally conclude that
$$t(\oct^{(k)}, H - \delta) + (2\delta)^{2^k} \leq (2\delta)^{2^k} + 2^{2^{k+1}} c_3^{1/2^k},$$
which is exactly property $(iv)$ with $c_4 = 2^{2^{k+1}} c_3^{1/2^k}$.

\medskip
$(iv) \Rightarrow (i)$:
If we suppose $t(\oct^{(k)},\, H - \delta) \leq c_4$, then $\| H - \delta \|_{\oct^k} \leq c_4^{1/2^k}$ and the claim follows from the inequality $\| H - \delta \|_{\square^{k}_{k-1}} \leq \| H - \delta \|_{\oct^k}$ given in Lemma \ref{cut<oct} (with $c_1 = c_4^{1/2^k}$).
\end{proof}


\subsection{Partite hypergraphs} \label{partiteHypergraphs}

As in the case of graphs, it is useful to also consider notions of quasirandomness for \emph{partite hypergraphs}.

A hypergraph is said to be \emph{$\ell$-partite} if its vertex set can be partitioned into $\ell$ classes in such a way that every edge of $H$ contains at most one vertex from any of these classes.
We shall assume such a partition $V(H) = V_1 \cup \dots \cup V_{\ell}$ is fixed and part of the description of the $\ell$-partite hypergraph $H$ in consideration.

Given a collection $(V_i)_{i \in [\ell]}$ of non-empty sets and any $B \subseteq [\ell]$, let us write $V_B := \prod_{i \in B} V_i$ for the Cartesian product.
For a given hypergraph $H$ and $k$ disjoint subsets $U_1, \dots, U_k \subset V(H)$, we denote by $H[U_1, \dots, U_k]$ the $k$-partite $k$-graph on $U_1 \cup \dots \cup U_k$ whose edges are the restriction of $H$ to $U_{[k]}$.
Note that we can write any $\ell$-partite $k$-graph $H$ on $V_1 \cup \dots \cup V_{\ell}$ as the edge-disjoint union of $\binom{\ell}{k}$ $k$-partite $k$-graphs:
$$H \,=\, \bigcup_{B \in \binom{[\ell]}{k}} H \big[(V_i)_{i \in B}\big] \,=\, \bigcup_{1 \leq i_1 < \dots < i_k \leq \ell} H \big[V_{i_1},\, \dots,\, V_{i_k}\big].$$

We then have the following definitions, which are the natural extensions of our earlier notions of quasirandomness and homomorphism density for partite hypergraphs.

\begin{definition}
Let $(V_i)_{i \in [k]}$ be a collection of $k$ non-empty sets.
Given a function $f: V_{[k]} \rightarrow \R$ and an integer $1 \leq d \leq k-1$, we define the \emph{$(k, d)$-cut norm} of $f$ by
$$\|f\|_{\square^k_d} := \max_{S_B \subseteq V_{B}\; \forall B \in \binom{[k]}{d}} \Bigg| \mathbb{E}_{\mathbf{x} \in V_{[k]}} \Bigg[f(\mathbf{x}) \prod_{B \in \binom{[k]}{d}}{S_B(\mathbf{x}_{B})}\Bigg] \Bigg|,$$
where the maximum is over all collections of sets $(S_B)_{B \in \binom{[k]}{d}}$ where each $S_B$ is a subset of $\prod_{i \in B} V_i$.
We say that a $k$-partite $k$-graph $H$ on $(V_i)_{i \in [k]}$ is \emph{$\varepsilon$-quasirandom of order $d$} if $\|H - \delta\|_{\square^k_d} \leq \varepsilon$, where $\delta := |H|/|V_{[k]}|$ denotes its edge density.
\end{definition}

\begin{definition}
Let $F$ and $H$ be $\ell$-partite hypergraphs with partition classes $(U_i)_{i \in [\ell]}$ and $(V_i)_{i \in [\ell]}$, respectively.
We say that a map $\phi: V(F) \rightarrow V(H)$ is \emph{$\ell$-partite} if it maps each $U_i$ into $V_i$, i.e. if $\phi(U_i) \subseteq V_i$ for all $1 \leq i \leq \ell$.
We define the \emph{canonical homomorphism density} of $F$ on $H$, denoted $t_{can}(F, H)$, as the probability that a uniformly chosen $\ell$-partite map $\phi: V(F) \rightarrow V(H)$ preserves edges;
in formulas:
$$t_{can}(F, H) := \mathbb{E}_{x_{u_1} \in V_1\; \forall u_1 \in U_1} \dots \mathbb{E}_{x_{u_{\ell}} \in V_{\ell}\; \forall u_{\ell} \in U_{\ell}} \Bigg[ \prod_{e \in F} H(\mathbf{x}_e) \Bigg].$$
\end{definition}

We note that the octahedral norms can also be naturally extended to the `partite case' of functions $f: V_{[k]} \rightarrow \R$ by defining
$$\|f\|_{\oct^k} := t_{can}(\oct^{(k)}, f)^{1/2^k} = \mathbb{E}_{\mathbf{x}^{(0)},\, \mathbf{x}^{(1)} \in V_{[k]}} \Bigg[ \prod_{\omega \in \{0, 1\}^k} f\big(\mathbf{x}^{(\omega)}\big) \Bigg]^{1/2^k}.$$
All of its properties, such as the Gowers-Cauchy-Schwarz inequality and the inequality $\|f\|_{\square^k_{k-1}} \leq \|f\|_{\oct^k}$, continues to hold in this case
(with unchanged proofs).

With these definitions in hand, one can easily obtain an analogue of the counting lemma (Lemma \ref{count_d-linear}) for partite hypergraphs.
Indeed, it is interesting to note that the counting lemma in the partite case rests valid also for \emph{non-uniform} hypergraphs, that is, when the host hypergraph $H$ (and also the smaller hypergraph $F$ being counted) contains edges of different sizes.
For it to hold it suffices to require the hypergraph $F$ to be $d$-linear:

\begin{lem}
Let $F$ be a $d$-linear hypergraph on $[\ell]$ and let $H$ be an $\ell$-partite hypergraph on $V_1 \cup \dots \cup V_{\ell}$.
Suppose that, for all edges $e \in F$, the
$|e|$-partite $|e|$-graph $H[(V_i)_{i \in e}]$
is $\varepsilon$-quasirandom of order $d$ with edge density $\delta_e$.
Then we have:
$$t_{can}(F, H) = \prod_{e \in F} \delta_e \pm \varepsilon |F|.$$
\end{lem}

The proof of this result is essentially identical to that of our counting lemma for quasirandomness of order $d$ (Lemma \ref{count_d-linear}), and so we refrain from giving it here.

The equivalence theorem for strong quasirandomness in the partite hypergraph setting was explicitly worked out by Kohayakawa, R\"odl and Skokan \cite{HypergraphQuasirandomnessRegularity}, who in fact used it as a step in their proof of its non-partite version.
The next theorem, which deals with quasirandomness of any fixed order $1 \leq d < k$ for $k$-partite $k$-graphs, follows (in a qualitative, asymptotically equivalent form) from the arguments of Towsner \cite{SigmaAlgebrasHypergraphs};
the polynomial bounds as stated follow from the methods presented in \cite{QuasirandomnessHypergraphs, CayleyTypeHypergraphs}.
As in the case of partite graphs, the notion of polynomial equivalence in this theorem must take into account the size of each one of the partition classes.

\begin{thm}
Let $1 \leq d < k$ be integers and let $H$ be a $k$-partite $k$-uniform hypergraph with edge density $\delta$.
Then the following properties are polynomially equivalent:
\begin{itemize}
    \item[$(i)$] $H$ is quasirandom of order $d$: \hspace{2mm}
    $\| H - \delta \|_{\square^{k}_d} \leq c_1$.
    \item[$(ii)$] $H$ correctly counts all $k$-partite $d$-linear hypergraphs:
    $$t_{can}(F, H) = \delta^{|F|} \pm c_2 |F| \hspace{5mm} \forall F \in \mathcal{L}^{(k)}_{d} \, k \text{-partite}.$$
    \item[$(iii)$] $H$ has few copies of $M = \textsc{M}^{(k)}_{d}$: \hspace{2mm}
    $t_{can}(M, H) \leq \delta^{|M|} + c_3$.
    \item[$(iv)$] $H$ has small deviation with respect to $M = \textsc{M}^{(k)}_{d}$:
    $$\mathbb{E}_{\mathbf{x} \in V(H)^{V(M)}} \Bigg[ \prod_{e \in M} \left( H(\mathbf{x}_e) - \delta \right) \Bigg] \leq c_4.$$
\end{itemize}
\end{thm}

\section{Comparing quasirandomness in additive groups and in hypergraphs} \label{Cayley}

We have seen that both in the setting of additive groups and in that of hypergraphs there is a natural hierarchy of notions of quasirandomness, depending on the `order' or `degree' of the structures it can detect.
The main goal of this section is to understand how these two different classes of quasirandomness notions relate to each other.

In order to do this it will be necessary to consider these two families of combinatorial objects in the same framework.
A simple and convenient way of doing so is by defining the \emph{Cayley hypergraph} $H^{(k)}\!A$ of an additive set $A \subseteq G$,
whose vertex set is the underlying group $G$ and elements $x_1, \dots, x_k \in G$ form an edge iff $x_1 + \dots + x_k \in A$.


More generally, one can define a `Cayley-type hypergraph' related to an additive set by any given linear form $\phi: G^k \rightarrow G$, or any system of such linear forms.
Such generalizations are also interesting and will be considered in Sections \ref{generalCayley} and \ref{LinearConfigSection}, but for now we concentrate on the simpler case of Cayley hypergraphs as given above.

Recall that we have already seen a strong connection between \emph{linearly} uniform sets and their associated Cayley graphs, which coincide with the definition above when $k = 2$.
Indeed, by the equivalence theorem for uniform sets (Theorem \ref{linearsets}) a set $A \subseteq G$ is linearly uniform if and only if its Cayley graph $\Gamma_A$ is quasirandom;
we shall now see how this generalizes to higher orders.



\subsection{Quasirandomness for additive sets and their Cayley hypergraphs} \label{set-hypergraph}

Let us start with a couple of definitions which will facilitate our study.
The first one is meant to simplify the notation somewhat:

\begin{definition}
Given an integer $k$ and an additive group $G$, we denote by $s: G^k \rightarrow G$ its \emph{summing operator}
$$s(x_1, x_2, \dots, x_k) := x_1 + x_2 + \dots + x_k.$$
\end{definition}

\begin{remark}
There is a slight abuse of notation here since the same designation is used no matter how many terms are being summed or which additive group the summands belong to.
These `hidden parameters' may change each time the operator is used.
\end{remark}

Note that, if we allow for repeated vertices inside edges of the Cayley hypergraph $H^{(k)}\!A$, then its indicator function can be written more economically as $A \circ s$ on $G^k$.
Since there are at most $\binom{k}{2} |G|^{k-1}$ tuples $\mathbf{x} \in G^k$ with a repeated element and $|G|^k$ $k$-tuples in total, when averaging the distinction will be of order $O(k^2/|G|)$ and thus negligible for our purposes.

In order not to clutter our estimates and proofs with these negligible error terms, we will assume from now on that \emph{a Cayley hypergraph $H^{(k)}\!A$ may have loops}:
its edges are all unordered $k$-tuples of (not necessarily distinct) elements $x_1, \dots, x_k$ whose sum lies in the set $A$.
We can similarly define a weighted Cayley hypergraph associated to a function $f: G \rightarrow \R$ by
$H^{(k)}f(x_1, \dots, x_k) = f \circ s(x_1, \dots, x_k)$.

A simple but important property of our notions of quasirandomness for Cayley hypergraphs, which allows them to be analyzed by more `arithmetical' means, is their \emph{translation invariance}:

\begin{definition}
Given an element $a \in G$, we define the \emph{translation operator} $T^{a}$ on $\R^G$ by $T^{a} f(x) := f(x+a)$.
If $A$ is (the indicator function of) a set, then $T^a A$ is (the indicator function of) the translated set $A - a$.
\end{definition}

For any function $f: G \rightarrow \R$ and any group element $a$, we then have that
$$\|H^{(k)} T^{a}f\|_{\square^k_d} = \|H^{(k)}f\|_{\square^k_d} \hspace{5mm} \text{for all } 1 \leq d < k;$$
this follows immediately from the easily checked identity
$$\mathbb{E}_{\mathbf{x} \in G^{[k]}} \Bigg[ T^a f \circ s(\mathbf{x}) \prod_{B \in \binom{[k]}{d}}{u_{B}(\mathbf{x}_{B})} \Bigg]
= \mathbb{E}_{\mathbf{x} \in G^{[k]}} \Bigg[ f \circ s(\mathbf{x}) \prod_{B \in \binom{[k]}{d}} v_{B}(\mathbf{x}_{B}) \Bigg],$$
where $v_B = u_B$ if $1 \notin B$ and $v_B(x_1, \mathbf{x}_{B \setminus \{1\}}) = u_B(x_1-a, \mathbf{x}_{B \setminus \{1\}})$ if $1 \in B$.

As the first step in formally connecting the notion of quasirandomness in additive groups to that in hypergraphs, we will now show a strong connection between the $U^k$ uniformity norms (defined in Section \ref{HigherDegUniformity}) and the $\oct^k$ octahedral norms:

\begin{lem}[Relationship between the $U^k$ and $\oct^{k}$ norms] \label{RelationUkOctk}
For every real function $f: G \rightarrow \R$ we have that $\| f \circ s \|_{\oct^{k}} = \| f \|_{U^k}$.
\end{lem}

\begin{proof}
We make the change of variables
$$x := s \big(\mathbf{x}^{(0)}\big) = x_1^{(0)} + \dots + x_k^{(0)}, \hspace{5mm} h_i := x^{(1)}_i - x^{(0)}_i \hspace{3mm} \forall i \in [k].$$
Then $s \big(\mathbf{x}^{(\omega)}\big) = x + \sum_{i=1}^k \omega_i h_i$ for all $\omega \in \{0, 1\}^k$, and the result follows.
\end{proof}

As a special case of this relationship, we note that
$$\|H^{(k)}\!A - \delta\|_{\oct^{k}} = \|(A - \delta)\circ s\|_{\oct^{k}} = \|A - \delta\|_{U^k};$$
since the octahedral norm is stronger than the cut norm, it follows that $H^{(k)}\!A$ is quasirandom of order $k-1$ whenever $A$ is uniform of degree $k-1$.

The next theorem shows that a similar phenomenon holds for any degree $1 \leq d < k$ of uniformity:
a uniform set $A$ of some degree $d$ generates quasirandom Cayley hypergraphs of order $d$ having any edge-size.
Moreover, it suffices for one of those hypergraphs to be quasirandom of order $d$ for us to conclude that $A$ is uniform of degree $d$.
This result is due to Castro-Silva \cite{CayleyTypeHypergraphs}, and generalizes a theorem of Aigner-Horev and H\`an \cite{LinearQuasirandomness} who showed a similar relationship for linearly uniform sets and quasirandomness of order $1$.

\begin{thm} \label{AignerHan}
Let $G$ be a finite additive group and $A \subseteq G$ be a subset.
\begin{itemize}
    \item[$a)$] If $A$ is $\varepsilon$-uniform of degree $d$, then for all $k \geq d + 1$ the Cayley hypergraph $H^{(k)}\!A$ is $\varepsilon$-quasirandom of order $d$.
    \item[$b)$] Conversely, if $H^{(d + 1)}A$ is $\varepsilon$-quasirandom of order $d$, then $A$ is $(2\varepsilon^{1/2^{d+1}})$-uniform of degree $d$.
\end{itemize}
\end{thm}

\begin{proof}
We will prove the result more generally for bounded functions $f: G \rightarrow [-1, 1]$ instead of sets $A \subseteq G$.
The statement then follows by taking $f = A - \delta$ to be the balanced indicator function of the considered set $A$.

\medskip
$a)$
Choose optimal functions $u_B: G^B \rightarrow [0, 1]$, $B \in \binom{[k]}{d}$, so that
$$\|H^{(k)}f\|_{\square^{k}_{d}} = \Bigg| \mathbb{E}_{\mathbf{x} \in G^{k}} \Bigg[ f(s(\mathbf{x})) \prod_{B \in \binom{[k]}{d}} u_B(\mathbf{x}_B) \Bigg] \Bigg|.$$
We may separate the first $d+1$ variables $\mathbf{x}_{[d+1]}$ from the rest and write
\begin{equation*}
    \|H^{(k)}f\|_{\square^{k}_{d}}
    = \Bigg| \mathbb{E}_{\mathbf{x}_{[k] \setminus [d+1]}} \mathbb{E}_{\mathbf{x}_{[d+1]}} \Bigg[ f \big( s(\mathbf{x}_{[d+1]}) + s(\mathbf{x}_{[k] \setminus [d+1]}) \big) \prod_{B \in \binom{[k]}{d}} u_B(\mathbf{x}_B) \Bigg] \Bigg|,
\end{equation*}
where the first expectation is over $G^{[k] \setminus [d+1]}$ and the second is over $G^{[d+1]}$.

Let us now fix $\mathbf{x}_{[k] \setminus [d+1]} \in G^{[k] \setminus [d+1]}$ and consider the inner expectation in the last expression.
Writing $y := s(\mathbf{x}_{[k] \setminus [d+1]})$, this expression can be written as
$$\mathbb{E}_{\mathbf{x}_{[d+1]}} \Bigg[ T^{y}f \circ s(\mathbf{x}_{[d+1]}) \prod_{D \in \binom{[d+1]}{d}} v_{D}(\mathbf{x}_{D}) \Bigg]$$
for some suitable functions $v_{D}: G^D \rightarrow [0, 1]$, $D \in \binom{[d+1]}{d}$, and thus has absolute value at most
$$\| T^{y}f \circ s \|_{\square^{d+1}_{d}} = \| H^{(d+1)}T^{y}f \|_{\square^{d+1}_{d}} = \| H^{(d+1)}f \|_{\square^{d+1}_{d}}.$$
Since the octahedral norm is stronger than the cut norm (Lemma \ref{cut<oct}), this last term is at most
$\| H^{(d+1)}f \|_{\oct^{d+1}} = \| f \|_{U^{d+1}}$, which by assumption is bounded by $\varepsilon$.
Averaging over $\mathbf{x}_{[k] \setminus [d+1]} \in G^{[k] \setminus [d+1]}$ and using the triangle inequality we conclude that $\|H^{(k)}f\|_{\square^{k}_{d}} \leq \varepsilon$, as wished.

\medskip
$b)$
By definition we have
$$\| H^{(d+1)}f \|_{\oct^{d+1}}^{2^{d+1}}
= \mathbb{E}_{\mathbf{x}^{(0)} \in G^{d+1}} \mathbb{E}_{\mathbf{x}^{(1)} \in G^{d+1}} \Bigg[ \prod_{\omega \in \{0, 1\}^{d+1}} f \circ s \big(\mathbf{x}^{(\omega)}\big) \Bigg],$$
and so there is a choice of $\mathbf{x}^{(0)} \in G^{d+1}$ for which the inner expectation above is at least $\| H^{(d+1)}f \|_{\oct^{d+1}}^{2^{d+1}}$.
Fix such a value of $\mathbf{x}^{(0)}$ and decompose $f = f^+ - f^-$ into its positive and negative parts, so that
$$\mathbb{E}_{\mathbf{x}^{(1)} \in G^{d+1}} \Bigg[ H^{(d+1)}f \big(\mathbf{x}^{(1)}\big) \prod_{\omega \in \{0, 1\}^{d+1} \setminus \{\mathbf{1}\}} \big(f^+ \circ s \big(\mathbf{x}^{(\omega)}\big) - f^- \circ s \big(\mathbf{x}^{(\omega)}\big)\big) \Bigg]$$
is at least $\| H^{(d+1)}f \|_{\oct^{d+1}}^{2^{d+1}}$.

Expanding this product and using the triangle inequality, we conclude there is a choice of functions $u_{\omega} \in \{f^+ \circ s, f^- \circ s\}$, $\omega \in \{0, 1\}^{d+1} \setminus \{\mathbf{1}\}$, for which
$$\Bigg|\mathbb{E}_{\mathbf{x}^{(1)} \in G^{d+1}} \Bigg[ H^{(d+1)}f \big(\mathbf{x}^{(1)}\big) \prod_{\omega \in \{0, 1\}^{d+1} \setminus \{\mathbf{1}\}} u_{\omega} \big(\mathbf{x}^{(\omega)}\big) \Bigg]\Bigg| \geq 2^{-(d+1)}\| H^{(d+1)}f \|_{\oct^{d+1}}^{2^{d+1}}.$$
Since these functions $u_{\omega}$ take values in $[0, 1]$ and each depends on at most $d$ of the variables $x_1^{(1)}, \dots, x_{d+1}^{(1)}$, the expression on the left-hand side is at most $\|H^{(d+1)}\!f\|_{\square^{d+1}_{d}}$.
The claim now follows from the identity $\|f\|_{U^{d+1}} = \|H^{(d+1)}f\|_{\oct^{d+1}}$.
\end{proof}

In the special case where $d = 1$, we can generalize the result given in $b)$ and show that the implication
$$H^{(k)}\!A \text{ is quasirandom of order } 1 \,\implies\, A \text{ is linearly uniform}$$
holds for any fixed $k \geq 2$.
Indeed, since $\|f\|_{U^2} = \|\widehat{f}\|_{\ell^4}$ the condition that $\| A - \delta \|_{U^2} \geq \varepsilon$ implies that
$|\widehat{A}(\gamma)| = \big|\mathbb{E}_{x} \big[ A(x) \overline{\gamma(x)} \big] \big| \geq \varepsilon^2$
for some character $\gamma \in \widehat{G} \setminus \{\mathbf{1}\}$.
Decomposing $\gamma = \text{Re}(\gamma) + i \text{Im}(\gamma)$ into its real and imaginary parts and using the triangle inequality, we see there is a choice of $u_j \in \{\text{Re}(\gamma), \text{Im}(\gamma)\}$, $1 \leq j \leq k$, for which
\begin{align*}
    &\Bigg| \mathbb{E}_{x_1, \dots, x_k \in G} \Bigg[\big( A(x_1 + \dots + x_k) - \delta \big) \prod_{j=1}^k u_j(x_j)\Bigg] \Bigg| \\
    &\hspace{3cm} \geq 2^{-k} \Bigg| \mathbb{E}_{x_1, \dots, x_k \in G} \Bigg[\big( A(x_1 + \dots + x_k) - \delta \big) \prod_{j=1}^k \gamma(x_j)\Bigg] \Bigg| \\
    &\hspace{3cm} = 2^{-k} \big| \mathbb{E}_{x \in G} \big[( A(x) - \delta) \gamma(x)\big] \big|
    \geq 2^{-k}\varepsilon^2.
\end{align*}
By further decomposing each $u_j$ into its positive and negative parts and then using the triangle inequality again, we conclude that
$\|H^{(k)}\!A - \delta\|_{\square^k_1} \geq 4^{-k} \varepsilon^2$
holds for all $k \geq d$ whenever $\|A - \delta\|_{U^2} \geq \varepsilon$.
Together with item $a)$ from the last theorem, this shows a complete equivalence (with polynomial bounds) between an additive set being linearly uniform and its $k$-uniform Cayley hypergraph being weakly quasirandom, for any fixed $k \geq 2$;
this equivalence was first obtained (in a different way) by Aigner-Horev and H\`an \cite{LinearQuasirandomness}.

\begin{remark}
It is natural to wonder if the same holds when $d > 1$, and one can relax the condition on item $b)$ to requiring that $H^{(k)}\!A$ is quasirandom for any one fixed $k > d$.
It is an easy consequence of the (very difficult) \emph{inverse theorem for the uniformity norms} on $\F_p^n$ \cite{InverseGowersAction, InverseConjectureFiniteFields, InverseConjectureLowCharacteristic} that this is indeed the case whenever the ambient group is $\F_p^n$ for some fixed prime $p$ and very large $n$
(though with far worse quantitative bounds);
see the author's paper \cite{CayleyTypeHypergraphs} for the details.
We leave the generalization to other additive groups $G$ as an open question.
\end{remark}

It follows from the last theorem and the counting lemma (Lemma \ref{count_d-linear}) given in the last section that one can count all $d$-linear subhypergraphs inside Cayley hypergraphs of sets that are uniform of degree $d$.
Interestingly, the extra symmetries satisfied by Cayley hypergraphs imply that a much \emph{stronger} result is true.

In order to show this we need to define another family of hypergraphs:

\begin{definition}
Given $d \geq 1$, we say that a hypergraph $F$ is \emph{$d$-simple} if the following is true:
for every edge $e \in F$, there exists a set of $d$ vertices $\{v_1, \dots, v_{d}\} \subseteq e$ which is not contained in any other edge of $F$ (i.e. $\{v_1, \dots, v_{d}\} \nsubseteq e'$ for all $e' \in F \setminus \{e\}$).
We denote the set of all $d$-simple $k$-graphs by $\mathcal{S}^{(k)}_{d}$.
\end{definition}

It is easy to see from the definition that all $d$-linear hypergraphs are $(d+1)$-simple, but as the next example shows the converse is false.

\begin{ex}
Let $F$ be the connected $k$-graph on $2k-d$ vertices and two edges (also considered last section in Example \ref{simplekgraph}).
Then $F$ is only $d$-linear, but it is $1$-simple.
\end{ex}

This very easy example shows that the difference between `how linear' and `how simple' a hypergraph can be is unbounded.
It also shows (in view of Example \ref{simplekgraph}) that one cannot hope to control the count of all $d$-simple subhypergraphs by using only quasirandomness of order $d$, say.

The next example will be very important in what follows;
it might be instructive to think of it as the `cheapest' way of transforming the $d$-octahedron into a $k$-uniform hypergraph.

\begin{ex}
Given $1 \leq d \leq k$, define the \emph{squashed octahedron} $\oct^{(k)}_{d}$ as the $k$-graph on vertex set $\big\{x^{(0)}_1,\, x^{(1)}_1,\, \dots,\, x^{(0)}_d,\, x^{(1)}_d,\, y_{d+1},\, \dots,\, y_{k}\big\}$ given by
$$\oct^{(k)}_{d} = \Big\{ \big\{x^{(\omega_1)}_1,\, \dots,\, x^{(\omega_d)}_d,\, y_{d+1},\, \dots,\, y_{k}\big\}:\, \omega \in \{0, 1\}^d \Big\}.$$
This hypergraph is only $(k-1)$-linear, but it is $d$-simple.
\end{ex}

The importance of the squashed octahedron $\oct^{(k)}_{d}$ stems from the fact that it is \emph{complete} for counting $d$-simple $k$-graphs inside Cayley hypergraphs $H^{(k)}\!A$, and also for concluding uniformity of degree $d-1$ for this set $A$.
(For a clearer exposition of the result we have changed the considered degree of uniformity from $d$ to $d-1$.)
More precisely, we have the equivalence theorem:

\begin{thm}[Equivalence theorem for quasirandom Cayley hypergraphs] \label{Cayley_equiv}
Let $A$ be a set of density $\delta$ in $G$ and let $d \geq 2$ be an integer.
Then for every fixed $k \geq d$ the following statements are polynomially equivalent:
\begin{itemize}
    \item[$(i)$] $A$ is uniform of degree $d-1$: \hspace{3mm}
    $\| A - \delta \|_{U^{d}} \leq c_1$.
    \item[$(ii)$] $H^{(k)}\!A$ correctly counts all $d$-simple hypergraphs:
    $$t(F,\, H^{(k)}\!A) = \delta^{|F|} \pm c_2 |F| \hspace{5mm} \forall F \in \mathcal{S}^{(k)}_{d}.$$
    \item[$(iii)$] $H^{(k)}\!A$ has few squashed octahedra $\oct^{(k)}_{d}$:
    $$t(\oct^{(k)}_{d},\, H^{(k)}\!A) \leq \delta^{2^{d}} + c_3.$$
    \item[$(iv)$] $H^{(k)}\!A$ has small $d$-deviation:
    $$\hspace{-1cm} \mathbb{E}_{\mathbf{x}^{(0)}, \mathbf{x}^{(1)} \in G^d} \,\mathbb{E}_{y_{d+1}, \dots, y_k \in G} \Bigg[ \prod_{\omega \in \{0, 1\}^d} \big( H^{(k)}\!A \big(x_1^{(\omega_1)}, \dots, x_d^{(\omega_d)}, y_{d+1}, \dots, y_k\big) - \delta \big) \Bigg] \leq c_4.$$
\end{itemize}
\end{thm}

\begin{proof}
$(i) \Rightarrow (ii)$:
Write $F = \{e_1, \dots, e_{|F|}\}$, $V = V(F)$, and for each $1 \leq i \leq |F|$ let $f_i \subseteq e_i$ be a set of $d$ elements which is not contained in any other edge $e_j$.
By the usual telescoping sum argument we have
\begin{align*}
    &\big| t(F, H^{(k)}A) - \delta^{|F|} \big| \\
    &\hspace{5mm}\leq \sum_{i = 1}^{|F|} \Bigg|\mathbb{E}_{\mathbf{x}_V \in G^{V}} \Bigg[ \big(A \circ s(\mathbf{x}_{e_i}) - \delta\big) \prod_{j = i+1}^{|F|}{A \circ s(\mathbf{x}_{e_j})} \Bigg]\Bigg| \\
    &\hspace{5mm}\leq \sum_{i = 1}^{|F|} \mathbb{E}_{\mathbf{x}_{V \setminus f_i}} \Bigg|\mathbb{E}_{\mathbf{x}_{f_i}} \Bigg[ \big(A \big( s(\mathbf{x}_{f_i}) + s(\mathbf{x}_{e_i \setminus f_i}) \big) - \delta\big)
    \prod_{j = i+1}^{|F|} A \big( s(\mathbf{x}_{e_j \cap f_i}) + s(\mathbf{x}_{e_j \setminus f_i}) \big) \Bigg]\Bigg|.
\end{align*}

Consider the $i$-th term in the last sum.
For a fixed $\mathbf{x}_{V \setminus f_i} \in G^{V \setminus f_i}$ and each $i+1 \leq j \leq |F|$, define on $G^{e_j \cap f_i}$ the function $u_j = u_{j, \mathbf{x}_{V \setminus f_i}} := T^{s(\mathbf{x}_{e_j \setminus f_i})}A \circ s$;
the last sum then becomes
\begin{align*}
    \sum_{i = 1}^{|F|} \mathbb{E}_{\mathbf{x}_{V \setminus f_i}} \Bigg|\mathbb{E}_{\mathbf{x}_{f_i}} \Bigg[ \big(T^{s(\mathbf{x}_{e_i \setminus f_i})}A
    &\circ s(\mathbf{x}_{f_i}) - \delta\big) \prod_{j = i+1}^{|F|} u_j(\mathbf{x}_{e_j \cap f_i}) \Bigg]\Bigg| \\
    &\leq \sum_{i = 1}^{|F|} \mathbb{E}_{\mathbf{x}_{V \setminus f_i}} \big\| T^{s(\mathbf{x}_{e_i \setminus f_i})}A \circ s - \delta \big\|_{\square^{d}_{d-1}} \\
    &= |F| \cdot \| A \circ s - \delta \|_{\square^{d}_{d-1}}.
\end{align*}
Item $(ii)$ now follows from the fact that the octahedral norm is stronger than the cut norm (Lemma \ref{cut<oct}), since
$$\| A \circ s - \delta \|_{\square^{d}_{d-1}} = \| H^{(d)}\!A - \delta \|_{\square^{d}_{d-1}} \leq \| H^{(d)}A - \delta \|_{\oct^{d}} = \| A - \delta \|_{U^{d}} \leq c_1,$$
and we may take $c_2 = c_1$.

\medskip
$(ii) \Rightarrow (iii)$:
This is a special case, and we may take $c_3 = 2^{d}c_2$.

\medskip
$(iii) \Rightarrow (iv)$:
First we note that $t(\oct^{(k)}_{d}, H^{(k)}f) = t(\oct^{(d)}, H^{(d)}f)$ holds for all functions $f: G \rightarrow \R$.
Indeed, we have that
\begin{align*}
    t\big(\oct^{(k)}_{d}, H^{(k)}f\big) &= \mathbb{E}_{\mathbf{y} \in G^{k-d}} \mathbb{E}_{\mathbf{x}^{(0)}, \mathbf{x}^{(1)} \in G^{d}} \Bigg[ \prod_{\omega \in \{0, 1\}^{d}} f \big( s(\mathbf{y}) + s \big(\mathbf{x}^{(\omega)}\big) \big) \Bigg] \\
    &= \mathbb{E}_{\mathbf{y} \in G^{k-d}} \big[ t\big( \oct^{(d)}, H^{(d)}T^{s(\mathbf{y})}f\big) \big] \\
    &= t\big(\oct^{(d)}, H^{(d)}f\big).
\end{align*}
Item $(iii)$ is then the same as requiring that $t(\oct^{(d)}, H^{(d)}\!A) \leq \delta^{2^d} + c_3$.
By the equivalence theorem for strong quasirandomness (Theorem \ref{octequiv}) and its proof, we conclude that $t(\oct^{(d)}, H^{(d)}\!A - \delta) \leq 2^{2^{d+1}} c_3^{1/2^d}$;
using the identity above for $f = A - \delta$, this is the same as saying that $t(\oct^{(k)}_{d}, H^{(k)}A - \delta) \leq 2^{2^{d+1}} c_3^{1/2^d}$, which is exactly item $(iv)$ with $c_4 = 2^{2^{d+1}} c_3^{1/2^d}$.

\medskip
$(iv) \Rightarrow (i)$:
As discussed in the previous equivalence, item $(iv)$ is the same as requiring that $t(\oct^{(d)}, H^{(d)}\!A - \delta) \leq c_4$.
Since
$$t\big(\oct^{(d)}, H^{(d)}\!A - \delta\big) = \|H^{(d)}\!A - \delta\|_{\oct^d}^{2^d} = \|A - \delta\|_{U^d}^{2^d},$$
we obtain item $(i)$ with $c_1 = c_4^{1/2^d}$.
\end{proof}

This last result is due to Castro-Silva \cite{CayleyTypeHypergraphs}, and nicely illustrates one way in which the notion of quasirandomness
of order $d$ for hypergraphs differs from that of uniformity of degree $d$ for additive sets
(compare it with Theorem \ref{Towsnerthm} for hypergraph quasirandomness of order $d$).

The property of having small $d$-deviation (item $(iv)$ in the last theorem) was also studied in the general hypergraph setting by Chung \cite{QuasiClasses, QuasiClassesRevisited}.
Among other results, Chung claimed that this property was asymptotically equivalent to some given notions of discrepancy;
unfortunately, as explained by Lenz and Mubayi \cite{PosetQuasirandomness}, the proofs presented for one of the directions of equivalence contained a mistake, and it turns out that the claimed equivalences were incorrect.

Inspired by the case of Cayley hypergraphs we will next give a new notion of hypergraph quasirandomness which mimics the one induced by uniformity of degree $d-1$,
in particular obtaining quasirandom properties which are polynomially equivalent to having small $d$-deviation.

\subsection{Quasirandomness from counting $d$-simple hypergraphs} \label{NewQuasirandomSection}

We saw in the last theorem that Cayley hypergraphs of sets which are uniform of degree $d-1$ contain the expected number of all $d$-simple hypergraphs, even though they are only guaranteed to be quasirandom of order $d-1$.
The ability to correctly count all $d$-simple hypergraphs is due to some extra `symmetries' satisfied by Cayley hypergraphs, which we now describe.

Suppose we are given a $k$-graph $H$ on vertex set $V$ and a $t$-tuple of vertices $\mathbf{y} = (y_1, \dots, y_t) \in V^{t}$, $1 \leq t < k$.
We define the \emph{link of $H$ at $\mathbf{y}$} as the $(k-t)$-graph $H_{\mathbf{y}}$ corresponding to all sets of vertices which, together with $y_1, \dots, y_t$, form an edge of $H$;
more precisely, the vertex set of $H_{\mathbf{y}}$ is $V$ and its edges are all sets $\{x_1, \dots, x_{k-t}\} \subseteq V$ such that
$\{x_1, \dots, x_{k-t}, y_1, \dots, y_t\} \in H$.

The links of a Cayley hypergraph $H^{(k)}\!A$ are all quite similar to each other, and in fact can also be written as Cayley hypergraphs of translates of the original set $A$:
for all $\mathbf{y} \in G^t$ we have $(H^{(k)}\!A)_{\mathbf{y}} = H^{(k-t)}(T^{s(\mathbf{y})}\!A)$.
Hypergraphs induced by translates of the same set behave very similarly, in particular having a similar cut structure which allows one to control the count of subhypergraphs;
recall that $\|H^{(k)} T^{a}f\|_{\square^k_d} = \|H^{(k)}f\|_{\square^k_d}$ holds for all elements $a \in G$ and functions $f: G \rightarrow \R$.

If $A \subseteq G$ is uniform of degree $d-1$ and $k > d$, then all link $d$-graphs $(H^{(k)}\!A)_{\mathbf{y}}$ (with $\mathbf{y} \in G^{k-d}$) are strongly quasirandom and have the same density, which is what we actually used in the proof that $H^{(k)}\!A$ contains the correct count of $d$-simple hypergraphs.
We will now show that such a property, which can be seen as a new notion of quasirandomness for hypergraphs, is in fact \emph{necessary and sufficient} for correctly counting all $d$-simple hypergraphs
(or even for counting only $\oct^{(k)}_{d}$):

\begin{thm} \label{d-deviationThm}
Let $1 \leq d \leq k$ be integers and let $H$ be a $k$-uniform hypergraph with edge density $\delta$.
Then the following properties are polynomially equivalent:
\begin{itemize}
    \item[$(i)$] For all but at most $c_1 v(H)^{k-d}$ tuples $\mathbf{y} = (y_1, \dots, y_{k-d}) \in V^{k-d}$, the hypergraph $H_{\mathbf{y}}$ is strongly $c_1$-quasirandom and has edge density $\delta \pm c_1$.
    \item[$(ii)$] $H$ correctly counts all $d$-simple hypergraphs:
    $$t(F,\, H) = \delta^{|F|} \pm |F|c_2 \hspace{5mm} \forall F \in \mathcal{S}^{(k)}_{d}.$$
    \item[$(iii)$] $H$ has few squashed octahedra $\oct^{(k)}_{d}$:
    $$t(\oct^{(k)}_{d},\, H) \leq \delta^{2^{d}} + c_3.$$
    \item[$(iv)$] $H$ has small $d$-deviation:
    $$\hspace{-1cm} \mathbb{E}_{\mathbf{x}^{(0)}, \mathbf{x}^{(1)} \in V^d} \,\mathbb{E}_{y_{d+1}, \dots, y_k \in V} \Bigg[ \prod_{\omega \in \{0, 1\}^d} \big( H\big(x_1^{(\omega_1)}, \dots, x_d^{(\omega_d)}, y_{d+1}, \dots, y_k\big) - \delta \big) \Bigg] \leq c_4.$$
\end{itemize}
\end{thm}

Note that when $d = 1$ the condition of `quasirandomness of order 0' in the first item is trivially satisfied, and when $d = k$ this result is equivalent to Theorem \ref{octequiv} on strongly quasirandom properties of hypergraphs.

The theorem as stated above and the proof we will give below were taken from the author's paper \cite{CayleyTypeHypergraphs};
a qualitative version of this result can be also obtained from the main theorem of Towsner \cite{SigmaAlgebrasHypergraphs}, by considering the collection $\mathcal{I} = \left\{S \in \binom{k}{k-1}:\, [d] \subseteq S \right\}$.

\begin{proof}
$(i) \Rightarrow (ii)$:
Write $U$ for the vertex set of $F$ and $\{e_1, \dots, e_{|F|}\}$ for its edge set.
Proceeding as in the proof of the last theorem, we obtain that
$$\big| t(F, H) - \delta^{|F|} \big| \,\leq\, \sum_{i = 1}^{|F|} \mathbb{E}_{\mathbf{x}_{U \setminus f_i}} \Bigg|\mathbb{E}_{\mathbf{x}_{f_i}} \Bigg[ \big( H(\mathbf{x}_{e_i}) - \delta \big) \prod_{j = i+1}^{|F|}{H(\mathbf{x}_{e_j})} \Bigg]\Bigg|,$$
where each $f_i \subseteq e_i$ is a set of $d$ vertices in $U$ which is not completely contained in any other edge of $F$.

Let us consider the $i$-th term in the last sum.
For a fixed $\mathbf{x}_{U \setminus f_i} \in V^{U \setminus f_i}$ and each $i+1 \leq j \leq |F|$, define the function $u_j = u_{j, \mathbf{x}_{U \setminus f_i}}$ on $V^{e_j \cap f_i}$ by $u_j(\mathbf{x}_{e_j \cap f_i}) = H(\mathbf{x}_{e_j})$.
Since $|e_j \cap f_i| \leq d-1$ for $j \neq i$, this last sum is
\begin{align*}
    \sum_{i = 1}^{|F|} \mathbb{E}_{\mathbf{x}_{U \setminus f_i}} \Bigg|\mathbb{E}_{\mathbf{x}_{f_i}} \Bigg[ \big(H_{\mathbf{x}_{e_i \setminus f_i}}(\mathbf{x}_{f_i}) - \delta\big) &\prod_{j = i+1}^{|F|}{u_j(\mathbf{x}_{e_j \cap f_i})} \Bigg]\Bigg| \\
    &\leq \sum_{i = 1}^{|F|} \mathbb{E}_{\mathbf{x}_{U \setminus f_i}} \| H_{\mathbf{x}_{e_i \setminus f_i}} - \delta \|_{\square^d_{d-1}} \\
    &= |F| \cdot \mathbb{E}_{\mathbf{y} \in V^{k-d}} \left\| H_{\mathbf{y}} - \delta \right\|_{\square^d_{d-1}}.
\end{align*}
From the condition in item $(i)$ we have
$\mathbb{E}_{\mathbf{y} \in V^{k-d}} \left\| H_{\mathbf{y}} - \delta \right\|_{\square^d_{d-1}} \leq 3c_1$,
so we obtain item $(ii)$ with $c_2 = 3c_1$.

\medskip
$(ii) \Rightarrow (iii)$:
This is a special case, and we may take $c_3 = 2^{d}c_2$.

\medskip
$(iii) \Rightarrow (iv)$:
First we note that
\begin{align*}
    t(\oct^{(k)}_{d}, H) \,&=\, \mathbb{E}_{\mathbf{y} \in V^{k-d}} \,\mathbb{E}_{\mathbf{x}^{(0)}, \mathbf{x}^{(1)} \in V^{d}} \Bigg[ \prod_{\omega \in \{0, 1\}^{d}} H\big(\mathbf{x}^{(\omega)}, \mathbf{y}\big) \Bigg] \\
    &=\, \mathbb{E}_{\mathbf{y} \in V^{k-d}} \big[ t(\oct^{(d)}, H_{\mathbf{y}}) \big].
\end{align*}

By expanding the left-hand side into a sum and using Gowers-Cauchy-Schwarz as we did in the proof of Theorem \ref{octequiv} on strongly quasirandom hypergraphs, we see that for each fixed $\mathbf{y} \in V^{k-d}$ we have
$$t(\oct^{(d)}, H_{\mathbf{y}} + \delta) + t(\oct^{(d)}, H_{\mathbf{y}} - \delta) \leq \sum_{F \subseteq \oct^{(d)}} \|H_{\mathbf{y}}\|_{\oct^d}^{|F|} \big( \delta^{2^d - |F|} + (-\delta)^{2^d - |F|} \big)$$
Taking the expectation over all $\mathbf{y} \in V^{k-d}$ and then using convexity, we conclude that
\begin{align*}
    \mathbb{E}_{\mathbf{y}}\big[ t(\oct^{(d)}, H_{\mathbf{y}} + \delta) &+ t(\oct^{(d)}, H_{\mathbf{y}} - \delta) \big] \\
    &\leq \sum_{F \subseteq \oct^{(d)}} \mathbb{E}_{\mathbf{y}} \big[\|H_{\mathbf{y}}\|_{\oct^d}^{|F|}\big] \big( \delta^{2^d - |F|} + (-\delta)^{2^d - |F|} \big) \\
    &\leq \sum_{F \subseteq \oct^{(d)}} \mathbb{E}_{\mathbf{y}} \big[\|H_{\mathbf{y}}\|_{\oct^d}^{2^d}\big]^{|F|/2^d} \big( \delta^{2^d - |F|} + (-\delta)^{2^d - |F|} \big) \\
    &= \Big( \mathbb{E}_{\mathbf{y}} \big[\|H_{\mathbf{y}}\|_{\oct^d}^{2^d}\big]^{1/2^d} + \delta \Big)^{2^d} + \Big( \mathbb{E}_{\mathbf{y}} \big[\|H_{\mathbf{y}}\|_{\oct^d}^{2^d}\big]^{1/2^d} - \delta \Big)^{2^d},
\end{align*}
where the last equality follows from the binomial expansion.

Let us denote the edge density of each link hypergraph $H_{\mathbf{y}}$ by $\delta_{\mathbf{y}}$;
it is clear that $\mathbb{E}_{\mathbf{y} \in V^{k-d}} [\delta_{\mathbf{y}}] = \delta$.
Using convexity and our assumption from item $(iii)$ we have that
$$\delta \leq \mathbb{E}_{\mathbf{y}} \big[\delta_{\mathbf{y}}^{2^d}\big]^{1/2^d} \leq \mathbb{E}_{\mathbf{y}} \big[ \|H_{\mathbf{y}}\|_{\oct^d}^{2^d} \big]^{1/2^d} \leq \delta + c_3^{1/2^d},$$
and
$$\mathbb{E}_{\mathbf{y}} \big[t(\oct^{(d)}, H_{\mathbf{y}} + \delta)\big] = \mathbb{E}_{\mathbf{y}} \big[\|H_{\mathbf{y}} + \delta\|_{\oct^d}^{2^d}\big] \geq \mathbb{E}_{\mathbf{y}} \big[(\delta_{\mathbf{y}} + \delta)^{2^d}\big] \geq (2\delta)^{2^d}.$$

Taking stock of everything, we conclude that
$$\mathbb{E}_{\mathbf{y}} \big[t(\oct^{(d)}, H_{\mathbf{y}} - \delta)\big] + (2\delta)^{2^d} \leq \big( 2\delta + c_3^{1/2^d} \big)^{2^d} + c_3.$$
By a simple computation this implies
$$t(\oct^{(k)}_{d}, H - \delta) = \mathbb{E}_{\mathbf{y}} \big[t(\oct^{(d)}, H_{\mathbf{y}} - \delta)\big] \leq 2^{2^{d+1}} c_3^{1/2^d},$$
which is item $(iv)$ with constant $c_4 = 2^{2^{d+1}} c_3^{1/2^d}$.

\medskip
$(iv) \Rightarrow (i)$:
As noted before, the $d$-deviation of $H$ can be written as
$$t(\oct^{(k)}_{d}, H - \delta) = \mathbb{E}_{\mathbf{y} \in V^{k-d}} \big[t(\oct^{(d)}, H_{\mathbf{y}} - \delta)\big] = \mathbb{E}_{\mathbf{y} \in V^{k-d}} \big[ \|H_{\mathbf{y}} - \delta\|_{\oct^d}^{2^d} \big].$$
If this is at most $c_4$, then at most $c_4^{1/2} |V|^{k-d}$ tuples $\mathbf{y} \in V^{k-d}$ can satisfy the inequality $\|H_{\mathbf{y}} - \delta\|_{\oct^d}^{2^d} \geq c_4^{1/2}$.
For all other choices of $\mathbf{y} \in V^{k-d}$ we have
$$\|H_{\mathbf{y}} - \delta\|_{\square^d_{d-1}} \leq \|H_{\mathbf{y}} - \delta\|_{\oct^d} \leq c_4^{1/2^{d+1}}$$
(where we have used Lemma \ref{cut<oct} for the first inequality);
this implies item $(i)$ with constant $c_1 = c_4^{1/2^{d+1}}$.
\end{proof}


\subsection{Generalized Cayley hypergraphs} \label{generalCayley}

It is possible (and also useful) to study `Cayley-type' hypergraphs in greater generality, thus considering hypergraphs associated to an additive set via any given system $\Phi$ of linear forms.
Our methods from Section \ref{hypergraphs} will then allow us to count the number of such configurations inside suitably uniform additive sets, which is an important statistic to have for several applications.

We will start by studying the simpler case where $\Phi = \{\phi\}$ consists of only one linear form, generalizing the results of Section \ref{set-hypergraph} to hypergraphs associated to an additive set $A$ by means of any given linear form $\phi$.
The more general case of multiple linear forms will be analyzed in the next subsection.

We formally define a \emph{linear form} $\phi: G^k \rightarrow G$ as any map of the type
$$(x_1, \dots, x_k) \,\mapsto\, \lambda_1 x_1 + \dots + \lambda_k x_k,$$
where $\lambda_1, \dots, \lambda_k$ are integers.
Since such forms are not necessarily symmetric in their variables, we will need to consider $k$-partite hypergraphs associated to them;
the notion of quasirandomness in this case was defined in Section \ref{partiteHypergraphs}.

\begin{definition}
Given a linear form $\phi: G^k \rightarrow G$ and a set $A \subseteq G$, we define the $k$-partite $k$-graph $H_{\phi}A$ as follows:
each vertex class $V_i$, $1 \leq i \leq k$, is a copy of $G$ and $\{x_1, \dots, x_k\} \in \prod_{i\leq k} V_i$ is an edge of $H_{\phi}A$ if $\phi(x_1, \dots, x_k) \in A$.
\end{definition}

It is then possible to obtain a result relating quasirandomness of some order $d$ of $H_{\phi}A$ to uniformity of the same degree $d$ for the set $A$, as we did in Theorem \ref{AignerHan} for usual Cayley hypergraphs.
The possibility of divisibility issues caused by the form $\phi$ in the group considered, however, makes both the statement and the proof of such a result somewhat more complicated than those of Theorem \ref{AignerHan}.

\begin{thm} \label{unifQsrForms}
Let $G$ be a finite additive group and $A \subseteq G$ be a subset.
\begin{itemize}
    \item[$a)$] If $A$ is $\varepsilon$-uniform of degree $d$, then for any $k \geq d+1$ and any linear form $\phi(x_1, \dots, x_k) = \lambda_1 x_1 + \dots + \lambda_k x_k$ the hypergraph $H_{\phi}A$ is $\varepsilon'$-quasirandom of order $d$, where
    $$\varepsilon' = \varepsilon \Bigg( \prod_{i=1}^{d+1} \frac{|G|}{|\lambda_i G|} \Bigg)^{1/2^{d}}.$$
    \item[$b)$] Conversely, if $\phi: G^{d+1} \rightarrow G$ is surjective and $H_{\phi}A$ is $\varepsilon$-quasirandom of order $d$, then $A$ is $\varepsilon'$-uniform of degree $d$ with $\varepsilon' = 2 \varepsilon^{1/2^{d+1}}$.
\end{itemize}
\end{thm}

We refer the reader to \cite{CayleyTypeHypergraphs} for the proof of this result, which proceeds via a more careful analysis of the arguments used in our proof of Theorem \ref{AignerHan}.

An interesting (if somewhat undesirable) aspect of this last theorem is the \emph{asymmetry} in the conditions required in each item:
for item $a)$ we require the subgroups $\lambda_i G$ to not be much smaller than $G$ itself (otherwise the conclusion still holds but is trivial), while for item $b)$ we instead require the linear form $\phi$ to be surjective.
The next simple examples, both easy to generalize, show that this asymmetry is not a defect of the proof but is in fact \emph{necessary}:

\begin{ex} \label{aexample}
    Let $G = \F_2^n \oplus \F_3^n$ and define the linear form $\phi: G^2 \rightarrow G$ by
    $$\phi(x_1, x_2) = 3x_1 + 2x_2;$$
    this form is clearly surjective.
    Let $R \subset \F_2^n$ be a very linearly uniform set of density $1/2$ (for instance $R$ can be a set of $2^{n-1}$ elements of $\F_2^n$ chosen uniformly at random),
    and consider the set
    $$A := R \oplus \F_3^n = \{(a, b) \in \F_2^n \oplus \F_3^n:\, a \in R\}.$$
    This set has density $1/2$, and it is easy to see that $3x_1 + 2x_2 \in A$ if and only if $x_1 \in A$.
    The bipartite graph $H_{\phi}A$ will then have edge density $1/2$ and satisfies
    $$\|H_{\phi}A - 1/2\|_{\square^2_1} \geq
    \mathbb{E}_{x_1, x_2 \in G}\big[ (A(3x_1 + 2x_2) - 1/2) A(x_1) \big] = 1/4;$$
    $H_{\phi}A$ is therefore \emph{not} quasirandom (of order $1$).
    
    We now compute the Fourier coefficients of $A$.
    Any character $\gamma \in \widehat{G}$ can be decomposed as $\gamma = \gamma_1 \cdot \gamma_2$ for some $\gamma_1 \in \widehat{\F_2^n}$ and $\gamma_2 \in \widehat{\F_3^n}$;
    thus
    \begin{align*}
        |\widehat{A}(\gamma)|
        &= \big|\mathbb{E}_{(a, b) \in \F_2^n \oplus \F_3^n}\big[ A((a, b)) \overline{\gamma_1(a)} \overline{\gamma_2(b)} \big]\big| \\
        &= \big|\mathbb{E}_{a \in \F_2^n}\big[ R(a) \overline{\gamma_1(a)} \big]\big|\, \big|\mathbb{E}_{b \in \F_3^n}\big[ \gamma_2(b) \big]\big|.
    \end{align*}
    If $\gamma_2$ is \emph{not} the trivial character $\mathbf{1}$, then this last expression is zero.
    If $\gamma_2 = \mathbf{1}$ then it is equal to $|\widehat{R}(\gamma_1)|$, which by uniformity of $R$ will be $o(1)$ whenever $\gamma_1 \neq \mathbf{1}$.
    Thus $|\widehat{A}(\gamma)| = o(1)$ whenever $\gamma$ is a non-trivial character, showing that $A$ is very linearly uniform.
\end{ex}

\begin{ex} \label{bexample}
    Let $G = \Z_2 \oplus \Z_N$ with $N$ being a large odd integer, and define the linear form $\phi: G^{d+1} \rightarrow G$ by
    $$\phi(x_1, x_2, \dots, x_{d+1}) = 2x_1 + 2x_2 + \dots + 2x_{d+1}.$$
    This form is not surjective, but the subgroup $2G$ is quite `large': $|2G|/|G| = 1/2$.
    Take a set $R \subset \Z_N$ which is very uniform of degree $d$, and consider
    $$A := \{0\} \oplus 2R = \{(0, 2y) \in \Z_2 \oplus \Z_N:\, y \in R\}.$$
    The hypergraph $H_{\phi}A$ will have the same density as $R$ and will be very quasirandom of order $d$, but since $A$ is concentrated on the subgroup $\{0\} \oplus \Z_N$ it will \emph{not} be uniform of degree $d$ (or even linearly uniform).
\end{ex}

The simplest way to do away with the divisibility issues illustrated in these last two examples is to assume that all coefficients of the linear form considered are coprime to the order $|G|$ of the group;
if this happens we say that $\phi$ is \emph{coprime}.
For such linear forms we then obtain in item $a)$ of the last theorem that $H_{\phi}A$ is $\varepsilon$-quasirandom of order $d$ whenever $A$ is $\varepsilon$-uniform of degree $d$;
and in item $b)$ that $A$ is $(2 \varepsilon^{1/2^{d+1}})$-uniform of degree $d$ whenever $H_{\phi}A$ is $\varepsilon$-quasirandom of order $d$.

Under this same assumption we can also easily generalize our quasirandom equivalence theorem for Cayley hypergraphs:

\begin{thm}
Let $A \subseteq G$ be a set of density $\delta$ in $G$, $1 < d \leq k$ be integers and let $\phi: G^k \rightarrow G$ be a coprime linear form.
Then the following statements are polynomially equivalent:
\begin{itemize}
    \item[$(i)$] $A$ is uniform of degree $d-1$: \hspace{3mm}
    $\| A - \delta \|_{U^{d}} \leq c_1$.
    \item[$(ii)$] $H_{\phi}A$ correctly counts all $k$-partite $d$-simple hypergraphs:
    $$t_{can}(F,\, H_{\phi}A) = \delta^{|F|} \pm c_2 |F| \hspace{5mm} \forall F \in \mathcal{S}^{(k)}_{d} \hspace{2mm} k\text{-partite}.$$
    \item[$(iii)$] $H_{\phi}A$ has few copies of $\oct^{(k)}_{d}$:
    $$t_{can}(\oct^{(k)}_{d},\, H_{\phi}A) \leq \delta^{2^{d}} + c_3.$$
    \item[$(iv)$] $H_{\phi}A$ has small $d$-deviation:
    $$\hspace{-1cm} \mathbb{E}_{\mathbf{x}^{(0)}, \mathbf{x}^{(1)} \in G^d} \,\mathbb{E}_{y_{d+1}, \dots, y_k \in G} \Bigg[ \prod_{\omega \in \{0, 1\}^d} \big( H_{\phi}A \big(x_1^{(\omega_1)}, \dots, x_d^{(\omega_d)}, y_{d+1}, \dots, y_k\big) - \delta \big) \Bigg] \leq c_4.$$
\end{itemize}
\end{thm}

\begin{proof}
Let $\lambda_1, \dots, \lambda_k \in \Z$ be the integer coefficients of $\phi$, so that
$$\phi(x_1, \dots, x_k) = \lambda_1 x_1 + \dots + \lambda_k x_k.$$
By assumption each $\lambda_i$ is coprime with $|G|$, implying that the maps $x \mapsto \lambda_i x$ are all bijective on $G$.

Relabel each vertex class $V_i = G$ of $H_{\phi}A$ using this bijection, i.e. each vertex $v \in V_i$ initially labeled $a \in G$ should be relabeled $\lambda_i a \in G$.
This new hypergraph, isomorphic to $H_{\phi}A$, is exactly the $k$-partite Cayley sum hypergraph $H_s A$:
$\{y_1, \dots, y_k\} \in \prod_{i \leq k} V_i$ is an edge of $H_s A$ iff $y_1 + \dots + y_k \in A$.
The rest of the proof is essentially identical to that of Theorem \ref{Cayley_equiv}.
\end{proof}

\subsection{Linear configurations in additive sets} \label{LinearConfigSection}

We next consider in full generality Cayley-type hypergraphs associated to an additive set via any given system of linear forms $\Phi = (\phi_1, \dots, \phi_m)$.

For a linear form $\phi$ in $k$ variables $x_1, \dots, x_k$, we define the \emph{support} of $\phi$ as the set of indices $i$ such that $\phi$ depends on $x_i$;
that is, if $\phi(x_1, \dots, x_k) = \lambda_1 x_1 + \dots + \lambda_k x_k$, then the support of $\phi$ is $\{i \in [k]: \lambda_i \neq 0\}$.

\begin{definition}
Let $\Phi = (\phi_1, \dots, \phi_m): G^k \rightarrow G^m$ be a system of $m$ linear forms on $k$ variables, and denote the support of $\phi_i$ by $\sigma_i$ for each $1 \leq i \leq m$.
We define the \emph{support hypergraph of $\Phi$}, denoted $\Sigma_{\Phi}$, as the hypergraph with vertex set $[k]$ and edge set $\{\sigma_i: i \in [m]\}$.
\end{definition}

Note that the support hypergraph of a linear system may have edges of several different sizes;
it needs not be a uniform hypergraph as we have considered up to now.
The same remark holds for their associated Cayley-type hypergraphs:

\begin{definition}
Let $\Phi = (\phi_1, \dots, \phi_m): G^k \rightarrow G^m$ be a system of linear forms, and denote the support of each form $\phi_i$ by $\sigma_i$.
Given a set $A \subseteq G$, we define the \emph{generalized Cayley hypergraph} $H_{\Phi} A$ as the $k$-partite hypergraph with vertex classes $V_1, \dots, V_k$, where each $V_i$ is a copy of $G$, and with edge set
$$H_{\Phi} A = \bigcup_{i = 1}^m \big\{\mathbf{x}_{\sigma_i} \in V_{\sigma_i}: \phi_i(\mathbf{x}) \in A\big\}.$$
In other words, we have a copy $V_j$ of $G$ corresponding to each variable $x_j$, and each linear form $\phi_i \in \Phi$ with support $\sigma_i$ induces on $\prod_{\ell \in \sigma_i} V_{\ell}$ a hypergraph which satisfies the relation $H_{\Phi}A(\mathbf{x}_{\sigma_i}) = A(\phi_i(\mathbf{x}))$.
\end{definition}


Our interest in these constructions
comes from the following simple identity, which is a direct consequence of the definitions:
$$t_{can}(\Sigma_{\Phi}, H_{\Phi}A) = \mathbb{E}_{\mathbf{x} \in G^k} \Bigg[ \prod_{i=1}^m A(\phi_i(\mathbf{x})) \Bigg].$$
We are then able to count linear configurations inside uniform additive sets by using the hypergraph-theoretical tools seen in the last section.

This motivates the following definition, which was
introduced\footnote{Green and Tao did not explicitly deal with hypergraphs, so their definition is differently worded but essentially equivalent.}
by Green and Tao \cite{LinearEquationsPrimes} in the context of studying linear equations in primes.
Intuitively, it represents which systems we can hope to control using our methods for uniformity of degree $s$.

\begin{definition}
Let $\Phi$ be a system of linear forms and $s \geq 0$ be an integer.
We say that $\Phi$ is in \emph{$s$-normal form} if its support hypergraph $\Sigma_{\Phi}$ is $(s+1)$-simple.
\end{definition}

We will soon give several examples of linear systems which are in some $s$-normal form, but first it is important to remark on one
characteristic of this definition:
being in $s$-normal form is a property not only of the arithmetic structure of the linear system $\Phi$ in consideration, but also of its \emph{representation}.
In other words, by a simple change of variables it is possible to change the `degree' $s$ for which $\Phi$ is in normal form, or even to make it cease being in any normal form whatsoever.

It is then important to consider distinct formulations of linear systems, akin to how one can consider different labellings of a graph.
This leads us to the notion of \emph{equivalent systems}:
we say that two linear systems $\Phi: G^k \rightarrow G^m$ and $\Phi': G^{k'} \rightarrow G^m$ are \emph{equivalent} if $\Phi(\mathbf{x})$ has the same distribution as $\Phi'(\mathbf{y})$ when $\mathbf{x}$ and $\mathbf{y}$ are uniformly distributed on $G^k$ and $G^{k'}$, respectively.
(This means that
$|\Phi^{-1}(S)|/|G^k| = |\Phi'^{-1}(S)|/|G^{k'}|$
for all subsets $S \subseteq G^m$.)

\begin{ex}
    Consider the examples from Section \ref{countLinearGraph}, namely 3-term arithmetic progressions $(x,\, x+r,\, x+2r)$, additive quadruples $(x,\, x + h_1,\, x + h_2,\, x + h_1 + h_2)$ and Schur triples $(x,\, y,\, x+y)$.
    None of these linear systems are in $s$-normal form for any $s$;
    however, they are equivalent to the systems $(-2x_1 - x_2,\, -x_1 + x_3,\, x_2 + 2x_3)$, $(x_1 + x_2,\, x_1 + x_2',\, x_1' + x_2,\, x_1' + x_2')$ and $(x_1 - x_2,\, x_3 - x_1,\, x_3 - x_2)$ respectively, which are all in $1$-normal form.
\end{ex}

\begin{ex}
    Consider the systems
    \begin{align*}
        &\Phi_1(x,\, r) \,=\, \big(x,\, x+r,\, \dots,\, x + (k-1)r\big), \\
        &\Phi_2(x,\, h_1,\, \dots,\, h_k) \,=\, \bigg( x + \sum_{i = 1}^k \omega_i h_i \bigg)_{\omega \in \{0, 1\}^k}
    \end{align*}
    corresponding to $k$-term arithmetic progressions and $k$-dimensional parallelepipeds respectively.
    They are not in $s$-normal form for any $s$, but they are equivalent to the systems
    \begin{align*}
        &\Phi'_1(x_1,\, \dots,\, x_k) \,=\, \bigg( \sum_{j = 1}^k (i - j) x_j \bigg)_{i \in [k]} \text{ and} \\
        &\Phi'_2(x_1,\, \dots,\, x_k,\, x_1',\, \dots,\, x_k') \,=\, \bigg( \sum_{i=1}^k \big( (1 - \omega_i) x_i + \omega_i x_i' \big) \bigg)_{\omega \in \{0, 1\}^k}
    \end{align*}
    which are in $(k-2)$-normal form and in $(k-1)$-normal form, respectively.
\end{ex}

In order to avoid unwanted linear dependencies due to divisibility issues in the group
(as illustrated in Examples \ref{aexample} and \ref{bexample}),
we shall from now on deal only with \emph{coprime} linear systems, meaning those where each integer coefficient in each linear form is either zero or coprime with the order $|G|$ of the group.

The next theorem is a special case of an important result by Green and Tao \cite{LinearEquationsPrimes}, which they called a `generalized von Neumann theorem'.
Due to their need to work with unbounded functions
(or rather with functions which are bounded by a pseudorandom measure $\nu$ instead of the constant function $1$),
the proof presented in their paper is fairly complicated;
by dealing only with bounded functions we are able to rely on our usual hypergraph-theoretic methods which greatly simplify the proof.

\begin{thm}[Counting lemma for systems in $s$-normal form] \label{control_snormalthm}
Let $\Phi: G^k \rightarrow G^m$ be a system of linear forms in $s$-normal form, and suppose that $\Phi$ is coprime.
Then for any functions $f_1$, $\dots$, $f_m$, $g_1$, $\dots$, $g_m: G \rightarrow [0, 1]$ we have
\begin{equation} \label{control_snormal}
    \Bigg| \mathbb{E}_{\mathbf{x} \in G^k} \Bigg[ \prod_{i=1}^m f_i(\phi_i(\mathbf{x})) \Bigg] - \mathbb{E}_{\mathbf{x} \in G^k} \Bigg[ \prod_{i=1}^m g_i(\phi_i(\mathbf{x})) \Bigg] \Bigg| \leq \sum_{i=1}^m \|f_i - g_i\|_{U^{s+1}}.
\end{equation}
\end{thm}

\begin{proof}
As usual, we use a telescoping sum and the triangle inequality to bound the left-hand side of the expression above by
$$\sum_{i=1}^{m} \Bigg| \mathbb{E}_{\mathbf{x} \in G^k} \Bigg[ \prod_{j=1}^{i-1} g_j(\phi_j(\mathbf{x})) \cdot \big( f_i(\phi_i(\mathbf{x})) - g_i(\phi_i(\mathbf{x})) \big) \cdot \prod_{j=i+1}^m f_j(\phi_j(\mathbf{x})) \Bigg] \Bigg|.$$
It then suffices to show that the $i$-th term in this sum is bounded by $\|f_i - g_i\|_{U^{s+1}}$ for each $1 \leq i \leq m$.
For notational convenience we will prove this for $i = 1$;
the argument is the same for any other choice of $i \in [m]$.

Denote the support of each $\phi_j \in \Phi$ by $\sigma_j$, and let $\tau_1 \subseteq \sigma_1$ be a set of $s+1$ indices which is not contained in any $\sigma_j$ for $j > 1$.
For each $j \in [m]$, decompose the linear form $\phi_j$ into its component $\phi_j'$ which uses the variables indexed by $\tau_1$ and its component $\phi_j''$ which does not use any of these variables;
that is, we can write $\phi_j(\mathbf{x}) = \phi_j'(\mathbf{x}_{\sigma_j \cap \tau_1}) + \phi_j''(\mathbf{x}_{\sigma_j \setminus \tau_1})$.

We can then separate the dependence on $\mathbf{x}_{\tau_1}$ and $\mathbf{x}_{[k] \setminus \tau_1}$ in the first term of the sum above, obtaining
\begin{multline*}
    \Bigg| \mathbb{E}_{\mathbf{x}_{[k] \setminus \tau_1}} \mathbb{E}_{\mathbf{x}_{\tau_1}} \Bigg[ \Big( f_1 \big( \phi_1'(\mathbf{x}_{\tau_1}) + \phi_1''(\mathbf{x}_{\sigma_1 \setminus \tau_1}) \big) - g_1 \big( \phi_1'(\mathbf{x}_{\tau_1}) + \phi_1''(\mathbf{x}_{\sigma_1 \setminus \tau_1}) \big) \Big) \\
    \times \prod_{j=2}^m f_j \big( \phi_j'(\mathbf{x}_{\sigma_j \cap \tau_1}) + \phi_j''(\mathbf{x}_{\sigma_j \setminus \tau_1}) \big) \Bigg] \Bigg|.
\end{multline*}
Let us now fix a choice of $\mathbf{x}_{[k] \setminus \tau_1} \in G^{[k] \setminus \tau_1}$, and for each $j \in [m]$ we denote $y_j := \phi_j''(\mathbf{x}_{\sigma_j \setminus \tau_1})$ and $u_j := T^{y_j}f_j \circ \phi_j'$.
The inner expectation in the last expression is then equal to
$$\mathbb{E}_{\mathbf{x}_{\tau_1}} \Bigg[ \big( T^{y_1}f_1 (\phi_1'(\mathbf{x}_{\tau_1})) - T^{y_1}g_1 (\phi_1'(\mathbf{x}_{\tau_1})) \big) \prod_{j=2}^m u_j(\mathbf{x}_{\sigma_j \cap \tau_1}) \Bigg].$$
Since $0 \leq u_j \leq 1$ and $|\sigma_j \cap \tau_1| \leq s$ for all $j > 1$, by definition this has absolute value at most
\begin{align*}
    \|H_{\phi_1'}(T^{y_1}f_1 - T^{y_1}g_1)\|_{\square_s^{s+1}}
    &= \|H_{\phi_1'}(f_1 - g_1)\|_{\square_s^{s+1}} \\
    &\leq \|H_{\phi_1'}(f_1 - g_1)\|_{\oct^{s+1}} \\
    &= \|H^{(s+1)}(f_1 - g_1)\|_{\oct^{s+1}}
    = \|f_1 - g_1\|_{U^{s+1}},
\end{align*}
where in the second-to-last equality we have used that $\phi_1'$ is coprime.
The result follows by averaging over all choices of $\mathbf{x}_{[k] \setminus \tau_1} \in G^{[k] \setminus \tau_1}$.
\end{proof}

As an immediate corollary of this result, we obtain that
$$\mathbb{E}_{\mathbf{x} \in G^k} \Bigg[ \prod_{i=1}^m A(\psi_i(\mathbf{x})) \Bigg] = \delta^m \pm m \|A - \delta\|_{U^{s+1}}$$
whenever $A \subseteq G$ is a set and $\psi_1, \dots, \psi_m: G^k \rightarrow G$ are coprime \emph{affine-linear}
forms\footnote{An affine-linear form $\psi: G^k \rightarrow G$ is the sum $\psi = \phi + c$ of a linear form $\phi: G^k \rightarrow G$ (its linear part) with a constant term $c \in G$.}
whose linear part is in $s$-normal form;
the study of affine-linear systems taking values in the primes was the main motivation for Green and Tao to prove
(a stronger version of)
this counting lemma.
We also conclude the related estimate
$$\mathbb{E}_{\mathbf{x} \in G^k} \Bigg[ \prod_{i=1}^m A_i(\psi_i(\mathbf{x})) \Bigg] = \prod_{i=1}^m \delta_i \pm \sum_{i=1}^m \|A_i - \delta_i\|_{U^{s+1}},$$
where each individual term $\psi_i(\mathbf{x})$ is now required to lie on possibly distinct target sets $A_1, \dots, A_m \subseteq G$.
This shows it is possible to control the count of any linear configuration equivalent to a (coprime) system in $s$-normal form in terms of the Gowers $U^{s+1}$ norm
(note that equivalent systems give the same value for expressions like those appearing in inequality (\ref{control_snormal})).


    

The main issue with using this counting lemma to estimate the count of linear configurations is that the considered system $\Phi$ should be known to have an equivalent representation in $s$-normal form for some reasonable (and hopefully optimal) integer $s$, and it is not clear at first how to find such a value of $s$.
In the same paper \cite{LinearEquationsPrimes} where they defined the notion of $s$-normal form, Green and Tao provided a simple linear algebra recipe for computing such a sufficient value $s(\Phi)$,
which holds whenever there is an underlying field structure on the group $G$
(as in the two important cases $G = \F_p^n$ and $G = \Z_N$ with $N$ prime).

\begin{definition}
The \emph{Cauchy-Schwarz complexity} of a linear system $\Phi = \{\phi_1, \dots, \phi_m\}$ is the minimal integer $s$ such that the following holds.
For every $1 \leq i \leq m$, one can partition the $m - 1$ forms $\{\phi_j:\, j \in [m]\setminus\{i\}\}$ into $s + 1$ classes, so that $\phi_i$ does not lie in the linear span of any of these classes.
\end{definition}

We shall henceforth abbreviate the expression `Cauchy-Schwarz complexity' by `CS-complexity';
this name was coined by Gowers and Wolf \cite{TrueComplexity}, both to highlight the main tool employed when using this notion to count linear configurations, and to distinguish it from the notion of true complexity which we will see later.
The definition can be trivially modified in order to hold also for affine-linear systems, and all results presented here will continue to hold in this slightly greater generality.

The notion of CS-complexity implicitly assumes that the group $G$ under consideration has the linear structure of a vector space,
so that one can define the linear span of the classes considered.
For the rest of this section we shall then assume that $G = \F_p^n$ for some field $\F_p$ of
prime order\footnote{The assumption of prime order gives no loss of generality, since a field of order $q = p^r$ is (as an additive group) isomorphic to $\F_p^r$.}
and an integer $n \geq 1$;
note that the important cyclic case $G = \Z_N$ with $N$ prime is also of this form, with $p = N$ and dimension $n = 1$.

In such cases a linear form $\phi$ on $k$ variables can be written as
$$\phi(x_1, \dots, x_k) = c_1 x_1 + \dots + c_k x_k$$
with $c_1, \dots, c_k \in \F_p$, and so we may regard it as an element of $\F_p^k$:
$\phi \simeq (c_1, \dots, c_k)$.
Note that \emph{all} linear forms and systems in these groups considered are coprime.

It is easy to see that every system which is in $s$-normal form has CS-complexity at most $s$, but not every system of CS-complexity $s$ is in $s$-normal form.
However, as
shown\footnote{They proved this result for $G$ being the integers and using the field structure of the rationals, but the proof is essentially the same in the cases we consider here.}
by Green and Tao, every system of CS-complexity $s$ is \emph{equivalent} to a system in $s$-normal form:

\begin{lem}
Suppose $G = \F_p^n$ for some prime $p$ and integer $n \geq 1$, and let $\Phi: G^k \rightarrow G^m$ be a system of linear forms of CS-complexity $s$.
Then there exists a system $\Phi': G^{k'} \rightarrow G^m$ in $s$-normal form which is equivalent to $\Phi$, where $k' = k + m(s+1)$.
\end{lem}

\begin{proof}
Let us fix $i \in [m]$.
We will construct a system $\Phi': G^{k+s+1} \rightarrow G^m$ which is equivalent to $\Phi$ and which is in $s$-normal form at $\phi_i'$, in the sense that there is a set $\tau \subset [k+s+1]$ of size $s+1$ that is contained in the support of $\phi_i'$ but is not contained in the support of any other form $\phi_j' \in \Phi' \setminus \{\phi_i'\}$.
Applying this extension procedure once to each $i \in [m]$ will prove the result.

By hypothesis we can partition $[m] \setminus \{i\}$ into $s+1$ classes $A_1, \dots, A_{s+1}$ so that $\phi_i$ is \emph{not} in $\text{span}_{\F_p} \{\phi_j: j \in A_r\}$ for any $r \in [s+1]$.
Since each form $\phi_j$ can be seen as an element of $\F_p^k$, by basic linear algebra there exist vectors $f_1, \dots, f_{s+1} \in \F_p^k$ such that, for each $1 \leq r \leq s+1$, $\phi_i(f_r) \neq 0$ and yet $\phi_j (f_r) = 0$ for all $j \in A_r$.

Let $\phi_j': G^{k+s+1} \rightarrow G$ be the linear form given by
$$\phi_j'(\mathbf{x}_{[k]}, y_1, \dots, y_{s+1}) = \phi_j \bigg( \mathbf{x}_{[k]} + \sum_{r=1}^{s+1} f_r y_r \bigg),$$
where $\mathbf{x}_{[k]} \in G^k$, $y_r \in G$ for $r \in [s+1]$, and we write $f_r y_r$ for the element in $G^k$ whose $j$-th coordinate is $(f_r)_j y_r$ for each $j \in [k]$
(recall that $(f_r)_j \in \F_p$).
The coefficient of $\phi_j'$ at every variable $y_r$ is $\phi_j(f_r)$, which is zero if $j \in A_r$ and non-zero if $j=i$;
the form $\phi_i'$ is then the only one which contains all of the variables $y_1, \dots, y_{s+1}$ in its support, which is what we wanted.
\end{proof}

In light of this result, we see that any linear system $\Phi$ of CS-complexity at most $s$ in a group of the form $\F_p^n$ will be controlled by the Gowers $U^{s+1}$ norm, in the sense that the inequality
\begin{equation} \label{truecomplexity}
    \Bigg| \mathbb{E}_{\mathbf{x} \in G^k} \Bigg[ \prod_{i=1}^m A(\phi_i(\mathbf{x})) \Bigg] - \delta^m \Bigg| \leq m \|A - \delta\|_{U^{s+1}}
\end{equation}
holds for all sets $A \subseteq G$
(and in fact also the stronger inequality (\ref{control_snormal}) from Theorem \ref{control_snormalthm} holds).
Moreover, every linear system $\Phi: G^k \rightarrow G^m$ containing no two terms which are multiple of one another will clearly have CS-complexity at most $m-2$;
this justifies our assertion made in Section \ref{HigherDegUniformity} that every non-degenerate linear system is controlled by some $U^s$
norm.\footnote{This assertion is easily seen to be false when two terms of the system are linearly dependent.
In additive groups not of the form $\F_p^n$ there will be elements of multiple distinct orders, so one must also take some care with divisibility issues for a similar assertion to hold.}

But is this the best one can do?
The example of $(k+1)$-dimensional parallelepipeds and that of $(k+2)$-term arithmetic progressions show there are important cases for which the bound obtained with the CS-complexity is sharp:
both have CS-complexity $k$ and one can construct sets which are very uniform of degree $k-1$ but which contain neither of these patterns in the expected number.

However, as observed by Gowers and Wolf \cite{TrueComplexity}, the bound thus obtained is not always optimal:
for every $s \geq 2$ there are linear systems $\Phi$ of CS-complexity $s$ but for which uniformity of some degree $r < s$ is sufficient to control their count inside any set.
Since having small $U^s$ norm is a significantly stronger condition than that of having small $U^r$ norm for $r < s$, it is of interest to have the best result possible.

Gowers and Wolf then made a deep and beautiful conjecture on what the optimal degree of uniformity needed to control \emph{any} given linear system $\Phi$ is
(on groups of the form $G = \F_p^n$ as we are considering here).
They conjectured that this value,
which they called the \emph{true complexity} of the system,
is the smallest $d$ for which the $(d+1)$-th powers of the linear forms contained in $\Phi$ are linearly
independent.\footnote{More precisely, this means that the symmetric multilinear forms $\phi_1^{d+1}, \dots, \phi_m^{d+1}: (\F_p)^{d+1} \rightarrow \F_p$ given by $\phi_i^{d+1}(x_1, \dots, x_{d+1}) := \phi_i(x_1) \cdots \phi_i(x_{d+1})$ are linearly independent.}

The necessity of having \emph{at least} this degree of uniformity can be shown by generalizing our `quadratic' Example \ref{quad_ex}.
In that example we saw that the linear dependence between the squares of each term on a 4-term arithmetic progression allows for them to be more concentrated than one would expect in the quadratically structured set $\{x \in \mathbb{Z}_N:\, \|x^2/N\|_{\R/\Z} \leq \delta/2\}$, which is itself very linearly uniform.
A similar argument can be used to show concentration of linear configurations $\{\phi_1(x), \dots, \phi_m(x)\}$
whose $d$-th powers are linearly dependent
inside some sets exhibiting structure of degree $d$;
see Section 3.1 of \cite{TrueComplexity} for a more detailed discussion.

Gowers and Wolf's conjecture on the true complexity of linear systems has since been proven in most interesting cases.
On groups $\F_p^n$ of bounded characteristic $p$ this was established by Gowers and Wolf \cite{TrueComplexity, QuadraticUniformityFpn, HigherDegreeUniformityFpn} when the CS-complexity of the linear system is less than $p$,
and by Hatami, Hatami and Lovett \cite{GeneralSystemsComplexity} in the general case;
and on cyclic groups $\Z_N$ with $N$ prime this was proven by Gowers and Wolf \cite{QuadraticUniformityZn} when the CS-complexity is at most $2$, and by Green and Tao \cite{ArithmeticRegularity} for all linear systems satisfying some general condition called the flag property.\footnote{This is a technical condition regarding the linear subspaces spanned by successive powers of the linear system in consideration;
see the arXiv version \cite{ArithmeticRegularityArxiv} of Green and Tao's paper for its definition and for a discussion on why it is necessary in their proof.
We note that this condition is satisfied, for instance, by all translation-invariant linear systems such as arithmetic progressions and parallelepipeds.}

The proofs of these results are fairly difficult and proceed via techniques from higher-order Fourier analysis, being reliant in particular on the very deep
\emph{inverse theorems for the Gowers norms},
which characterize the arithmetic obstructions to uniformity of each degree.
It is natural to wonder whether they can be proven in a simpler way, by making use of hypergraph theoretical tools as we used throughout this section.
Other than giving a much simpler proof, such an argument would have the extra advantage of obtaining vastly improved \emph{quantitative} control on the count of linear configurations inside suitably uniform sets or functions.

A surprising example of Manners \cite{GoodBounds} in a way shows that such an approach cannot work.
Manners constructed, for each sufficiently large prime $N \equiv \pm 1 \mod 8$, a system $\Phi_N$ of six forms in three variables on $\Z_N$ which have the following properties:
\begin{itemize}
    \item their squares are linearly independent (so they have true complexity $1$);
    \item there exist (complex) functions $f_1, \dots, f_6$ on $\Z_N$, bounded in magnitude by $1$, such that $\|f_i\|_{U^2} \leq N^{-1/8}$, $1 \leq i \leq 6$, but
    $$\Bigg| \mathbb{E}_{x_1, x_2, x_3 \in \Z_N} \Bigg[ \prod_{i=1}^6 f_i(\phi_i(x_1, x_2, x_3)) \Bigg] \Bigg| \geq 10^{-12}.$$
\end{itemize}
In other words, even though each system $\Phi_N$ must be controlled by the $U^2$ norm, the quantitative control obtained must take into account also the specific \emph{coefficients} of the linear system in consideration, not only the powers for which the forms are linearly independent or the number of forms and variables.

Since all our hypergraph-theoretic methods invariably give bounds independent of the coefficients of the forms involved
(as long as they do not incur in divisibility issues on the group considered),
it follows that these methods cannot be used to prove that the systems $\Phi_N$ constructed really do have true complexity $1$.

One way of looking at this conclusion is to say that the true complexity of a linear system is a purely arithmetical result, which cannot be deduced from a coarser structure that can be encoded in Cayley-type hypergraphs representing such systems.
This beautifully illustrates another fundamental way in which the notions of quasirandomness for additive groups and for hypergraphs are different, despite their many connections to each other.

\section{Regularity lemmas} \label{regularity}

As was mentioned in the introduction, results on quasirandomness can be used to help analyze arbitrary objects (quasirandom or not) by making use of some decomposition theorems usually known in combinatorics as \emph{regularity lemmas}.
In this section we will give a quick exposition of such results in order to illustrate how this might be accomplished.

Before diving into the details, let us first give an intuitive idea of what these results say.
Regularity lemmas may be thought of as rough structure theorems, where we decompose an arbitrary object of a given type into either two or three terms:
\begin{itemize}
    \item The first component is a combination of a few simpler, highly structured objects which should be relatively easy to analyze directly.
    This is the `structured part' of our original object.
    \item The second component is quasirandom, in the way that we have studied throughout this paper, and its contribution to the statistics one is interested in may be estimated using the methods here presented.
    \item In order to obtain a better control on the first two components, it might be necessary to introduce a third component which is small in $L^2$ norm and may be thought of as a minor error term.
\end{itemize}

By virtue of being small, the error term will contribute little to the statistics we are interested in and may be readily discarded.\footnote{This is the main idea, but in practice one has to take some care as the error introduced by this term might swamp the other terms in some small part of the domain.
In applications it is necessary to first `localize' this error term and exploit some type of positivity condition on the terms or some equidistribution property of the substructures being counted.}
Using our methods and results on quasirandom objects, we can also disregard the contribution of the second term and see it as `random noise'.
We then reduce the analysis of our original (possibly very complex) object to its much simpler structured part.

We shall now give more details on the most important results of this type which were established in the three settings considered in this paper.

\subsection{Graph regularity} \label{GraphReg}

The first and most well-known of the regularity lemmas is the so-called \emph{Szemer\'edi regularity lemma}, which was obtained by Szemer\'edi \cite{SzemerediTheorem, RegularPartitionsGraphs} as a step in his celebrated proof of the Erd\H{o}s-Tur\'an conjecture.\footnote{This conjecture, now proven and known as Szemer\'edi's theorem, states that any set of integers with positive upper density contains arbitrarily long arithmetic progressions.}

This important result roughly asserts that the vertices of \textit{any} graph $G$ may be partitioned into a bounded number of equal-sized parts, in such a way that for almost all pairs of partition classes the bipartite graph between them is quasirandom.
Both the upper bound we get for the size of this partition and the quality of the quasirandomness behaviour of the graph between these pairs depend only on an accuracy parameter $\varepsilon > 0$ we are at liberty to choose.

This theorem has become a cornerstone of extremal combinatorics and has found a very large number of applications in both combinatorics and theoretical computer science (see the surveys \cite{SzemerediRegularityApplications, SzemerediRegularityApplications2, SzemerediRegularityQuasirandomness}).
In applications it is usually used together with a counting lemma, which is essentially the same as the one given in Section \ref{PartiteGraphs}.

As a final piece of notation before stating the regularity lemma, suppose we have a graph $G$ on vertex set $V$ and two disjoint subsets $U$, $W$ of $V$.
Recall that $G[U, W]$ denotes the bipartite graph on vertex set $U \cup W$ whose edges are the restriction of $E(G)$ to $U \times W$.
For a fixed parameter $\varepsilon > 0$, we then say that the pair $(U, W)$ is \emph{$\varepsilon$-regular for $G$} if the bipartite graph $G[U, W]$ is $\varepsilon$-quasirandom (as defined in Section \ref{PartiteGraphs}).

\begin{thm}[Szemer\'edi regularity lemma] \label{SzRL}
For every $\varepsilon > 0$ and $k_0 \geq 1$, there exists an integer $K_0$ such that the following holds.
Every graph $G = (V, E)$ admits a partition $\mathcal{P}: V = V_1 \cup V_2 \cup \cdots \cup V_k$ of its vertex set with the following properties:
\begin{itemize}
\item $k_0 \leq k \leq K_0$;
\item $\big| |V_i| - |V_j| \big| \leq 1$ for all $1 \leq i, j \leq k$;
\item all but at most $\varepsilon \binom{k}{2}$ of the pairs $(V_i, V_j)$ are $\varepsilon$-regular for $G$.
\end{itemize}
\end{thm}

Let us now see how this result fits into our general description given in the beginning of this section.
Here the object to be decomposed is the edge set of a given graph.
The structured component in this decomposition then represents the pairs $(V_i, V_j)$ of partition classes together with the density of edges between them, and it has a very simple `cut structure' which makes it easy to analyze.
The quasirandom component represents the actual edges between those pairs which are $\varepsilon$-quasirandom, and the small error term accounts for the $\varepsilon \binom{k}{2}$ pairs of partition classes which are not necessarily $\varepsilon$-regular.

The usual proof of the regularity lemma is not hard and proceeds by an `energy increment' argument, where one starts with an arbitrary partition of $V$ into $k_0$ parts and iteratively refines it while it doesn't satisfy the third condition of the statement.
This refinement is done using sets which `witness' the irregularity of those pairs which are not $\varepsilon$-regular, and by defining a suitable notion of energy of the partition (which is the average squared edge density between pairs of its classes) one can show that it must grow significantly in this refinement.
At each refinement step the number of parts in the partition will increase at most exponentially, while its energy increases by at least $\varepsilon^5/10$ (say);
since this energy is bounded between $0$ and $1$, in at most $10/\varepsilon^5$ steps the algorithm must stop and we obtain a regular partition of $V$ with a bounded number of classes.

The main issue with Szemer\'edi's regularity lemma, which severely limits its applications, is the very poor bound it gives for the maximal size $K_0$ of the promised partition:
the proof outlined above obtains a bound on $K_0$ which is given by an exponential tower of $2$s of height proportional to $1/\varepsilon^5$.
Somewhat amazingly, such terrible bounds cannot be avoided:
Gowers \cite{GowersLowerBound} constructed graphs for which the smallest vertex partition satisfying the requirements of the theorem
(and in fact even weaker requirements)
has size a tower of $2$s of height proportional to $1/\varepsilon^{1/16}$.
See also \cite{TightLowerBound} for a tight lower bound of $O(1/\varepsilon^2)$ on the tower height in a version the regularity lemma.

We note that there are several other variants of the regularity lemma for graphs, each one tailored to be be more useful for a specific application.
The most famous of these variants are Frieze and Kannan's `weak regularity lemma' \cite{FriezeKannan},
which has weaker regularity properties but much better bounds, and
a `strong regularity lemma' by Alon, Fischer, Krivelevich and Szegedy \cite{AFKS}, which gives stronger regularity properties but has a more complicated statement and even worse bounds.
For the statements and proofs of these results (and also other variants of the graph regularity lemma) we refer the interested reader to R\"odl and Schacht's survey \cite{RegularityGraphs}.

\subsection{Hypergraph regularity} \label{HyperReg}

It is possible to generalize the regularity lemmas seen from the setting of graphs to that of hypergraphs.
As when passing from the study of graph quasirandomness to that of hypergraph quasirandomness, this will require the introduction of a somewhat heavy notation and gives rise to several different results depending on which order of quasirandomness one considers.

The first version of regularity lemmas of each order $d$ for hypergraphs was developed by Chung \cite{RegularityHypergraphsQuasirandomness}.
Roughly speaking, this result states that one can partition the underlying structure $\binom{V}{d}$ of order $d$ in the vertex set of a given hypergraph $H$ into boundedly many parts, in such a way that the restriction of the hypergraph $H$ above almost all of those parts is quasirandom of order $d$.

In order to formally state Chung's theorem, let us introduce the following piece of notation.
Let $V$ be a vertex set and let $S_1, S_2, \dots, S_{\binom{k}{d}}$ be disjoint subsets of $\binom{V}{d}$.
The \emph{cell of $\binom{V}{k}$ induced by $S_1, S_2, \dots, S_{\binom{k}{d}}$} is defined by
$$\mathcal{C}^{(k)} \big(S_1, \dots, S_{\binom{k}{d}}\big) := \bigg\{ \mathbf{x} \in \binom{V}{k}:\, \bigg| \binom{\mathbf{x}}{d} \cap S_i \bigg| = 1 \text{ for all } 1 \leq i \leq \binom{k}{d} \bigg\}.$$
This can be seen as a $k$-uniform hypergraph which has a very strong cut structure of order $d$.
Given a $k$-uniform hypergraph $H$ on $V$ and disjoint sets $S_1, S_2, \dots, S_{\binom{k}{d}} \subset \binom{V}{d}$, define $H\big[S_{1}, \dots, S_{\binom{k}{d}}\big] := H \cap \mathcal{C}^{(k)} \big(S_1, \dots, S_{\binom{k}{d}}\big)$ as the part of $H$ which `sits above' the cell of $\binom{V}{k}$ induced by these sets.

As in the case of partite hypergraphs, we think of $H[S_{1}, \dots, S_{\binom{k}{d}}]$ as being quasirandom of order $d$ if it does not correlate with cut structure of order $d$ \emph{other than} the one induced by the cell $\mathcal{C}^{(k)}(S_1, \dots, S_{\binom{k}{d}})$;
that is, if $H$ `sits quasirandomly' above this cell.
One way of measuring this notion is given as follows.

Denote the density of $H$ above the cell $\mathcal{C}^{(k)} \big(S_1, \dots, S_{\binom{k}{d}}\big)$ by
$$\delta_H \big(S_{1}, \dots, S_{\binom{k}{d}}\big) := \frac{\big| H \big[S_{1}, \dots, S_{\binom{k}{d}}\big] \big|}{\big| \mathcal{C}^{(k)} \big(S_1, \dots, S_{\binom{k}{d}}\big) \big|}.$$
We say that $\big(S_1, \dots, S_{\binom{k}{d}}\big)$ is \emph{$\varepsilon$-regular of order $d$} for $H$ if
$$\big| \delta_H \big(T_1, \dots, T_{\binom{k}{d}}\big) - \delta_H \big(S_1, \dots, S_{\binom{k}{d}}\big) \big| \leq \varepsilon$$
whenever $T_i \subseteq S_i$ are subsets with
$$\big| \mathcal{C}^{(k)} \big(T_1, \dots, T_{\binom{k}{d}}\big) \big| \geq \varepsilon \big|\mathcal{C}^{(k)} \big(S_1, \dots, S_{\binom{k}{d}}\big) \big|.$$
Up to substituting $\varepsilon$ by some small power of $\varepsilon$, this is equivalent to requiring that
$$\|(H - \delta_H(\mathcal{S})) \,\mathcal{C}^{(k)}(\mathcal{S})\|_{\square^k_d} \leq \varepsilon \|\mathcal{C}^{(k)}(\mathcal{S})\|_{\square^k_d},$$
where $\mathcal{S} := \big(S_1, \dots, S_{\binom{k}{d}}\big)$,
which is more similar to how we have measured quasirandomness throughout this paper.
Note the presence of the term $\mathcal{C}^{(k)}(\mathcal{S})$ on both sides of the inequality, which is needed since this is a notion of quasirandomness \emph{relative} to that cell.

Chung's regularity lemma \cite{RegularityHypergraphsQuasirandomness} may then be stated as follows:

\begin{thm} \label{HyperRL}
For all integers $1 \leq d < k$ and every $\varepsilon > 0$ there exists an integer $M \geq 1$ for which the following holds.
For any $k$-uniform hypergraph $H$ on vertex set $V$, $\binom{V}{d}$ can be partitioned into sets $S_1, \dots, S_m$ for some $m \leq M$ so that all but at most $\varepsilon n^k$ edges of $H$ are contained in $H\big[S_{i_1}, \dots, S_{i_{\binom{k}{d}}}\big]$ for some $i_1, \dots, i_{\binom{k}{d}}$, where $1 \leq i_1 < \dots < i_{\binom{k}{d}} \leq m$ are distinct indices and $\big(S_{i_1}, \dots, S_{i_{\binom{k}{d}}}\big)$ is $\varepsilon$-regular of order~$d$.
\end{thm}

The proof of this theorem proceeds by an energy increment argument quite similar to that of the original regularity lemma, which is essentially the case $k=2$, $d=1$ of this last result.


For many applications, however, such a result is unsuitable as the partition $\binom{V}{d} = S_1 \cup \dots \cup S_m$ obtained may be very complex and have no `regularity properties' themselves;
in fact, it is not clear even how to estimate the size of the cells of $\binom{V}{k}$ induced by such a partition.
In order to remedy this issue, it is necessary to further regularize the classes $S_i$ of this partition (which may be seen as $d$-uniform hypergraphs on $V$) in terms of a partition of $\binom{V}{d-1}$, whose classes themselves should be regularized in terms of a partition of $\binom{V}{d-2}$ and so on.

The size of each partition should stay bounded independently of the size of $H$, and the strength of quasirandomness obtained for the cells of each partition should be sufficiently strong in order to make up for the small errors and the increase in size of the partitions at lower orders.
This can all be done, but it is much harder than in the case of graphs and the details and notation get somewhat complicated.

The first to obtain such a strong regularity lemma for hypergraphs were Gowers \cite{Regularity3graphs, HypergraphRegularityGowers} and, independently, R\"odl and Skokan \cite{RegularityRodlSkokan};
we refer the reader to the original papers for the precise statement and proof of their results.
Gowers' papers \cite{Regularity3graphs, HypergraphRegularityGowers} introduced the octahedral norms we saw in Section \ref{OctNormsSection} and obtained its main properties (most notably the Gowers-Cauchy-Schwarz inequality).
The analysis of the regular hypergraph partitions obtained by R\"odl and Skokan's regularity lemma requires a counting lemma developed by Nagle, R\"odl and Schacht \cite{HypergraphCountingLemma},
which is reminiscent of (but more complicated than) our Lemma \ref{count_d-linear} which bounds the contribution of the quasirandom component when counting subhypergraphs.

\subsection{Arithmetic regularity}

Let us now turn our attention towards arithmetic regularity lemmas, where the objects we wish to regularize are (bounded) functions $f$ defined on a given additive group $G$.

In order to get a feeling for such results, we start by considering a very simple regularity lemma which is valid for any additive group $G$.

\begin{lem} \label{EasyRegLem}
Let $f: G \rightarrow [-1, 1]$ be a bounded function and let $\varepsilon > 0$.
Then we can decompose $f = f_{\str} + f_{\qsr}$ into structured and quasirandom parts, with $f_{\str}$ being a linear combination of at most $1/\varepsilon^2$ characters with coefficients bounded in magnitude by $1$, and $f_{\qsr}$ being Fourier $\varepsilon$-uniform.
\end{lem}

\begin{proof}
Let $R \subseteq \widehat{G}$ be the set of characters $\chi$ for which $|\widehat{f}(\chi)| \geq \varepsilon$.
By hypothesis we have that $\|\widehat{f}\|_{\ell^2} = \|f\|_{L^2} \leq 1$, which easily implies that $|R| \leq 1/\varepsilon^2$.

Now let $f_{\str} := \sum_{\chi \in R} \widehat{f}(\chi) \chi$ and $f_{\qsr} := \sum_{\chi' \in \widehat{G} \setminus R} \widehat{f}(\chi') \chi'$.
Then $f_{\str}$ is a linear combination of at most $1/\varepsilon^2$ characters with coefficients bounded in magnitude by $\|\widehat{f}\|_{\ell^{\infty}} \leq \|\widehat{f}\|_{\ell^2} \leq 1$, $f_{\qsr}$ is Fourier $\varepsilon$-uniform and $f = f_{\str} + f_{\qsr}$ by the Fourier inversion formula.
\end{proof}

Suppose then we have a bounded function $f$ which we wish to analyze, and we use the above lemma to decompose it as $f_{\str} + f_{\qsr}$ (for some suitable $\varepsilon > 0$).
By the results of Section \ref{GroupsSection}, the quasirandom term $f_{\qsr}$ will give only a negligible contribution to the count of some linear patterns such as 3-term arithmetic progressions or additive quadruples.
Moreover, since characters are highly structured functions for which one can explicitly compute multilinear averages, one might expect that dealing with the structured term $f_{\str}$ will be simple given that it is just a bounded linear combination of characters.

There are, however, two main problems with this simple lemma.
The first is that the complexity of the structured term (which can depend on up to $1/\varepsilon^2$ distinct characters) is too high for the relatively weak control we get for the quasirandom term (which is only Fourier $\varepsilon$-uniform).
The second problem is that, while the function $f$ we started with was bounded in magnitude by 1, the structured part $f_{\str}$ might take values of magnitude $O(1/\varepsilon^2)$.
In other words, the bounds we have on the original function are not preserved (even approximately) when passing to the structured component, which causes many issues when estimating the count of linear configurations.

A more involved (and more useful) regularity lemma for Fourier uniformity which has neither of these issues was first obtained by Green in \cite{AbelianRegularity}, where some applications in additive combinatorics are also shown.
While the formal statement of this result for a general additive group $G$ is rather complicated, let us quickly describe it in the particular case where $G = \F_p^n$ is a vector space over a (small) prime field $\F_p$;
in this setting the result is easy to understand, and it is also more suitable for noticing the similarities with Szemer\'edi's regularity lemma.

Suppose then we are given a subset $A \subseteq \F_p^n$ and an accuracy parameter $\varepsilon > 0$.
Green's regularity lemma states that one can decompose $\F_p^n$ into cosets of a subspace $H \leq \F_p^n$ of bounded codimension, in such a way that the restriction of $A$ to
all but an $\varepsilon$-fraction of
these cosets $g + H$ is Fourier $\varepsilon$-uniform.\footnote{More precisely, their translates $(A - g) \cap H$ are Fourier $\varepsilon$-uniform when considered as subsets of the additive group $H$.}
As in the case of the graph regularity lemma, the codimension of this subspace $H$ is bounded by a function of $\varepsilon$ which is independent of the dimension $n$ of the space, but which has a quite bad dependence on $\varepsilon$:
building upon the methods of Gowers \cite{GowersLowerBound}, Green showed that for sufficiently large $n$ there are sets $A \subset \F_2^n$ for which the largest such subspace
has codimension \emph{at least} a tower of $2$s of height logarithmic in $1/\varepsilon$.
This lower bound was later improved by Hosseini, Lovett, Moshkovitz and Shapira \cite{AbelianRegularityBound} to an exponential tower of height about $1/\varepsilon$, thus of a similar type as that needed for graph regularity.

Stronger arithmetic regularity lemmas which deal with higher-degree uniformity have since then been obtained for functions on vector spaces $\F_p^n$ over (bounded) prime fields \cite{NewBoundsI, QuadraticUniformityFpn, HigherDegreeUniformityFpn, GeneralSystemsComplexity}, and for functions on cyclic groups $\Z_N$ \cite{NewBoundsII, QuadraticUniformityZn, ArithmeticRegularity}.
These results are rather deep, and their proofs rely heavily on the inverse theorem for the uniformity norms,
whose general form was obtained by Bergelson, Tao and Ziegler \cite{InverseGowersAction, InverseConjectureFiniteFields, InverseConjectureLowCharacteristic} for vector spaces $\F_p^n$
and by Green, Tao and Ziegler \cite{InverseTheoremGowers} for cyclic groups $\Z_N$.

In the case of finite vector spaces $\F_p^n$ the regularity lemma is in some ways quite similar to our simple Lemma \ref{EasyRegLem}.
Its main idea is that one can decompose an arbitrary function $f: \F_p^n \rightarrow [-1, 1]$ into a linear combination
$f_{\str}(x) = \sum_{j=1}^{C(\varepsilon)} \lambda_j e^{2\pi i P_j(x)}$
with boundedly many terms, with coefficients $|\lambda_j| \leq 1$ and each $P_j: \F_p^n \rightarrow \R/\Z$ being a
polynomial\footnote{When the degree $d$ is higher than the characteristic $p$ of the field, it is necessary to also allow for \emph{non-classical polynomials} as defined in \cite{InverseConjectureLowCharacteristic}.}
of degree at most $d$, plus a component $f_{\qsr}$ which is $\varepsilon$-uniform of degree $d$.
Moreover, the linear combination $f_{\str}$ can
be made to have same bounds as the original function $f$, and the polynomials $P_j$ can be required to have `high rank' so that
their linear combinations are all highly uniformly distributed on $\F_p^n$.

In many applications it is important to have a stronger control on the quasirandomness of $f_{\qsr}$ relative to the number $C$ of terms $\lambda_j e^{2\pi i P_j(x)}$ which form the structured component $f_{\str}$.
This can be achieved by allowing a third `error term' $f_{\err}$ into this decomposition, which is small in the sense that $\|f_{\err}\|_{L^2} \leq \varepsilon$.
Then, for any fixed function $\eta: \mathbb{N} \rightarrow (0, 1]$ representing the relative control on $f_{\qsr}$ over $f_{\str}$ we wish to have,
we can require that $\|f_{\qsr}\|_{U^{d+1}} \leq \eta(C)$;
this is done for instance in \cite{HigherDegreeUniformityFpn, GeneralSystemsComplexity}.
For much more information about arithmetic regularity lemmas on $\F_p^n$ and several applications in combinatorics and computer science, we refer the reader to Hatami, Hatami and Lovett's book \cite{HigherFourierCS}.

A similar decomposition result which holds in the more technically challenging case of functions on the cyclic group $\Z_N$
(and even on the discrete interval $[N] \subset \Z$)
was obtained by Green and Tao \cite{ArithmeticRegularity},
who derived from it several interesting theorems in additive combinatorics.
Here we will only be able to give a very high-level informal overview of this important result, as the details are somewhat complicated and it would take us too far afield to even properly define the notions needed.

The main difficulty in understanding Green and Tao's regularity lemma is to understand what notion of structure is captured by the structured term.
Contrary to the case of vector spaces $\F_p^n$, in cyclic groups $\Z_N$ a function having large correlation with phase polynomials $e^{2\pi i P(x)}$ no longer constitutes the only source of obstruction to having small uniformity norms;
one must consider also a much greater class of functions called \emph{nilsequences}.

Nilsequences are generalizations of almost periodic sequences
first introduced by Bergelson, Host and Kra \cite{Nilsequences} in the context of studying multiple recurrence in ergodic theory.
Their formal definition will not be recalled here, but we note that
they are related to the dynamics of orbits on objects known as \emph{nilmanifolds},\footnote{A $k$-step nilmanifold is a compact symmetric space $G/\Gamma$, where $G$ is a $k$-step nilpotent Lie group (i.e. all $(k+1)$-fold commutators of $G$ are trivial) and $\Gamma$ is a discrete subgroup.}
and have a lot of structure which permits them to be analyzed.
The class of $k$-step nilsequences contains all polynomial phases $n \mapsto e^{2\pi i P(n)}$ with $P$ being a polynomial of degree at most $k$, and they \emph{characterize} functions with non-negligible $U^{k+1}(\Z_N)$ norm:
if $N$ is prime, then a bounded function $f: \Z_N \rightarrow [-1, 1]$ has non-negligible $U^{k+1}$ norm if and only if it correlates with a $k$-step nilsequence of bounded complexity
(see \cite{InverseTheoremGowers} for a precise statement).

Green and Tao's regularity lemma for uniformity of degree $d \geq 1$ then permits one to decomposes an arbitrary function $f: [N] \rightarrow [-1, 1]$ into a sum of three terms $f_{\str} + f_{\qsr} + f_{\err}$.
The error term $f_{\err}$ is small in the sense that $\|f_{\err}\|_{L^2} \leq \varepsilon$ (for some previously chosen quantity $\varepsilon > 0$), and one checks that its contribution to multilinear averages involving the function $f$ is negligible if the parameter $\varepsilon$ is small enough.

The quasirandom term $f_{\qsr}$ is extremely uniform of degree $d$, in the following sense.
Since the discrete interval $[N]$ is not a group,
we must first embed $f_{\qsr}$ into a cyclic group $G = \Z_{\tilde{N}}$ for some integer $\Tilde{N} \geq 2^{d+1} N$
(this restriction is made to prevent `wrapping around' issues):
define $\Tilde{f}_{\qsr}: G \rightarrow \R$ by
$\Tilde{f}_{\qsr}(x) = f_{\qsr}(x)$ for $x = 1, \dots, N$ and $\Tilde{f}_{\qsr}(x) = 0$ otherwise.
We then set
$\|f_{\qsr}\|_{U^{d+1}[N]} := \|\Tilde{f}_{\qsr}\|_{U^{d+1}(G)}/\|\mathbbm{1}_{[N]}\|_{U^{d+1}(G)}$,
where $\mathbbm{1}_{[N]}$ is the indicator function of $[N]$ in $G$;
this definition is easily checked to be independent of the choice of $\Tilde{N}$.
Our quasirandomness condition on $f_{\qsr}$ is that its norm $\|f_{\qsr}\|_{U^{d+1}[N]}$ is smaller than any (previously defined) quantity depending on the parameter $\varepsilon$ and on the complexity of the structured term $f_{\str}$;
its contribution can then be easily dealt with using some version of our counting lemma from Section \ref{LinearConfigSection}.

Finally, the structured term $f_{\str}$ is (a more general variant of) a $d$-step nilsequence having bounded complexity,
which can be analyzed through the quantitative equidistribution theory of nilmanifolds also developed by Green and Tao on an earlier paper \cite{PolynomialOrbitsNilmanifolds}.
To complement the regularity lemma, they also provide in \cite{ArithmeticRegularity, ArithmeticRegularityArxiv} an arithmetic counting lemma which gives an integral formula for counting linear configurations weighted by such generalized nilsequences.

\section*{Acknowledgements}

This paper grew out of a mini-course given by the author at the University of Cologne, and he would like to thank his advisor Frank Vallentin for the opportunity of giving this mini-course, and for helpful comments.
The author is also indebted to Victor Souza for corrections and many helpful suggestions on an earlier version of this paper.

This work is supported by the European Union’s EU Framework Programme for Research and Innovation Horizon 2020 under the Marie Sk\l{}odowska-Curie Actions Grant Agreement No 764759 (MINOA).

\appendix

\section{Basic probabilistic notions and results} \label{App}

In this appendix we provide the definitions and results in finite probability theory which are most useful for our purposes.

Let $(\Omega, 2^{\Omega}, \mathbb{P})$ be a finite probability space;
thus $\Omega$ is a finite set and $\mathbb{P}: \Omega \rightarrow [0, 1]$ is a nonnegative function satisfying $\sum_{\omega \in \Omega} \mathbb{P}(\omega) = 1$.
A \emph{random event} is simply a subset $E \subseteq \Omega$, and its probability is denoted $\mathbb{P}(E) := \sum_{\omega \in E} \mathbb{P}(\omega)$.
Two events $E_1, E_2$ are \emph{independent} if
$\mathbb{P}(E_1 \cap E_2) = \mathbb{P}(E_1) \mathbb{P}(E_2)$.
A family of events $E_1, \dots, E_n$ is \emph{jointly independent} if
$\mathbb{P}\left( \bigcap_{j=1}^k E_{i_j} \right) = \prod_{j=1}^k \mathbb{P}(E_{i_j})$
for all $k \leq n$ and all $1 \leq i_1 < \dots < i_k \leq n$.

A simple but very useful inequality is the \emph{union bound}:
for every collection of random events $E_1, \dots, E_n$ we have
$\mathbb{P}\left( \bigcup_{i=1}^n E_i \right) \leq \sum_{i=1}^n \mathbb{P}(E_i)$;
the proof is immediate from the definitions.

A (real-valued) \emph{random variable} is a function $X: \Omega \rightarrow \R$;
its \emph{mean} or \emph{expectation} is given by
$$\mathbb{E}[X] := \sum_{\omega \in \Omega} X(\omega) \mathbb{P}(\omega) = \sum_{x} x \,\mathbb{P}(X = x).$$
From this formula it is clear that $\min X \leq \mathbb{E}[X] \leq \max X$.
Another immediate consequence is the \emph{linearity of expectation}:
for all random variables $X_1, \dots, X_n$ and all constants $c_1, \dots, c_n$ we have
$$\mathbb{E}[c_1 X_1 + \dots + c_n X_n] = c_1 \mathbb{E}[X_1] + \dots + c_n \mathbb{E}[X_n].$$
Given an event $E$ we denote by $\mathbf{1}_E$
its \emph{indicator random variable}, which equals $1$ if $E$ occurs and $0$ otherwise;
note that $\mathbb{E}\left[\mathbf{1}_{E}\right] = \mathbb{P}(E)$.

We define the \emph{variance} of a random variable $X$ by
$$\var(X) := \mathbb{E}\big[|X - \mathbb{E}[X]|^2\big] = \mathbb{E}[X^2] - \mathbb{E}[X]^2,$$
where this last equality follows easily from linearity of expectation.
If the random variables $X_1, \dots, X_n$ are pairwise independent then
$\var(X_1 + \dots + X_n) = \var(X_1) + \dots + \var(X_n)$,
an identity which is \emph{not} valid in general.

Many times we will have to deal with random events $E_n$ which depend on some asymptotic parameter $n \in \mathbb{N}$, for instance when considering random graphs on $n$ vertices.
In such cases, we say that the event $E_n$ holds \emph{with high probability} (sometimes written `w.h.p.') if the probability that it holds tends to $1$ as the parameter gets large.

If $X$ is a nonnegative random variable, then for all $\lambda > 0$ we have $X \geq \lambda \mathbf{1}_{\{X \geq \lambda\}}$.
Taking expectation on both sides and dividing by $\lambda$ we obtain \emph{Markov's inequality}:
$\mathbb{P}(X \geq \lambda) \leq \mathbb{E}[X]/\lambda$.
Applying this inequality to the nonnegative random variable $|X - \mathbb{E}[X]|^2$ we obtain:

\begin{lem}[Chebyshev's inequality]
Let $X$ be a random variable of mean $\mu = \mathbb{E}[X]$ and variance $\var(X)$.
For any $\lambda > 0$ we have
$$\mathbb{P}\left( |X - \mu| \geq \lambda \right) \leq \frac{\var(X)}{\lambda^2}.$$
\end{lem}

Recall that a function $f: [a, b] \rightarrow \R$ is \emph{convex} if
$$f(tx + (1-t)y) \leq t f(x) + (1-t) f(y)$$
for all $x, y \in [a, b]$ and all $0 \leq t \leq 1$.
By a simple induction argument (which will not be given here) one obtains the fundamental \emph{Jensen's inequality}:
if $f$ is a convex function on $[a, b]$ and $X$ is a random variable taking values on this interval, then
$f\left(\mathbb{E}[X]\right) \leq \mathbb{E}\left[f(X)\right]$.
This basic fact is sometimes referred to by writing only `by convexity'.

We will also have cause to use the following more advanced (but standard) result, which is an instance of a class of large deviation inequalities usually known as `Chernoff bounds'.
The result in this form was taken from Tao and Vu's book \cite{TaoVu}, and we refer the reader to this book for its proof.

\begin{lem}[Chernoff's inequality] \label{Chernoff}
Suppose $X_1, X_2, \dots, X_n$ are jointly independent real random variables satisfying $|X_i - \mathbb{E}[X_i]| \leq 1$ for all $i \in [n]$.
Set $X := X_1 + \dots + X_n$ and let $\sigma := \sqrt{\var(X)}$ be the standard deviation of $X$.
Then for any $\lambda > 0$ we have
$$\mathbb{P}\big(| X - \mathbb{E}[X]| \geq \lambda \sigma \big) \leq 2 \max \big\{e^{-\lambda^2/4},\, e^{-\lambda \sigma/2}\big\}.$$
\end{lem}

\bibliography{main}
\bibliographystyle{siam}

\Addresses

\end{document}